\documentclass[a4paper,reqno,11pt]{amsart}
\usepackage[T1]{fontenc}
\usepackage[utf8]{inputenc}
\usepackage[english]{babel}
\usepackage{xspace}
\usepackage{amsfonts}
\usepackage{amsmath}
\usepackage{amssymb}
\usepackage{amsthm}
\usepackage{graphicx}
\usepackage{bbm}
\usepackage{tikz}
\usepackage{subcaption}
\usepackage{xfrac}
\usepackage{dsfont}
\usepackage{amsthm}
\usepackage{mathrsfs}
\usepackage{enumitem}
\usepackage{verbatim}
\usepackage{dsfont}
\usepackage{xcolor}
\usepackage{soul}

\usepackage{mathtools}
\mathtoolsset{showonlyrefs}

\usepackage[autostyle]{csquotes}  
\usepackage[dvipsnames]{xcolor}
\usepackage{hyperref}
\hypersetup{colorlinks = true}
\hypersetup{
    citecolor=red,
    linkcolor=blue}
    
\usepackage{geometry}
\theoremstyle{plain}
\newtheorem{theorem}{Theorem}[section]
\newtheorem{proposition}[theorem]{Proposition}
\newtheorem{lemma}[theorem]{Lemma}
\newtheorem{corol}[theorem]{Corollary}

\theoremstyle{remark}
\newtheorem{remark}[theorem]{Remark}

\theoremstyle{definition}
\newtheorem{definition}[theorem]{Definition}
\newtheorem{assumption}[theorem]{Assumption}

\numberwithin{equation}{section}

\DeclareFontFamily{U}{mathx}{}
\DeclareFontShape{U}{mathx}{m}{n}{<-> mathx10}{}
\DeclareSymbolFont{mathx}{U}{mathx}{m}{n}
\DeclareMathAccent{\widehat}{0}{mathx}{"70}
\DeclareMathAccent{\widecheck}{0}{mathx}{"71}

\DeclareMathOperator{\Span}{span}
\makeatletter

\newcommand{\numberset}{\mathbb}
\newcommand{\N}{\numberset{N}}
\newcommand{\R}{\numberset{R}}
\newcommand{\C}{\numberset{C}}
\newcommand{\Z}{\numberset{Z}}

\newcommand{\D}{\mathcal{D}}

\newcommand{\FF}{\mathcal{F}}
\newcommand{\LL}{\mathcal{L}}
\newcommand{\RR}{\mathcal R}

\newcommand{\ee}{\mathrm e}
\newcommand{\oo}{\mathrm o}

\title[Dirac solitons in 1D NLS]{Dirac solitons in one--dimensional\\ nonlinear Schr\"odinger equations}

\author[W. Borrelli]{William Borrelli$^*$}
\address{W. Borrelli: Politecnico di Milano, Dipartimento di Matematica, P.zza Leonardo da Vinci, 32, 20133 Milano, Italy.}
\email{william.borrelli@polimi.it}

\author[E. Danesi]{Elena Danesi}
\address{E. Danesi: Politecnico di Torino, Dipartimento di Scienze Matematiche ``G.L. Lagrange'', Corso Duca degli Abruzzi, 24, 10129, Torino, Italy}
\email{elena.danesi@polito.it}

\author[S. Dovetta]{Simone Dovetta}
\address{S. Dovetta: Politecnico di Torino, Dipartimento di Scienze Matematiche ``G.L. Lagrange'', Corso Duca degli Abruzzi, 24, 10129, Torino, Italy}
\email{simone.dovetta@polito.it}

\author[L. Tentarelli]{Lorenzo Tentarelli}
\address{L. Tentarelli: Politecnico di Torino, Dipartimento di Scienze Matematiche ``G.L. Lagrange'', Corso Duca degli Abruzzi, 24, 10129, Torino, Italy}
\email{lorenzo.tentarelli@polito.it}

\begin{document}

\begin{abstract}
In this paper we study a family of one--dimensional stationary cubic nonlinear Schr\"odinger (NLS) equations with periodic potentials and linear part displaying \emph{Dirac points} in the dispersion relation. By introducing a suitable periodic perturbation, one can open a spectral gap around the Dirac--point energy. This allows to construct standing waves of the NLS equation whose leading--order profile is a modulation of Bloch waves by means of the components of a spinor solving an appropriate cubic nonlinear Dirac (NLD) equation. We refer to these solutions as \emph{Dirac solitons}. Our analysis thus provides a rigorous justification for the use of the NLD equation as an effective model for the original NLS equation.
\end{abstract}

\thanks{$^*$\emph{Corresponding author}. Dipartimento di Matematica, Politecnico di Milano, Via Bonardi, 9, I-20133, Milano, Italy. E-mail: william.borrelli@polimi.it}
\maketitle

\noindent {\footnotesize {AMS Subject Classification:} 35Q55, 35Q51, 35Q40, 81Q05, 35B32, 35B20}

\smallskip
\noindent {\footnotesize {Keywords:} Dirac solitons, Dirac points, NLS equations, NLD equations, standing waves, periodic Schr\"odinger operators}

\section{Introduction and main results}
\label{sec:intro}

The aim of the paper is to investigate the validity of one--dimensional nonlinear Dirac (NLD) equations as effective models for a class of one--dimensional cubic focusing nonlinear Schr\"odinger (NLS) equations with periodic potentials.

It is well-known that the spectrum of a periodic Schr\"odinger operator
\begin{equation}
\label{eq:H}
\mathcal{H}:=-\partial^2_x+V(\cdot),\qquad\text{with}\quad V\:\text{smooth, real--valued and 1--periodic},
\end{equation}
is purely absolutely continuous  and displays a band structure (see, e.g., \cite[Chapter XIII]{RS81}). Heuristically, this suggests the possibility to look for solutions of equations
\begin{equation}
\label{eq-generalissima}
(\mathcal{H}-\mu)\,u=f(x,u)
\end{equation}
in the form
\[
u(x)\sim \Psi(\varepsilon,x)\,\Phi(x)
\]
where $\mu$ lies in a spectral gap of $\mathcal{H}$ close to a band edge $\mu_*$, $\Phi$ is a generalized eigenfunction (usually called \emph{Bloch wave}) associated with $\mu_*$, $\varepsilon\to0^+$ as $\mu\to\mu_*$. The function $\Psi$ is a suitable modulating coefficient to be found, which is expected to vanish as $\mu$ approaches $\mu_*$. In view of this, the equation satisfied by $\Psi$ can be seen as an effective model for \eqref{eq-generalissima} and depends on the local geometry of the bands involved.

In general, this is a mere intuition and an actual realization of this programme requires a rigorous justification. Clearly, it is unreasonable to expect that such a construction may work for any choice of the function $f$. In view of its well--established relevance in the modeling of a variety of physical systems, ranging from nonlinear optics to solild state physics, a major attention has been devoted thus far to focusing power--like terms as
\begin{equation}
\label{eq-fpower}
f(x,u)=|u|^{p-1}\,u,
\end{equation}
with a special focus on the cubic case $p=3$. Several works have been developed in this direction, addressing various aspects of the NLS with periodic potentials and the bifurcation of localized nonlinear modes close to band edges of spectral gaps. For a general overview in this context we refer for instance to the monograph  \cite{Pely11} with the references therein, and to, e.g., \cite{DP19,DPR11,DW20,IW10} for more recent results.

\medskip
In the present paper we are interested in equations of the form \eqref{eq-generalissima} with a cubic nonlinearity \eqref{eq-fpower} and a periodic Schr\"odinger operator \eqref{eq:H} that admits conical points in the dispersion relation in the quasi--momentum/energy plane. These points are known as \emph{Dirac points} and arise at an energy $\mu_*$ where two spectral bands touch each other, corresponding to two linearly independent Bloch waves. The existence of this type of points has been established by \cite{FLW17} and \cite{FW12hd} in one and two spatial dimensions, respectively. To date, this model seems to have been discussed in the two--dimensional time--dependent version
\begin{equation}
\label{eq:nls}
i\partial_t v =\mathcal{H}v  - \lvert  v\rvert^2 v.
\end{equation}
In this setting, an effective nonlinear Dirac equation has been formally derived in \cite{FW12w}, and then rigorously justified in \cite{AS18}. More precisely, letting $\Phi_\pm(x,k_*)$ be the two linearly independent Bloch waves at $\mu_*$, indexed with the associated quasi--momentum $k_*$, \cite{AS18} establishes that, starting with an initial datum 
\begin{equation}\label{concentrated}
v^{\varepsilon}_{0}(x)\underset{\epsilon\rightarrow0^{+}}{\sim}\sqrt{\varepsilon}\,\big(\psi_{0,-}(\varepsilon x)\Phi_{-}(x,k_*)+\psi_{0,+}(\varepsilon x)\Phi_{+}(x,k_*)\big),
\end{equation}
the corresponding solution of \eqref{eq:nls} evolves as 
\begin{equation}\label{eq:approxdyn}
v^{\varepsilon}(t,x)\underset{\epsilon\rightarrow0^{+}}{\sim}e^{-i\mu_* t}\sqrt{\varepsilon}\big(\psi_{-}(\varepsilon t,\varepsilon x)\Phi_{-}(x,k_*)+\psi_{+}(\varepsilon t,\varepsilon x)\Phi_{+}(x,k_*)\big),
\end{equation}
with $\psi=(\psi_-,\psi_+)^T$ solving the time--dependent nonlinear Dirac equation obtained in \cite{FW12w} with initial datum $(\psi_{0,-},\psi_{0,+})^T$. This has been proved on a finite timescale $[0,T_\varepsilon)$, with $\displaystyle\lim_{\varepsilon\to0^+}T_\varepsilon=+\infty$. The finiteness of such timescale reflects the fact that the approximation deteriorates after a certain time, given the dispersive nature of the dynamics. This result is the nonlinear counterpart of that proved in the seminal paper \cite{FW14w} in the linear case.

\medskip
In this paper, we address the analogous question for \emph{standing waves} in one space dimension, that is, solutions to \eqref{eq:nls} of the form
\begin{equation}
	\label{eq:vu}
v(t,x)=e^{-i\mu t}u(x),\qquad (t,x)\in\R\times\R,
\end{equation}
where $\mu \in \R$ and $u$ is the corresponding spatial profile to be determined as a non--trivial localized solution of
\begin{equation}
\label{eq-stationarywrong}
(\mathcal{H}-\mu)\,u=|u|^2u.
\end{equation}
In view of the time--dependent analogue mentioned before, it is natural to search for solutions of \eqref{eq-stationarywrong} with frequency $\mu$ close to the energy $\mu_*$ of a Dirac point of $\mathcal{H}$. However, no such solution exists. Indeed, assume by contradiction that $u$ is a non--trivial $L^2(\R)$--solution decaying to zero at infinity. Since $\mu_*$ is an interior point of the continuous spectrum of $\mathcal{H}$, the same is true for $\mu$. Then, letting $\mathcal V(x):=-\vert u(x)\vert^2$, it is not difficult to check that the difference between the resolvents of $\mathcal{H}$ and $\mathcal{H}+\mathcal V$ is compact on $L^2(\R)$, so that the essential spectra of the two operators coincide by the Weyl Theorem (see, e.g., \cite[Thm. XIII.14]{RS81}). Hence, $\mu$ is also in the essential spectrum of $\mathcal{H}+\mathcal V$, but cannot be an embedded eigenvalue as it is not an endpoint of a spectral band of $\mathcal{H}+\mathcal V$ and $\int_{\R}\vert \mathcal V\vert\,dx=\int_{\R} \vert u\vert^2\,dx<+\infty$ (see, e.g., \cite[Section 5]{KV00} and references therein). However, by \eqref{eq-stationarywrong}, $\mu$ has to be an eigenvalue of $\mathcal{H}+\mathcal{V}$ too, which is a contradiction.

Note that the very same problem occurs also in the linear case considered in \cite{FLW17}, where one assumes $f\equiv0$. Remarkably, that paper shows that this difficulty can be overcome by perturbing the operator $\mathcal{H}$ with a term of the form $
\delta\kappa(\delta x)W(x)$, with $0<\delta\ll1$, $W$ 1--periodic, and $\kappa$ a smooth connection between two asymptotic constants $\pm\kappa_\infty$. In this way, a spectral gap around $\mu_*$ can be opened, allowing to find eigenfunctions of the resulting Hamiltonian constructed by means of zero--energy modes of a suitable Dirac operator.

Building on \cite{FLW17}, we address \eqref{eq-stationarywrong} adding a perturbation $\delta W(x)$ to $\mathcal{H}$, with $W$ a suitable real--valued smooth and $1$--periodic potential. Precisely, we take $W$ so that there exists $\vartheta_\sharp\in\R\setminus\{0\}$ for which, given an arbitrary constant $a\in(0,1)$, there exists $\delta(a)>0$ such that the operator
\begin{equation}
\label{eq-opwithgap}
\mathcal{H}_\delta:=\mathcal{H}+\delta W(\cdot)=-\partial_x^2+V(\cdot)+\delta W(\cdot)\qquad\text{on}\quad L^2(\R),
\end{equation}
admits a spectral gap containing the interval
\begin{equation}\label{eq:intgap}
\mathcal I_\delta:= \big( \mu_* - a\delta \lvert \vartheta_\sharp \rvert, \mu_* +a \delta \lvert \vartheta_\sharp \rvert \big),
\end{equation}
for $\delta\in (0,\delta(a))$ (see Figure \ref{fig:gap-opening}). 

	\begin{figure}[h!]
		\centering
		\begin{subfigure}[t]{0.48\columnwidth}
			\centering
			\includegraphics[width=\linewidth]{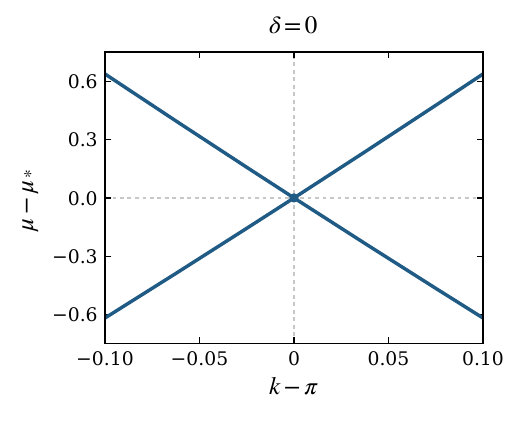}
			\caption{}
			\label{fig:bands-unperturbed}
		\end{subfigure}
		\hfill
		\begin{subfigure}[t]{0.48\columnwidth}
			\centering
			\includegraphics[width=\linewidth]{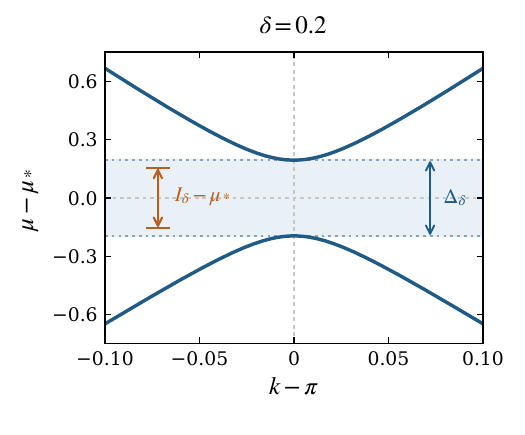}
			\caption{}
			\label{fig:bands-perturbed}
		\end{subfigure}
		\caption{Numerically computed dispersion bands (\textsc{a}) near the Dirac
			point $(\pi,\mu_*)$ for $V(x)=2\cos(4\pi x)$, and a corresponding gap opening (\textsc{b}) with
			$W(x)=2\cos(2\pi x)$. The bands were computed by a Fourier--Galerkin discretization
			of the Bloch eigenvalue problem using 41 modes.}
		\label{fig:gap-opening}
\end{figure}

Then, we can search for $(\mu_\delta,u_\delta)\in\mathcal{I}_\delta\times L^2(\R^2)$ satisfying
\begin{equation}
\label{NLS}
\left(\mathcal{H}_\delta-\mu_\delta \right)u_\delta=\vert u_\delta\vert^2 u_\delta,\qquad x\in\R,
\end{equation}
for instance taking
\begin{equation}
\label{eq:muas}
\mu_\delta=\mu_*+\mu_\sharp\,\delta,
\end{equation}
for some $\mu_\sharp$ to be determined (which can be seen as a first order approximation of a general frequency close to $\mu_*$). More precisely, we will establish the existence of a solution to this problem, which we call \emph{Dirac soliton}, of the following form
  \begin{equation}
  \label{eq:uform}
  u_\delta(x)=\sqrt{\delta}\big(\Psi_-(\delta x)\Phi_-(x,k_*)+\Psi_+(\delta x)\Phi_+(x,k_*)+r_\delta(x)\big),
  \end{equation}
  where $r_\delta(x)$ is a lower order correction term as $\delta\to0^+$, and
  \[
  \Psi(y):=\big(\Psi_-(y),\Psi_+(y)\big)^T,\qquad y\in\R,
  \]
is a spinor solving the nonlinear Dirac equation
  \medskip
  \begin{equation}
    \label{NLD}
  \left( i c_\sharp \sigma_3 \partial_y + \vartheta_\sharp \sigma_1 -  \mu_\sharp\right) \Psi = \mathcal G_{\beta_1, \beta_2} (\Psi).
    \end{equation}
   Here $c_\sharp$ is a real non--vanishing coefficient depending on the Bloch waves $\Phi_\pm(\cdot,k_*)$, $\sigma_1$ and $\sigma_3$ are the usual Pauli matrices
    \[
    \sigma_1:=\begin{pmatrix} 0 & 1 \\ 1 & 0 \end{pmatrix}\qquad\text{and}\qquad\sigma_3:=\begin{pmatrix} 1 & 0 \\ 0 & -1 \end{pmatrix}\,,
    \]
    and $\mathcal G_{\beta_1, \beta_2}$ is the following cubic nonlinearity 
    \begin{equation}
    \label{eq-nonlinearity}
    \mathcal G_{\beta_1, \beta_2} (\Psi): =  \begin{pmatrix}
        \beta_1 \lvert \Psi_- \rvert^2 + 2 \beta_1 \lvert \Psi_+ \rvert^2 & \beta_2 \ \overline{\Psi_-}\Psi_+ \\
        \beta_2 \ \Psi_-\overline{\Psi_+} & \beta_1 \lvert \Psi_+ \rvert^2 + 2 \beta_1 \lvert \Psi_- \rvert^2
    \end{pmatrix} \Psi\,,
    \end{equation}
$\beta_1,\,\beta_2\in\R$ being coefficients depending solely on the potential $V$. 

We highlight that we prove this result not only for a particular value of $\mu_\sharp$, but for a \emph{whole} interval of values. Indeed, if $\mu_\delta$ is as in \eqref{eq:muas} and belongs to $\mathcal{I}_\delta$ given by \eqref{eq:intgap}, then $\mu_\sharp\in(-\vert\vartheta_\sharp\vert,\vert\vartheta_\sharp\vert)$; that is, $\mu_\sharp$ is in the spectral gap of the one--dimensional Dirac operator 
\begin{equation}
\label{eq-diracop}
\mathcal{D}:=ic_\sharp \sigma_3 \partial_y + \vartheta_\sharp \sigma_1
\end{equation}
present in the linear part of \eqref{NLD}. This somehow explains the existence of Dirac solitons, since (up to a proper scaling) they can be roughly seen as bifurcating from bound states of \eqref{NLD}.

Our result belongs to the rigorous theory of coupled--mode approximations for gap solitons.

 A construction closely related to ours was carried out in \cite{PS07} for a cubic nonlinear Schr\"odinger equation with $V\equiv 0$, where it is proved that solutions of a stationary coupled--mode system give rise to gap solitons through a Lyapunov–Schmidt reduction in Fourier space. In the present work, the gap is opened at a Dirac point of an operator with a nontrivial periodic potential $V$. This changes both the effective equation and the analysis required to justify it. From the point of view of the effective equation, a more general Dirac system arises and existence of suitable solutions has to be derived. As for the Lyapunov--Schmidt scheme, a careful construction based on Bloch transform decomposition is needed to handle the different contributions coming from modes spectrally concentrated near and far the considered Dirac point.
 
In \cite{DPS08}, gap opening is studied for a twodimensional separable potential. In that case the authors work at quadratic band extrema corresponding to three modes at three different quasi--momenta, leading to a NLS effective system for the modulation amplitudes. It is then proved that, if this system admits a suitably non--degenerate solution, then the corresponding gap soliton can be constructed. Our paper differs from the previous one for two main reasons. First, we deal with a linear crossing of spectral bands, yielding an effective nonlinear Dirac equation. Second, our proof does not rely on any conditional non--degeneracy assumption, since we construct explicitly the solution of the effective system needed to build the gap soliton.

To the best of our knowledge, our result provides the first rigorous construction of this type for a nontrivial periodic potential admitting a Dirac point, providing a rigorous construction of standing waves for periodic NLS equations originating from solutions of effective stationary NLD equations in one space dimension. The possibility of such construction was a natural question in view of the results available for the linear analogue discussed in \cite{FLW17} (and later significantly improved in \cite{DFW20}). However, this type of relation between NLS and NLD seemed to be quite unexplored so far, despite a rather wide literature addressing existence and qualitative properties for solutions of the latter (see, e.g., \cite{BD06,B17,B18,B21,B22,BF19,CPS16,DR08,DW08,DW21,ES95,PS14,PS12}).


\subsection{Setting and main results}
\label{sec:main}

Now, we discuss in detail our results. Preliminarily, we fix the following standard notation (specific notation concerning periodic Schr\"odinger operators is introduced in Section \ref{sec:fb}):
\begin{enumerate}[label=(\roman*)]
\item we denote by $\N_+:=\N\setminus\{0\}$;
\item all functions have to be meant as complex--valued unless stated otherwise;
\item given a measurable set $X\subseteq\R$ we denote
\[
\left\langle f,g\right\rangle_{L^2(X)}:=\int_Xf\,\overline{g}\,dx\qquad\text{and}\qquad\|f\|_{L^2(X)}^2:=\left\langle f,f\right\rangle_{L^2(X)};
\]
\item for $s>0$, we set 
\[
L^{2,s}(\R):=\left\{f:\R\to\C:(1+|\cdot|^2)^{s/2}f(\cdot)\in L^2(\R)\right\},
\]
endowed with inner product and norms defined in the natural way;
\item Sobolev spaces $H^s(X)$, $s>0$, are defined in the usual way too;
\item $L^2(\R)$--Fourier transform and its inverse are understood as usual extensions by density from $L^1\cap L^2$ to $L^2$ functions of the following integrals
\[
\mathcal{F}(f)(\xi)=\widehat{f}(\xi):=\frac{1}{2\pi}\int_\R e^{-ix\xi}f(x)\,dx,\qquad\mathcal{F}^{-1}(g)(x)=\widecheck{g}(\xi):=\int_\R e^{ix\xi}g(\xi)\,d\xi\,,
\]
\end{enumerate}

\begin{remark}
Note that in \cite{FLW17} the inner product of $L^2(X)$ is defined with a conjugate on the first function. As a consequence, all the results that we refer to throughout the paper have to be meant as accordingly adapted.
\end{remark}

Before stating the main results of the paper, we introduce the assumptions on the potentials $V,\,W$ present in \eqref{NLS}. 

\begin{assumption}[Assumptions on $V$]
\label{itm:1}
$V \in C^\infty (\R)$, is real--valued and there exists a real sequence $(V_m)_{m\in2\N_+}\in\ell^2$ such that
\begin{equation}
\label{VF}
V(x) = \sum_{m \in 2\N_+} V_m \cos(2\pi mx),\qquad\forall x\in\R.
\end{equation}
Moreover, $V$ is such that the operator $\mathcal{H}$ introduced by \eqref{eq:H} admits a Dirac point, in the sense explained in Section \ref{subsec-Diracpoints}, that we denote by $(k_*,\mu_*)$.
\end{assumption}

Note that Assumption \ref{itm:1} implies that $V$ is even and that
	\begin{equation}\label{eq:Vminper}
	V(x+1/2)=V(x),\qquad \forall x\in\R\,,
	\end{equation}
	and we take
\begin{equation}
\label{eq-quasidiracpoint}
k_*=\pi\,.
\end{equation}
Note that, a priori, Dirac points can only occur at $k_*=0, \pi\in[0,2\pi)$, by \cite[Thm. XIII.89, item $e)$]{RS81}. However, as shown in \cite[Remark 3.4 \& Appendix E]{FLW17}, they generically do not appear at $k=0$, even for small $V$. 

Furthermore, denoting by $\Phi_{\pm}(\cdot,k)$, $k\in[0,2\pi]$, the two Bloch functions given by Proposition \ref{prop-Blochsym} below, there holds
\begin{equation}\label{eq:c}
 c_\sharp:=2i \langle \partial_x\Phi_-(\cdot,\pi),\Phi_-(\cdot,\pi)\rangle_{L^2([0,1])}=-2i \langle \partial_x\Phi_+(\cdot,\pi),\Phi_+(\cdot,\pi)\rangle_{L^2([0,1])}\in\R\setminus\{0\}
\end{equation}
(see Remark \ref{rem-csciarp} below). Finally, such $V$ makes $\mathcal{H}$ self-adjoint on $L^2(\R)$ with domain $H^2(\R)$.

\begin{remark}
\label{rem-diracexistence}
The existence of potentials satisfying Assumption \ref{itm:1} is guaranteed by the results established in \cite[Section 3 $\&$ Appendices C\&D]{FLW17}.
\end{remark}

\begin{assumption}[Assumptions on $W$]
\label{itm:3}
$W \in C^\infty (\R)$, is real--valued and there exists a real sequence $(W_m)_{m\in2\N+1}\in\ell^2$ such that
\begin{equation}
\label{WF}
W(x) = \sum_{m \in 2\N+1} W_m \cos(2\pi mx),\qquad\forall x\in\R.
\end{equation}
Moreover, $W$ satisfies
\begin{equation}
\label{eq:theta}
    \vartheta_\sharp := \langle W \Phi_+(\cdot,\pi),\Phi_-(\cdot,\pi)\rangle_{L^2([0,1])}\in\R\setminus\{0\},
    \end{equation}
    with $\Phi_\pm$ again given by Proposition \ref{prop-Blochsym}.
\end{assumption}

As a consequence of Assumption \ref{itm:3}, $W$ is even,
\[
   W(x+1)=W(x),\qquad\forall x\in\R,
\]
and
\begin{equation}
\label{eq:Wap}
   W(x+1/2)=-W(x)\,,\qquad\forall x \in\R\,.
\end{equation}
Moreover, such $W$ entails that $\mathcal{H}_\delta$ is self-adjoint on $L^2(\R)$ with domain $H^2(\R)$.

\begin{remark}
Assumption \ref{itm:3} contains the requirements on the perturbation $\delta W$ that ensure the existence of an interval of the form \eqref{eq:intgap} in the spectral gap of the operator $\mathcal{H}_\delta$ introduced in \eqref{eq-opwithgap} (see \cite[Eq. (1.14) $\&$ Proposition 4.1]{FLW17}). The existence of infinitely many such potentials, with a particular reference to condition \eqref{eq:theta}, is discussed for instance in \cite[Remark 4.1 $\&$ Remark 5.1]{FLW17} for $V\equiv 0$.
\end{remark}

We are now in position to state our main theorems. First, we establish existence of solutions for the effective equation \eqref{NLD}, which will be the building blocks to construct Dirac solitons for \eqref{NLS}.

\begin{theorem}[Solutions to \eqref{NLD}]\label{thm:main2}
Let $c_\sharp,\,\vartheta_\sharp,\,\beta_1\in\R\setminus\{0\}$,\,$\beta_2\in\R$, with $\beta_1\geq|\beta_2|$, and let $\mu_\sharp\in(-|\vartheta_\sharp|,|\vartheta_\sharp|)$. Then, equation \eqref{NLD} admits a non--trivial smooth solution of the form 
\begin{equation}\label{eq:uvsol}
\Psi=\frac{1}{2}\begin{pmatrix}
    u+i\, v \\
    u-i\, v
\end{pmatrix}\,,
\end{equation}
with
\begin{equation}
\label{eq:uvsimm}
u,\,v:\R\to\R\qquad\text{such that}\qquad\left\{\begin{array}{ll}u\,\text{is even and}\,v\,\text{is odd} & \text{if }\,\vartheta_\sharp>0\\[.1cm]u\,\text{is odd and}\,v\,\text{is even} & \text{if }\,\vartheta_\sharp<0.\end{array}\right.
\end{equation}
Moreover, for every $k\in\N$ and every $\varepsilon\in\big(0,(\vartheta_\sharp^2-\mu_\sharp^2)/c_\sharp^2\big)$, such solution satisfies
\begin{equation}\label{eq:expdecay}
\big| \Psi^{(k)}(x)\big| \lesssim_{\,\varepsilon,k,\Psi} \exp\left(- \frac{\sqrt{\vartheta_\sharp^2-\mu^2_\sharp-c_\sharp^2\varepsilon}}{|c_\sharp|}\, \lvert x \rvert\right) \qquad \forall x\in\R\,,
\end{equation}
where $\Psi^{(k)}$ denotes the $k$-th derivative of $\Psi$.
\end{theorem}

\begin{remark}
Note that \eqref{eq:uvsol}$\&$\eqref{eq:uvsimm} entail that any solution $\Psi$ obtained in Theorem \ref{thm:main2} satisfies
\begin{equation}
\label{eq:psisimm}
\overline{\Psi}=\sigma_1\Psi\qquad\text{and}\qquad\Psi(-x)=\left\{\begin{array}{ll}\sigma_1\Psi(x) & \text{if }\,\vartheta_\sharp>0\\[.1cm]-\sigma_1\Psi(x) & \text{if }\,\vartheta_\sharp<0.\end{array}\right.
\end{equation}
\end{remark}
\begin{figure}[h!]
	\centering
	\includegraphics[width=0.6\textwidth]{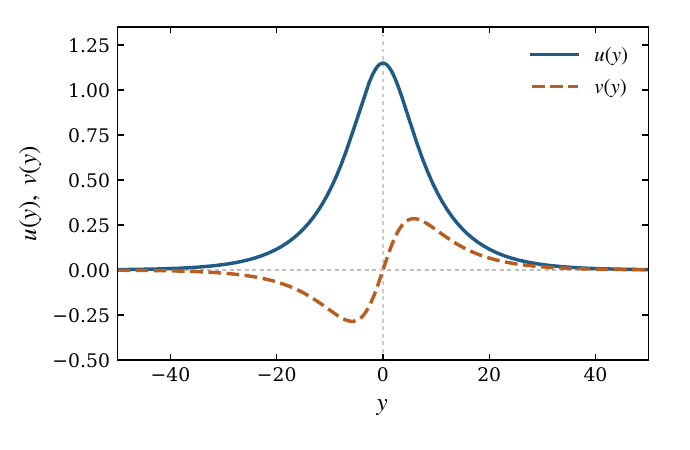}
	
	\caption{Numerically computed profiles $u(y),v(y)$
		defining $\Psi$ in Theorem \ref{thm:main2} for
		$V(x)=2\cos(4\pi x)$, $W(x)=2\cos(2\pi x)$,
		and $\mu_\sharp=\vartheta_\sharp/2$.
		The coefficients of the effective Dirac system were
		computed by a Fourier--Galerkin approximation of the
		Bloch waves. The profiles were obtained by integrating
		the system with the adaptive Runge--Kutta method DOP853.}
	\label{fig:dirac-profiles}
\end{figure}

The existence part of Theorem \ref{thm:main2} is not surprising, as we fix $\mu_\sharp$ in the spectral gap of the operator $\mathcal{D}$ defined in \eqref{eq-diracop} and accounting for the linear part of equation \eqref{NLD}. The same can be said for the established regularity and the decay estimates. On the other hand, the form provided by \eqref{eq:uvsol}, although it is a natural consequence of the dynamical system strategy used for the proof (and based on \cite{B22}), is somehow uncommon and crucial for the proof of the following result.

\begin{theorem}[Dirac solitons for \eqref{NLS}]\label{thm:main1}
    Let $V$ and $W$ satisfy Assumptions \ref{itm:1} and \ref{itm:3} and let $(\pi,\mu_*)$ be a Dirac point for the operator $\mathcal{H}$ defined by \eqref{eq:H}. Let $\Phi_\pm(\cdot,k)$, $k\in[0,2\pi]$, be the two Bloch waves given by Proposition \ref{prop-Blochsym} below. Moreover, let $c_\sharp,\,\vartheta_\sharp\in\R\setminus\{0\}$ be the two parameters given by \eqref{eq:c} and \eqref{eq:theta}, respectively, and
    \begin{gather}
    \label{eq:beta1}
    \beta_1 := \int_0^1 \left| \Phi_+(x,\pi) \right|^2 \left| \Phi_-(x,\pi) \right|^2 dx,\\[.2cm]
    \label{eq:beta2}
    \beta_2:=\int^1_0\overline{\Phi_+}^2(x,\pi)\Phi_-^2(x,\pi)\,dx.
    \end{gather}
   Finally, let
    \begin{equation}
    \label{eq-musciarp}
     \mu_\sharp\in(-\vert \vartheta_\sharp\vert,\vert \vartheta_\sharp\vert)\,,
    \end{equation}
    and $\mathcal{I}_\delta$ as in \eqref{eq:intgap} with $\vert\mu_\sharp\vert /\vert \vartheta_\sharp\vert <a<1$. Then, there exist $\delta_0>0$ and a continuous map
    \[
    (0,\delta_0)\ni\delta\mapsto (\mu_\delta, u_\delta)\in \mathcal I_\delta\times H^2(\R)\,,
    \]
    such that $\mu_\delta$ is given by \eqref{eq:muas} and $u_\delta$ is a non--trivial real--valued Dirac soliton, i.e. a non--trivial real--valued $H^2(\R)$--solution of \eqref{NLS} satisfying
    \begin{equation}
    \label{eq:uapp}
    \left\| u_\delta-\sqrt{\delta}\big(\Psi_-(\delta\,\cdot)\,\Phi_-(\cdot,\pi)+\Psi_+(\delta\,\cdot)\,\Phi_+(\cdot,\pi)\big)\right\|_{H^2(\R)} \lesssim \delta, \qquad \forall\delta\in(0,\delta_0),
    \end{equation}
    where $\Psi:=(\Psi_-,\Psi_+)^T$ is a non--trivial solution of \eqref{NLD} provided by Theorem \ref{thm:main2}.
\end{theorem}

	\begin{remark}
Formula \eqref{eq:uapp} tells us that the solution $u_\delta$ constructed for the NLS \eqref{NLS} is given by the modulation of the two Bloch waves $\Phi_\pm(\cdot,\pi)$ at the linear band crossing through the components of a non-trivial spinor solving \eqref{NLD}, up to an error which is of order $\delta$ in $H^2(\R)$, while the principal part is of order $O(1)$.
\end{remark}

\begin{figure}[h!]
	\centering
	\begin{minipage}[t]{0.32\textwidth}
		\centering
		\includegraphics[width=\linewidth]{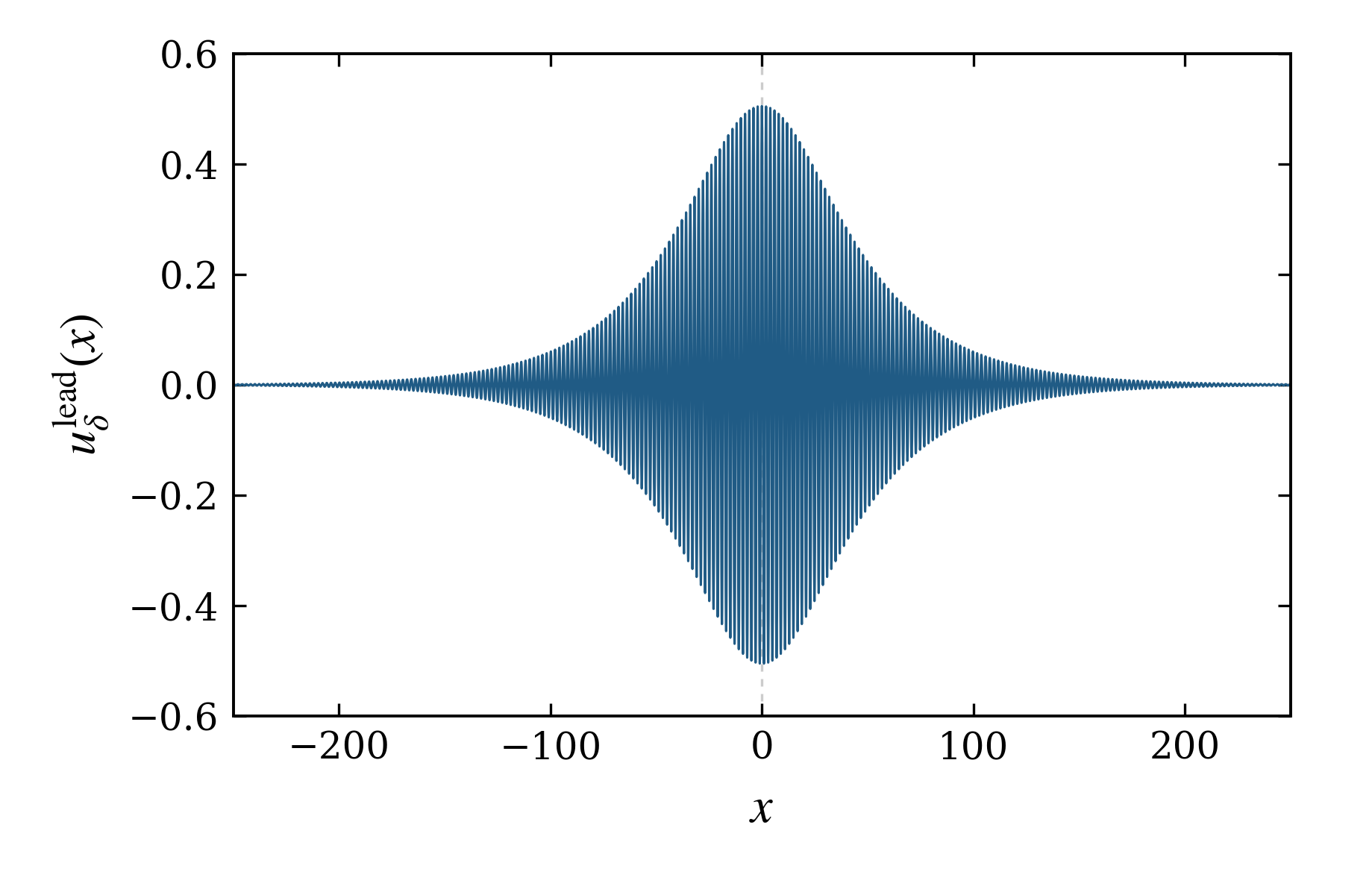}
		
		\small \textsc{(a)} $x\in[-250,250]$
	\end{minipage}
	\hfill
	\begin{minipage}[t]{0.32\textwidth}
		\centering
		\includegraphics[width=\linewidth]{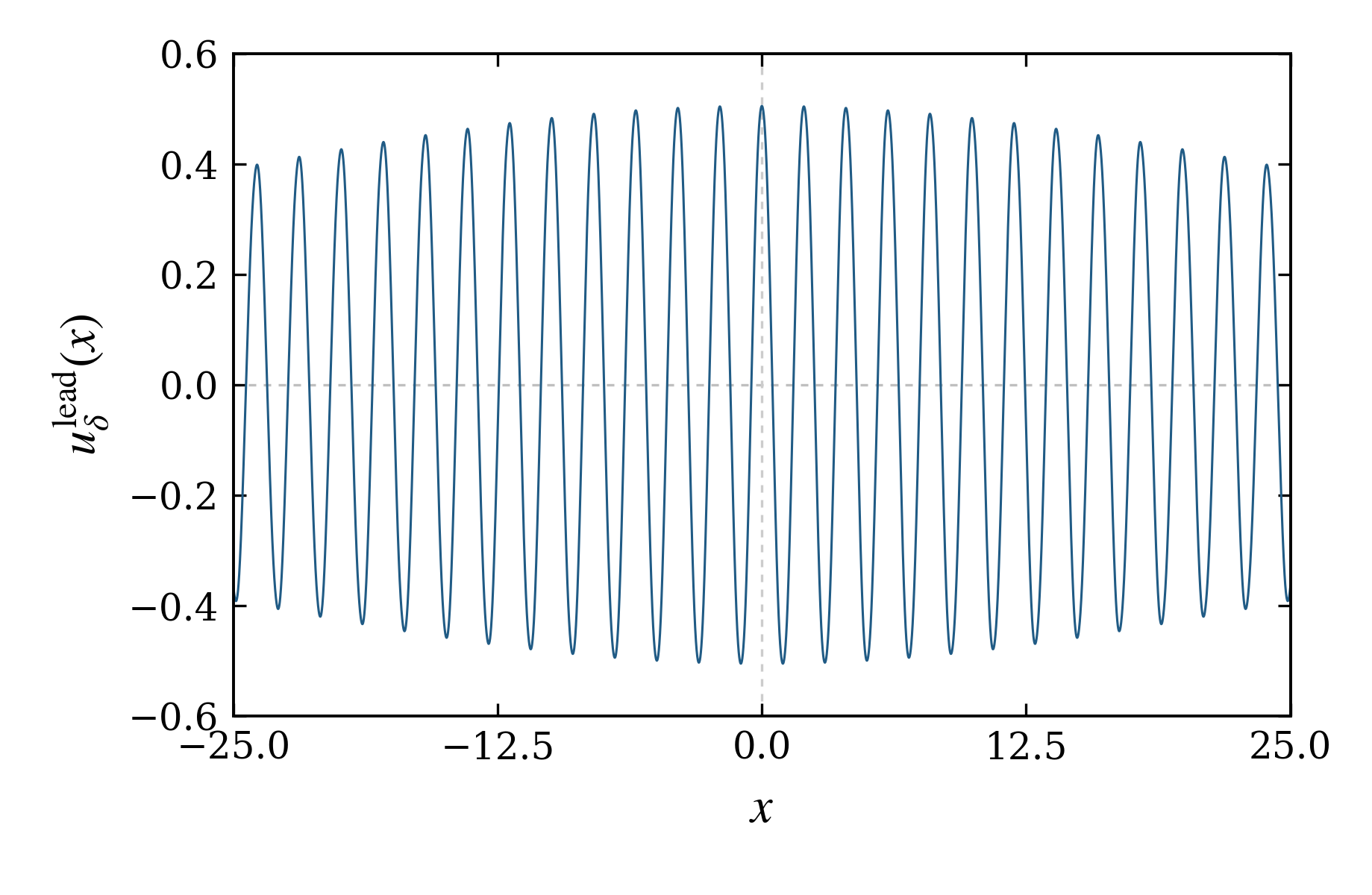}
		
		\small \textsc{(b)} $x\in[-25,25]$
	\end{minipage}
	\hfill
	\begin{minipage}[t]{0.32\textwidth}
		\centering
		\includegraphics[width=\linewidth]{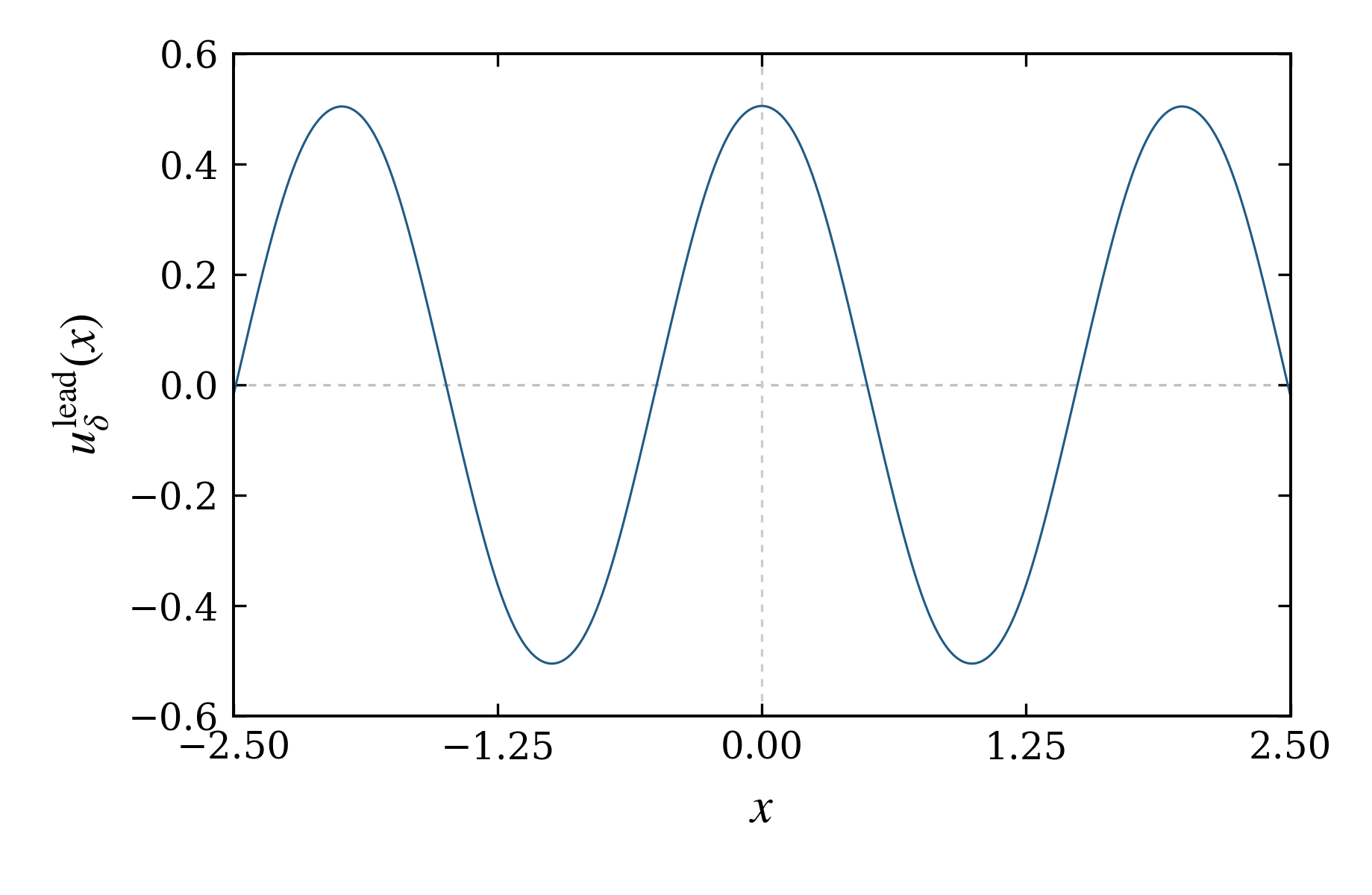}
		
		\small \textsc{(c)} $x\in[-2.5,2.5]$
	\end{minipage}
	
	\caption{Leading-order approximation $u_\delta$ of the Dirac soliton constructed in Theorem \ref{thm:main1}
		for $V(x)=2\cos(4\pi x)$, $W(x)=2\cos(2\pi x)$,
		$\delta=0.2$ and $\mu_\sharp=\vartheta_\sharp/2$,
		displayed at three spatial scales. The Bloch waves were approximated by a Fourier--Galerkin
		method with 33 modes.}
	\label{fig:dirac-soliton-scales}
\end{figure}

\begin{remark}
\label{rem-teorel}
It is straightforward to see that parameters $c_\sharp,\,\vartheta_\sharp,\,\beta_1,\,\mu_\sharp$ set in Theorem \ref{thm:main1} satisfy the assumptions of Theorem \ref{thm:main2}. Moreover, Remark \ref{rem-beta2real} below shows that $\beta_2\in\R$, thus granting the existence of the solution of \eqref{NLD} required to construct the aimed Dirac soliton.
\end{remark}

As already mentioned, Theorem \ref{thm:main1} can be considered a nonlinear analogue of \cite[Theorem 5.1]{FLW17}). In fact, the strategy of the proof is strongly inspired by \cite{FLW17} and consists of making the ansatz \eqref{eq:uform} of the form of $u_\delta$ and detecting the modulating coefficients $\Psi_\pm$ and the correction term $r_\delta$ for this to be a solution. More precisely, this requires first to study the effective equation satisfied by $\Psi$ and then to find a suitably small solution to the equation of the correction term. Concerning the former issue, a crucial step in our argument is to construct specific solutions enjoying the symmetries obtained in  Theorem \ref{thm:main2}. Together with the choice of a family of Bloch waves as in Proposition \ref{prop-Blochsym} below, this guarantees that the leading order term $\Psi_-(\delta\,\cdot)\,\Phi_-(\cdot,\pi)+\Psi_+(\delta\,\cdot)\,\Phi_+(\cdot,\pi)$ is real--valued and even/odd, depending on the sign of $\vartheta_\sharp$ (see \eqref{eq:psisimm}). Such features have a twofold consequence on the study of the correction term: they strongly simplify the explicit form of its equation, and allow to search for a solution $r_\delta$ with the same symmetries. In particular, this latter fact enables us to invert the linearization of \eqref{NLD} on a suitable subspace of its domain, which is crucial to conclude the fixed-point argument that lies at the very heart of the paper.

Finally, note that the operator $\mathcal{H}$ defined in \eqref{eq:H} may have more than a single Dirac point (and this is in fact the case as showed by \cite{FLW17}), but necessarily on the same line $k=\pi$ in the quasi--momentum/energy plane (see, e.g., \cite[Theorem XIII.89]{RS81}). However, this does not affect the proof of Theorem \ref{thm:main1}, as we can repeat the same argument at the energy levels of any other Dirac point of $\mathcal{H}$.


\subsection{Organization of the paper}

The paper is organized as follows:

\begin{enumerate}[label=(\roman*)]
\item in Section \ref{sec:fb} we recall briefly the main tools on periodic Schr\"odinger operators (and Floquet--Bloch theory) in dimension one which are necessary in our paper, and prove that our choice for the Bloch waves is consistent;
\item in Section \ref{sec:limiteq} we prove Theorem \ref{thm:main2};
\item in Section \ref{sec:der_NLD} we explain the reasons why we search for solutions of the form \eqref{eq:uform} in Theorem \ref{thm:main1} and also show the structure the correction term $r_\delta$ must have in order to get a solution of \eqref{NLS};
\item in Section \ref{sec:proofmain2} we prove Theorem \ref{thm:main1}.
\end{enumerate}


 \section{Features of periodic Schr\"odinger operators in dimension one}
 \label{sec:fb}
 
In this section, for the convenience of the reader, we collect some features of one-dimensional periodic Schr\"odinger operators which are widely used throughout, with particular attention to the framework set in \cite{FLW17} where Dirac points appear in the dispersion bands. A more general discussion (in any dimension) can be found e.g. in \cite{RS81}. Here we limit ourselves to mention only what is needed for our purposes, referring to \cite{FLW17, RS81} for further details. 

From now on we let $V$ be a potential satisfying Assumption \ref{itm:1} and $\mathcal{H}$ be the associated self-adjoint operator on $L^2(\R)$, with domain $H^2(\R)$, defined by \eqref{eq:H}. Moreover, since the potential $V$ is 1--periodic, it is natural to consider the lattice $\mathbb Z$ with fundamental cell $[0,1]$ in physical space, and its dual lattice $2\pi \mathbb Z$ with fundamental cell $[0,2\pi]$ (the \emph{Brillouin zone}) in the frequency domain.


\subsection{Preliminaries and specific notation}
\label{subsec-preliminaries}

We start recalling the following definitions.

\begin{definition}[Pseudoperiodic Lebesgue/Sobolev spaces]
\label{def-lebsobper}
Given $k\in[0,2\pi]$, we define the space of $k$--pseudoperiodic $L^2$ functions as the set
\begin{equation}\label{eq:L2pp}
L^2_k(\R):=\left\{f \in L^2_{\mathrm{loc}}(\R): f(x+1)= e^{ik} f(x),\:\forall x\in\R\right\},
\end{equation}
endowed with inner product $\langle f,g\rangle_{L^2([0,1])}$ (and the associated norm). The Sobolev spaces $H^s_k(\R)$ are then defined in the natural way, for all $s\in\mathbb N$.
\end{definition}

\begin{remark}
\label{rem-fkdep}
Note that in principle, by Definition \ref{def-lebsobper}, functions in $L_k^2(\R)$ depend on $k$. When needed, we will make this dependence explicit, using e.g. $f(\cdot,k)$, to prevent ambiguities.
\end{remark}

\begin{remark}
\label{rem:frompseudotoper}
Definition \ref{def-lebsobper} implies that
\begin{equation}
\label{eq:frompseudotoper}
f\in L_k^2(\R)\qquad\Leftrightarrow\qquad f(x)=e^{ikx}P(x)\quad\text{with}\quad P\in L^2_{\mathrm{loc}}(\R)\quad\text{and}\quad\text{$1$--periodic}.
\end{equation}
Moreover, in view of Remark \ref{rem-fkdep}, the function $P$ may depend directly on $k$ too and thus may be denoted, when necessary, by $P(\cdot,k)$.
\end{remark}

\begin{definition}[Even/odd--index 1--periodic Fourier series]
We say that $P_\ee$ is an \emph{even--index 1--periodic Fourier series} if it is of the form
\begin{equation}
\label{eq-even}
P_\ee(x)=\sum_{m\in2\Z}p(m)\,e^{2\pi imx},\qquad\text{with}\qquad\sum_{m\in2\Z}|p(m)|^2<+\infty,
\end{equation}
while we say that $P_\oo$ is an \emph{odd--index 1--periodic Fourier series} if it is of the form
\begin{equation}
\label{eq-odd}
P_\oo(x)=\sum_{m\in2\Z+1}p(m)\,e^{2\pi imx},\qquad\text{with}\qquad\sum_{m\in2\Z+1}|p(m)|^2<+\infty.
\end{equation}
\end{definition}

\begin{definition}[Even/odd--index subspaces of $L_k^2(\R),\,H_k^s(\R)$]
\label{def:perspaces}
For any fixed $k\in [0,2\pi]$ we define the following subspaces of $L_k^2(\R)$:
\begin{gather}
\label{eq:L2ppe}
L_{k,\ee}^2(\R):=\left\{f\in L_k^2(\R): f(x)=e^{ikx}P_\ee(x),\,\text{for some } P_\ee\:\text{ as in \eqref{eq-even}}\right\}\\[.2cm]
\label{eq:L2ppo}
L_{k,\oo}^2(\R):=\left\{f\in L_k^2(\R): f(x)=e^{ikx}P_\oo(x),\,\text{for some } P_\oo\:\text{ as in \eqref{eq-odd}}\right\}
\end{gather}
endowed with the inner product of $L_k^2(\R)$. The Sobolev spaces $H^s_{k,\ee}(\R),\,H^s_{k,\oo}(\R)$ are then defined in the natural way, for all $s\in\mathbb N$. Finally, for the sake of simplicity, we set $L_{\ee}^2(\R):=L_{0,\ee}^2(\R)$, $L_{\oo}^2(\R):=L_{0,\oo}^2(\R)$, $H_{\ee}^s(\R):=H_{0,\ee}^s(\R)$ and $H_{\oo}^s(\R):=H_{0,\oo}^s(\R)$.
\end{definition}

\begin{remark}
By Assumptions \ref{itm:1}-\ref{itm:3}, $V\in L_\ee^2(\R)$ and $W\in L_\oo^2(\R)$. Moreover, it is easy to check that, for all $k\in[0,2\pi]$ and all $s\in\N_+$,
\[
L_k^2(\R)=L_{k,\ee}^2(\R)\oplus L_{k,\oo}^2(\R),\qquad\text{and}\qquad H_k^s(\R)=H_{k,\ee}^s(\R)\oplus H_{k,\oo}^s(\R).
\]
\end{remark}

 
 \subsection{Fundamentals of Floquet--Bloch theory}
\label{subsec-fundamentals}
 
We sketch here the main tools of Floquet--Bloch theory applied to the operator $\mathcal{H}$. First, for fixed $k\in[0,2\pi]$, denote by $\mathcal{H}(k)$ the operator on $L_k^2(\R)$ with the same action of $\mathcal{H}$ and domain $H_k^2(\R)$. One can check that it is self-adjoint and has compact resolvent. As a consequence, the associated eigenvalue problem, usually called \emph{Floquet--Bloch eigenvalue problem}, i.e.
 
\begin{equation}
\label{eq:FB}
\begin{cases}
\mathcal H(k)\,\Phi(x,k)=\mu(k)\,\Phi(x,k) \\[.2cm]
\Phi(x+1,k)=e^{ik}\Phi(x,k)\,,
\end{cases}
\end{equation}
originates a purely discrete spectrum for $\mathcal{H}(k)$, and its eigenvalues listed with repetitions are
\begin{equation}\label{eq:incev}
\mu_1(k)\leq \mu_2(k)\leq \ldots \mu_n(k)\leq\ldots\,.
\end{equation}
The functions
\begin{equation}\label{eq:bands}
[0,2\pi] \ni k \to \mu_n(k)\in\R\,,\qquad n\in\mathbb N_+,
\end{equation}
are usually called \emph{dispersion bands (or curves)} of the operator $\mathcal H$, and
\begin{equation}
\label{eq-genbloch}
\boxed{\text{the corresponding normalized eigenfunctions }\:\Phi_n(x,k)\:\text{ are called \emph{Bloch waves}.}}
\end{equation}
The main property of the dispersion bands is that they  provide the spectrum of $\mathcal{H}$ as follows (see, e.g., \cite[Section XIII.16]{RS81}):
\[
\sigma(\mathcal H)=\sigma_{\mathrm{a.c.}}(\mathcal H)=\bigcup_{n\in\N_+}\mu_n([0,2\pi])\,.
\]
Moreover, other relevant properties of the dispersion bands are (for details see \cite[Section 2]{FLW17} and again \cite[Section XIII.16]{RS81}):
\begin{enumerate}[label=\alph*)]
\item $\mu_n(k)=\mu_n(2\pi-k)$, for every $k\in[0,\pi]$ and every $n\in\N_+$;
\item $\mu_n(\cdot)$ is Lipschitz continuous on $[0,\pi]$, analytic on $(0,\pi)\cup(\pi,2\pi)$ and monotone on $[0,\pi]$ and on $[\pi,2\pi]$ (with different monotonicity on the two intervals) for every $n\in \N_+$;
\item $\mu_1(k)< \mu_2(k)< \ldots \mu_n(k)<\ldots$, for every $k\in(0,\pi)\cup(\pi,2\pi)$.
\end{enumerate}

\begin{remark}
Note that the previous features imply that any $\mu_n(k)$ is a simple eigenvalue for every $k\in(0,\pi)\cup(\pi,2\pi)$. On the other hand, at the point $k=\pi$ where two bands may touch each other (and thus the eigenvalue may be multiple) the analyticity of the dispersion bands may fail. As we mention below, this is in fact the case when a Dirac point arises. Moreover, since \eqref{eq:FB} is a second order one--dimensional ODE, at $k=\pi$ the multiplicity of an eigenvalue at a contact point is at most two. Thus, the collapse of more than two dispersion bands may never occur.
\end{remark}

In addition, (again from \cite[Section 2]{FLW17}) it is also well-known that, for any fixed $k\in[0,2\pi]$, $\big(\Phi_n(\cdot,k)\big)_{n\in\N_+}$ is an orthonormal basis of $L_k^2(\R)$ and that the family of Bloch waves
\begin{equation}
\label{eq-compl1}
\left\{\Phi_n(\cdot,k)\,:\, n\in\N_+,\, k\in[0,2\pi]\right\}
\end{equation}
is complete in $L^2(\R)$, in the sense that for every $g\in L^2(\R)$ one has
\begin{equation}
\label{eq:L2dec}
g(x)=\frac{1}{2\pi}\sum_{n\in\N_+}\int_0^{2\pi} \widetilde{g}_n(k)\Phi_n(x,k)\,dk
\end{equation}
where $\widetilde{g}_n$ are the so called \emph{Bloch coefficients} of $g$, i.e. those functions in $L^2([0,2\pi])$ such that
\begin{equation}
\label{eq-blochcoeff}
 \widetilde{g}_n(k):=\langle g,\Phi_n(\cdot,k)\rangle_{L^2(\R)},\qquad\forall n\in\N_+\,.
\end{equation}
Moreover, the following Parseval identity holds
\begin{equation}
\label{eq-parseval}
\Vert g\Vert^2_{L^2(\R)}=\frac{1}{(2\pi)^2}\sum_{n\in\N_+}\int_0^{2\pi} \left\vert \widetilde{g}_n(k)\right\vert^2\,dk.
\end{equation}
Finally, in view of the Weyl asymptotics, i.e.
\begin{equation}
\label{eq-weylas}
C_1 n^2\leq \mu_n(k)\leq C_2 n^2, \qquad\forall k\in[0,2\pi],\quad \forall n\in\N_+\,,
\end{equation}
for some constants $C_1,C_2>0$, one also sees that, for every $s\in\N$, 
\begin{align}
\label{eq:sobeq}
\Vert g\Vert^2_{H^s(\R)}  \sim \int_0^{2\pi} \sum_{n\in\N_+}(1+n^2)^s \left\vert \widetilde{g}_n(k)\right\vert^2\,dk\,.
\end{align}

\begin{remark}
\label{rem-blochcoeff}
As pointed out in \cite[Remark 2.1]{FLW17}, the Bloch coefficients $\widetilde{g}_n$ are well-defined provided that the Bloch waves $\Phi_n$ satisfy, for every $n\in\N_+$,
\begin{equation}
\label{eq-roughreg1}
\Phi_n\in L^2([a,b]\times[0,2\pi]),\qquad\forall a,b\in\R,\,a<b,
\end{equation}
and there exists $G_n\in L^2([0,2\pi])$ such that
\begin{equation}
\label{eq-roughreg2}
G_n^{a,b}(\cdot):=\int_a^b\Phi_n(x,\cdot)g(x)\,dx\longrightarrow G_n(\cdot)\quad\text{in}\quad L^2([0,2\pi]),\quad\text{as }a\to-\infty,\:b\to+\infty,
\end{equation}
for every $g\in L^2(\R)$. However, it is not difficult to find choices for which such requirements are met (use, e.g., \cite[Theorem XIII.89 -- item (f)]{RS81}).
\end{remark}

\begin{remark}
\label{rem-pseudotoper}
Note also that, for every $k\in[0,2\pi]$, $\big(\Phi_n(\cdot,k),\mu_n(k)\big)_{n\in\N_+}$ are eigenpairs in $L_k^2(\R)\times\R$ of \eqref{eq:FB} if and only if $\big(p_n(\cdot,k),\mu_n(k)\big)_{n\in\N_+}$, with $p_n(x,k):=e^{-ikx}\,\Phi_n(x,k)$, are eigenpairs in $L_0^2(\R)\times\R$ of
\begin{equation}
\label{eq:FBper}
\begin{cases}
\mathcal (-(\partial_x+ik)^2+V(x))\,p(x,k)=\mu(k)\,p(x,k) \\[.2cm]
p(x+1,k)=p(x,k)\,.
\end{cases}
\end{equation}
Moreover, for any fixed $k\in[0,2\pi]$, $\big(p_n(\cdot,k)\big)_{n\in\N_+}$ is an orthonormal basis of $L_{0}^2(\R)$.
\end{remark}


\subsection{Dirac points of \texorpdfstring{$\mathcal{H}$}{H}}
\label{subsec-Diracpoints}

In Remark \ref{rem-diracexistence}, we recalled that the existence of potentials $V$ satisfying Assumption \ref{itm:1} is guaranteed by \cite[Theorem 3.8 $\&$ Proposition D.1]{FLW17}. In particular, such results ensure that potentials giving rise to \emph{Dirac points} for the operator $\mathcal H$ do exist, and that these points are of the form $(\pi,\mu_*)$ in the quasimomentum/energy plane.

For the sake of brevity, we do not mention here neither the general definition of Dirac point (or \emph{linear band crossing of Dirac type}) given in \cite[Definition 3.1]{FLW17}, nor the precise results proving their existence mentioned before. In our context, we simply recall that a point $(\pi,\mu_*)$, with $\mu_*\in\R$, is a Dirac point for the operator $\mathcal{H}$ if the following holds:
\begin{enumerate}[label={\bf (\Roman*)}]
\item there exists $n^*\in\N$ such that $\mu_*=\mu_{n^*}(\pi)=\mu_{n^*+1}(\pi)$;
\item $\mu_*$ is an $L_\pi^2(\R)$ double eigenvalue for the operator $\mathcal{H}(\pi)$;
\item the domain of $\mathcal{H}(\pi)$, i.e. $H_\pi^2(\R)$, splits as $H_\pi^2(\R)=H_{\pi,\ee}^2(\R)\oplus H_{\pi,\oo}^2(\R)$, with $\mathcal{H}(\pi):H_{\pi,\ee}^2(\R)\to L_{\pi,\ee}^2(\R)$ and $\mathcal{H}(\pi):H_{\pi,\oo}^2(\R)\to L_{\pi,\oo}^2(\R)$;
\item the \emph{inversion operator} $\mathcal{I}$, defined as $\mathcal{I}(g)(x):=g(-x)$, is such that $\mathcal{I}: H_{\pi,\ee}^2(\R)\to H_{\pi,\oo}^2(\R)$, $\mathcal{I}: H_{\pi,\oo}^2(\R)\to H_{\pi,\ee}^2(\R)$, $\mathcal{I}\circ \mathcal{I} = \mathbb{I}$ and it commutes with $e^{-i\pi x} \mathcal{H}(\pi) e^{i\pi x}$; 
\item there exist two functions
\begin{equation}
\label{eq-kergenerators}
g_1\in L_{\pi,\ee}^2(\R)\qquad\text{and}\qquad g_2:=\mathcal{I}(g_1)\in L_{\pi,\oo}^2(\R)
\end{equation}
such that
\begin{equation}
\label{eq-ker}
\langle g_a,g_b\rangle_{L^2([0,1])}=\delta_{ab},\quad a,b=1,2,\quad\text{and}\quad \mathrm{ker}(\mathcal H(\pi)-\mu_*)=\mathrm{span}\{g_1,g_2\},
\end{equation}
with $\mathcal H(\pi)-\mu_*:H_\pi^2(\R)\subset L_\pi^2(\R)\to L_\pi^2(\R)$;
\item there are numbers $c_\sharp\in\R\setminus\{0\}$ and $\varepsilon_0>0$ and there exist, for every $k\in[0,2\pi]$, \begin{equation}
\label{eq-neweigen}
\boxed{\text{two (normalized) eigenpairs $(\Phi_-(\cdot,k),\mu_-(k)),\,(\Phi_+(\cdot,k),\mu_+(k))$  of \eqref{eq:FB}}}
\end{equation}
such that
\begin{equation}
\label{eq:linband}
\mu_\pm(k)-\mu_*=\pm c_\sharp(k-\pi)\left(1+r_\pm(k-\pi) \right),\qquad\forall k\in(\pi-\varepsilon_0,\pi+\varepsilon_0),
\end{equation}
with $r_\pm$ two smooth functions, defined on $(\pi-\varepsilon_0,\pi+\varepsilon_0)$, such that $r_\pm(0)=0$.
\end{enumerate}

\smallskip
Some comments are in order. First, we recall that in general the dispersion bands $\big(\mu_n(\cdot)\big)_{n\in\N_+}$ are only Lipschitz continuous on $[0,2\pi]$, since their monotonicity and ordering prevent higher regularity at Dirac points (as showed by the fact that $c_\sharp\neq0$ in \eqref{eq:linband}). However, \eqref{eq-neweigen} and \eqref{eq:linband} show that, in one space dimension, a local smooth reparametrization of the touching bands $\mu_{n^*}(\cdot),\,\mu_{n^*+1}(\cdot)$ can be performed, namely by setting, for every $k\in[0,2\pi]$,
\begin{equation}
\label{eq:lmeno}
\mu_{-}(k)=\left\{
\begin{array}{ll}
\displaystyle\mu_{n^*+1}(k) & k\leq\pi \\[.1cm]
\displaystyle\mu_{n^*}(k) & k\geq\pi
\end{array}
\right.
\end{equation}
and
\begin{equation}
\label{eq:lpiu}
\mu_{+}(k)=\left\{
\begin{array}{ll}
\displaystyle\mu_{n^*}(k) & k\leq\pi \\[.1cm]
\displaystyle\mu_{n^*+1}(k) & k\geq\pi\,.
\end{array}
\right.
\end{equation}
In practice, one reparametrizes the bands so that two coincident corner points become the crossing point of two locally linear functions (up to lower order smooth terms).

This smooth reparametrization also entails an improvement of the regularity in $k$ of the Bloch waves $\Phi_\pm$. Indeed, one also has that 
\begin{equation}
\label{eq-eureka}
\Phi_-(\cdot,\pi)=g_1\qquad\text{and}\qquad\Phi_+(\cdot,\pi)=g_2
\end{equation}
(see, e.g., \cite[Eq. (1.12)]{FLW17}), and, mainly, there exist two smooth functions $c_\pm$, defined on $(-\varepsilon_0,\varepsilon_0)$ and satisfying
\begin{equation}
\label{eq-asymptphi}
c_\pm(\kappa)=1+O(\kappa)\qquad\forall \kappa\in(-\varepsilon_0,+\varepsilon_0),
\end{equation}
and two smooth functions $d_\pm$, defined on $[0,1]\times(-\varepsilon_0,\varepsilon_0)$ and satisfying
\begin{equation}
\label{eq-asymptphirest}
d_\pm(x,\kappa)=O(\kappa)\qquad\forall (x,\kappa)\in[0,1]\times(-\varepsilon_0,\varepsilon_0),
\end{equation}
such that
\begin{equation}
\label{eq-smoothphi}
\Phi_\pm(x,k)=c_\pm(k-\pi)\big(\Phi_\pm(x,\pi)+d_\pm(x,k-\pi)\big),\qquad\forall (x,k)\in[0,1]\times(\pi-\varepsilon_0,\pi+\varepsilon_0)
\end{equation}
(see \cite[Eqs. (3.22) and (B.35)$\&$(B.36)]{FLW17}).

\begin{remark}
\label{rem:regP}
Consistently, the reparametrization induces an analogous smoothing effect also for the periodic functions $(p_n(\cdot,k))_{n\in\N},\,k\in[0,2\pi],$ introduced in Remark \ref{rem-pseudotoper}, in the sense that a decomposition as in \eqref{eq-smoothphi} holds also for $p_\pm(x,k):=e^{-ik x}\Phi_{\pm}(x,k)$  (see \cite[Eq. (3.22)]{FLW17}).
\end{remark}

In view of the previous comments, letting
\begin{equation}
\label{eq-indices}
I:=I_1\cup I_2,\qquad\text{with}\qquad I_1:=\{n\in\N_+:n\neq n^*,\,n^*+1\},\quad I_2:=\{-,+\}, 
\end{equation}
the sequence $\big(\Phi_n(\cdot,k)\big)_{n\in I}$, where $\Phi_n$'s are given by \eqref{eq-genbloch} for every $n\in I_1$ and 
\begin{equation}
\label{eq-ourbloch}
\boxed{\Phi_\pm\:\text{ are given by \eqref{eq-neweigen} and satisfy \eqref{eq-eureka}$\&$\eqref{eq-smoothphi}}}\,,
\end{equation}
is an orthonormal basis of $L_k^2(\R)$. As a consequence, the family of Bloch waves
\begin{equation}
\label{eq-compl2}
\left\{\Phi_n(\cdot,k)\,:\, n\in I,\, k\in[0,2\pi]\right\}
\end{equation}
is complete in $L^2(\R)$, as well as \eqref{eq-compl1}. In other words, \eqref{eq:L2dec}, \eqref{eq-blochcoeff}, \eqref{eq-parseval}, \eqref{eq-weylas} (with minor modifications) and \eqref{eq:sobeq} continue to hold when replacing $\N_+$ with the set of indices $I$ defined above.

\begin{remark}
\label{rem-csciarp}
The number $c_\sharp$ in expansion \eqref{eq:linband} is actually the quantity defined by \eqref{eq:c} (see \cite[Eq. (3.21) and before Remark 4.1]{FLW17} in view of \eqref{eq-eureka}).
\end{remark}


\subsection{Symmetries of Bloch waves}
\label{sec:blochsym}

In the last part of this section we show that, exploiting the symmetries of the potential $V$, one can choose a family of Bloch waves with some suitable additional symmetries that will be particularly useful in the next sections.

\begin{proposition}
\label{prop-Blochsym}
    It is possible to choose a family of Bloch waves as in \eqref{eq-compl2} so that
     \begin{equation}
         \label{sym_even}
         \left\{
      \begin{array}{ll}
        \displaystyle \Phi_n(-x,k)=\Phi_n (x,2\pi-k), & \forall k \in[0,\pi)\cup(\pi,2\pi],\quad\forall n\in I_1\\[.1cm]
        \displaystyle\Phi_\pm (-x,k)=\Phi_\mp(x, 2\pi -k),  &\forall k \in [0,2\pi],
    \end{array}
    \right.
    \end{equation}
    and
    \begin{equation}
         \label{sym_conj}
         \left\{
      \begin{array}{ll}
        \displaystyle \overline{\Phi_n}(x,k)=\Phi_n (x,2\pi-k), & \forall k \in[0,\pi)\cup(\pi,2\pi],\quad\forall n\in I_1\\[.1cm]
        \displaystyle\overline{\Phi_\pm} (x,k)=\Phi_\mp(x, 2\pi -k),  &\forall k \in [0,2\pi].
    \end{array}
    \right.
    \end{equation}
\end{proposition}

\begin{proof}
Let $\left\{\Lambda_n(\cdot,k):n\in I,\,k\in[0,2\pi]\right\}$ be a fixed family of Bloch waves as in \eqref{eq-compl2}, for instance the one used in \cite{FLW17}. We now prove that, starting from this family, we can construct a new family of Bloch waves satisfying also \eqref{sym_even} and \eqref{sym_conj}. We proceed by steps.

\smallskip
\emph{Step (i): construction of a family of Bloch waves as in \eqref{eq-compl2} satisfying \eqref{sym_even}}. First, fix $k\in[0,\pi)$. Exploiting item a) in Section \ref{subsec-fundamentals}, \eqref{eq:lmeno}, \eqref{eq:lpiu},  and the facts that the potential $V$ is even and the actions of $\mathcal{H}$ and $\mathcal{H}(k)$ are the same by definition, a direct computation shows that
\[
\begin{split}
&\mathcal{H}\Lambda_n(-x,k)=\mu_n(2\pi-k)\Lambda_n(-x,k)\qquad\forall n\in I_1,\\
&\mathcal{H}\Lambda_\pm(-x,k)=\mu_\mp(2\pi-k)\Lambda_\pm(-x,k)\,.
\end{split}
\]
Since
\[
\begin{split}
&\Lambda_n(-(x+1),k)=e^{-ik}\Lambda_n(-x,k)=e^{i(2\pi-k)}\Lambda_n(-x,k)\qquad\forall n\in I_1,\\
&\Lambda_\pm(-(x+1),k)=e^{-ik}\Lambda_\pm(-x,k)=e^{i(2\pi-k)}\Lambda_\pm(-x,k),
\end{split}
\]
$\Lambda_n(-x,k)$ is a normalized eigenfunction associated to $\mu_n(2\pi-k)$ and $\Lambda_\pm(-x,k)$ are normalized eigenfunctions with eigenvalues $\mu_\mp(2\pi-k)$, so that there exist $\alpha_n(k),\alpha_\pm(k)\in(-\pi,\pi]$ such that $\Lambda_n(-x,k)=e^{i\alpha_n(k)}\Lambda_n(x,2\pi-k)$ and $\Lambda_\pm(-x,k)=e^{i\alpha_\pm(k)}\Lambda_\mp(x,2\pi-k)$. Since this holds also for $k\in(\pi,2\pi]$, there actually exists a sequence of functions $\big(\alpha_n(k)\big)_{n\in I}$ such that $\alpha_n:[0,\pi)\cup(\pi,2\pi]\to(-\pi,\pi]$ and
\begin{gather}
\label{eq-alphandef}
\Lambda_n(-x,k)=e^{i\alpha_n(k)}\Lambda_n(x,2\pi-k)\qquad\forall k\in[0,\pi)\cup(\pi,2\pi],\quad\forall n\in I_1,\\[.1cm]
\label{eq-alphapmdef}
\Lambda_\pm(-x,k)=e^{i\alpha_\pm(k)}\Lambda_\mp(x,2\pi-k)\qquad\forall k\in[0,\pi)\cup(\pi,2\pi].
\end{gather}
Now, easy computations show that
\[
\begin{split}
&\Lambda_n(-x,2\pi-k)=e^{i\alpha_n(2\pi-k)}\Lambda_n(x,k),\qquad\forall k\in[0,\pi)\cup(\pi,2\pi],\quad\forall n\in I_1,\\
&\Lambda_\pm(-x,2\pi-k)=e^{i\alpha_\mp(2\pi-k)}\Lambda_\mp(x,k),\qquad\forall k\in[0,\pi)\cup(\pi,2\pi],
\end{split}
\]
and thus
\begin{gather}
\label{eq-alphansym}
\alpha_n(k)=-\alpha_n(2\pi-k)\qquad\forall k\in[0,\pi)\cup(\pi,2\pi],\quad\forall n\in I_1,\\[.1cm]
\label{eq-alphapmsym}
\alpha_\pm(k)=-\alpha_\mp(2\pi-k)\qquad\forall k\in[0,\pi)\cup(\pi,2\pi].
\end{gather}
Moreover, since
\begin{equation}
\label{eq-nonzero}
\Lambda_n(x,k),\Lambda_\pm(x,k)\neq0,\qquad\text{for a.e.}\quad(x,k)\in\R\times[0,2\pi],
\end{equation}
\eqref{eq-alphandef} and \eqref{eq-alphapmdef} imply that $e^{i\alpha_n(k)},\,e^{i\alpha_\pm(k)}$ are Borel--measurable on $[0,\pi)\cup(\pi,2\pi]$. As the function $\mathrm{atan}2$ is Borel--measurable too, one also has that $\alpha_n,\,\alpha_\pm$ are Borel--measurable on $[0,\pi)\cup(\pi,2\pi]$.

On the other hand, define the functions $h_\pm:[0,\pi)\cup(\pi,2\pi]\to\C$ and $H_\pm:[0,2\pi]\to\C$ such that
\[
h_\pm(k):=e^{i\alpha_\pm(k)}\qquad\text{and}\qquad H_\pm(k):=\left\langle\Lambda_\pm(-\,\cdot,k), \Lambda_\mp(\cdot,2\pi-k)\right\rangle_{L^2([0,1])}.
\]
By dominated convergence and \eqref{eq-asymptphi}, \eqref{eq-asymptphirest} and \eqref{eq-smoothphi}, one has that $H_\pm$ are smooth on $(\pi-\varepsilon_0,\pi+\varepsilon_0)$. Moreover, by \eqref{eq-alphapmdef} and the normalization of the Bloch waves,
\[
H_\pm(k)=h_\pm(k)\|\Lambda_\mp(\cdot,2\pi-k)\|_{L^2([0,1])}^2=h_\pm(k),\qquad\forall k\in(\pi-\varepsilon_0,\pi)\cup(\pi,\pi+\varepsilon_0).
\]
Hence $h_\pm$ are smooth in $(\pi-\varepsilon_0,\pi)\cup(\pi,\pi+\varepsilon_0)$ and can be extended in a smooth way at $k=\pi$. Moreover, since they cannot vanish by definition, by the properties of the complex logarithm we have that $\alpha_\pm(k)$ are smooth in $(\pi-\varepsilon_0,\pi)\cup(\pi,\pi+\varepsilon_0)$ and can be extended in a smooth way at $k=\pi$. Moreover, \eqref{eq-kergenerators}, \eqref{eq-eureka} and \eqref{eq-alphapmdef} entail that $\alpha_\pm(\pi)=0$.

As a consequence, if we define
\begin{gather}
\label{eq-phineven}
\Gamma_n(x,k):=\left\{
\begin{array}{ll}
\displaystyle \Lambda_n(x,\pi) & \text{if }\:k=\pi\\[.1cm]
\displaystyle e^{-i\frac{\alpha_n(k)}{2}}\Lambda_n(x,k) & \text{if }\:k\in[0,\pi)\cup(\pi,2\pi],
\end{array}
\right.\\[.1cm]
\label{eq-phipmeven}
\Gamma_\pm(x,k):=e^{-i\frac{\alpha_\pm(k)}{2}}\Lambda_\pm(x,k),\qquad\forall k\in[0,2\pi],
\end{gather}
then the family $\left\{\Gamma_n(\cdot,k):n\in I,\,k\in[0,2\pi]\right\}$ preserves integrability and regularity features of $\left\{\Lambda_n(\cdot,k):n\in I,\,k\in[0,2\pi]\right\}$ (i.e., Remark \ref{rem-blochcoeff} and \eqref{eq-asymptphi}$,\, $\eqref{eq-asymptphirest}$,\, $\eqref{eq-smoothphi}) and satisfies \eqref{sym_even}. 

\smallskip
\emph{Step (ii): construction of a family of Bloch waves as in \eqref{eq-compl2} satisfying \eqref{sym_even} and \eqref{sym_conj}}. Again, fix $k\in[0,\pi)$. Arguing as in Step (i) and using that the eigenvalues and the potential $V$ are real, one can check that $\overline{\Gamma_n}(x,k)$ is a normalized eigenfunction with eigenvalue $\mu_n(2\pi-k)$ for every $n\in I_1$ and that $\overline{\Gamma_\pm}(x,k)$ are normalized eigenfunctions associated to $\mu_\mp(2\pi-k)$. Hence, there exists a sequence of functions $\big(\gamma_n(k)\big)_{n\in I}$ such that $\gamma_n:[0,\pi)\cup(\pi,2\pi]\to(-\pi,\pi]$ and
\begin{gather}
\label{eq-gammandef}
\overline{\Gamma_n}(x,k)=e^{i\gamma_n(k)}\Gamma_n(x,2\pi-k)\qquad\forall k\in[0,\pi)\cup(\pi,2\pi],\quad\forall n\in I_1,\\[.1cm]
\label{eq-gammapmdef}
\overline{\Gamma_\pm}(x,k)=e^{i\gamma_\pm(k)}\Gamma_\mp(x,2\pi-k)\qquad\forall k\in[0,\pi)\cup(\pi,2\pi].
\end{gather}
Direct computations on $\overline{\Gamma_n}(x,2\pi-k)$ then entail
\begin{gather}
\gamma_n(k)=\gamma_n(2\pi-k)\qquad\forall k\in[0,\pi)\cup(\pi,2\pi],\quad\forall n\in I_1,\nonumber\\[.1cm]
\label{eq-gammapmsym}
\gamma_\pm(k)=\gamma_\mp(2\pi-k)\qquad\forall k\in[0,\pi)\cup(\pi,2\pi]\,.
\end{gather}
Moreover, by \eqref{eq-nonzero},
\begin{equation}
\label{eq-nonzerouno}
\Gamma_n(x,k),\Gamma_\pm(x,k)\neq0\qquad\text{for a.e.}\quad(x,k)\in\R\times[0,2\pi],
\end{equation}
so that \eqref{eq-gammandef} and \eqref{eq-gammapmdef} ensure that $e^{i\gamma_n(k)},\,e^{i\gamma_\pm(k)}$ are Borel--measurable on $[0,\pi)\cup(\pi,2\pi]$, in turn guaranteeing that $\gamma_n,\,\gamma_\pm$ are Borel--measurable on $[0,\pi)\cup(\pi,2\pi]$ too.

In addition, defining the functions $q_\pm:[0,\pi)\cup(\pi,2\pi]\to\C$ and $Q_\pm:[0,2\pi]\to\C$ such that
\[
q_\pm(k):=e^{i\gamma_\pm(k)}\qquad\text{and}\qquad Q_\pm(k):=\left\langle\overline{\Gamma_\pm}(\cdot,k),\Gamma_\mp(\cdot,2\pi-k)\right\rangle_{L^2([0,1])},
\]
by dominated convergence and \eqref{eq-asymptphi}$\&$\eqref{eq-asymptphirest}$\&$\eqref{eq-smoothphi} it follows that $Q_\pm$ are smooth on $(\pi-\varepsilon_0,\pi+\varepsilon_0)$, and, by \eqref{eq-gammapmdef} and the normalization of the Bloch waves,
\[
Q_\pm(k)=q_\pm(k)\|\Gamma_\mp(\cdot,2\pi-k)\|_{L^2([0,1])}^2=q_\pm(k),\qquad\forall k\in(\pi-\varepsilon_0,\pi)\cup(\pi,\pi+\varepsilon_0).
\]
Hence $q_\pm$ are smooth in $(\pi-\varepsilon_0,\pi)\cup(\pi,\pi+\varepsilon_0)$ and can be extended in a smooth way at $k=\pi$, and the same is true for $\gamma_\pm(k)$.

Thus, if we set
\begin{gather}
\label{eq-phinevendue}
\Phi_n(x,k):=\left\{
\begin{array}{ll}
\displaystyle \Gamma_n(x,\pi) & \text{if }\:k=\pi\\[.1cm]
\displaystyle e^{i\frac{\gamma_n(k)}{2}}\Gamma_n(x,k) & \text{if }\:k\in[0,\pi)\cup(\pi,2\pi],
\end{array}
\right.\\[.1cm]
\label{eq-phipmevendue}
\Phi_\pm(x,k):=e^{i\frac{\gamma_\pm(k)}{2}}\Gamma_\pm(x,k),\qquad\forall k\in[0,2\pi],
\end{gather}
then the family $\left\{\Phi_n(\cdot,k):n\in I,\,k\in[0,2\pi]\right\}$ preserves integrability and regularity features of $\left\{\Gamma_n(\cdot,k):n\in I,\,k\in[0,2\pi]\right\}$ (i.e., Remark \ref{rem-blochcoeff} and \eqref{eq-asymptphi}, \eqref{eq-asymptphirest}, \eqref{eq-smoothphi}) as well as \eqref{sym_even}, and satisfies \eqref{sym_conj}.
\end{proof}

\begin{remark}
\label{payattention}
Observe that the first line of \eqref{sym_even} does not say anything about $\Lambda_n(\cdot,\pi)$ for $n\in I_1$. In principle, one may wonder whether the first line of \eqref{sym_even} automatically implies also $\Phi_n(-x,\pi)=\Phi_n(x,\pi)$ for every $n\in I_1$. Actually, this is not the case. Indeed, for $k\neq\pi$ the Bloch waves $\Phi_n$ satisfying the first line of \eqref{sym_even} have been explicitly constructed in Step (i) of the proof of Proposition \ref{prop-Blochsym} by means of $e^{-i\alpha_n(k)/2}\Lambda_n(x,k)$, where $\Lambda_n$ is the Bloch wave associated to the eigenvalue $\mu_n(k)$ given in \cite{FLW17}, and $\alpha_n(k)$ satisfies \eqref{eq-alphansym}. Now, by \cite[Theorem XIII.89 -- item (f)]{RS81}, the original Bloch wave $\Lambda_n$ is analytic in $[0,\pi)\cup(\pi,2\pi]$ and continuous on $[0,2\pi]$. In general, however, higher regularity is not available at $k=\pi$ for $\Lambda_n$. Arguing as in Step (i) of the proof of Proposition \ref{prop-Blochsym}, these regularity properties of $\Lambda_n$ allow to extend $e^{i\alpha_n(k)}$ in a continuous way at $k=\pi$, but no more than the continuity can be granted. However, the sole continuity of $e^{i\alpha_n(k)}$ at $k=\pi$ does not ensure that $\alpha_n(k)$ is continuous, since it can be either that $\alpha_n$ is continuous at $k=\pi$ with $\alpha_n(\pi)=0$, or that $\alpha_n$ has a jump discontinuity at $k=\pi$, with left and right limits taking values $\pi$ or $-\pi$. In fact, both cases occur in general: the former corresponds to even Bloch waves, the latter to odd ones. Hence, the argument developed in the previous proof does not imply any specific relation at $k=\pi$ for the Bloch waves constructed in Proposition \ref{prop-Blochsym} when $n\in I_1$. Note that, for the very same reason, the possible lack of continuity of $\alpha_n$ at $k=\pi$ prevents in general the possibilty to extend also the first line of \eqref{sym_conj} at $k=\pi$.

On the contrary, this regularity issue does not affect the two Bloch waves $\Phi_\pm$, since from the very beginning of the proof we start with corresponding $\Lambda_\pm$ that are smooth around $k=\pi$ (as they satisfy \eqref{eq-smoothphi}). As shown in detail along the proof, this additional regularity (which is due to the re-parametrization of the branches of Bloch waves around $k=\pi$, and it is thus not in constrast with the aforementioned general regularity result reported in \cite[Theorem XIII.89 -- item (f)]{RS81}) yields a corresponding higher regularity for $\alpha_\pm$, which is turn consistent with the second lines of \eqref{sym_even}$\&$\eqref{sym_conj} at $k=\pi$. 
\end{remark}

\begin{remark}
\label{rem-nostrescelte}
In this paper we will use only the family of Bloch waves constructed in Proposition \ref{prop-Blochsym}. Observe that these waves display phase rotations, depending solely on $n$ and $k$, with respect to the ones used in \cite{FLW17}. Precisely,
\begin{gather}
\label{eq-phinevendef}
\Phi_n(x,k):=\left\{
\begin{array}{ll}
\displaystyle \Lambda_n(x,\pi) & \text{if }\:k=\pi\\[.1cm]
\displaystyle e^{i\nu_n(k)}\Lambda_n(x,k) & \text{if }\:k\in[0,\pi)\cup(\pi,2\pi],
\end{array}
\right.\\[.1cm]
\label{eq-phipmevendef}
\Phi_\pm(x,k):=e^{i\nu_\pm(k)}\Lambda_\pm(x,k),\qquad\forall k\in[0,2\pi],
\end{gather}
with
\[
\nu_n:=\frac{\gamma_n-\alpha_n}{2},\qquad\forall n\in I,
\]
with $\alpha_n,\,\gamma_n$ defined in the proof of Proposition \ref{prop-Blochsym}. Even though this phase rotation is crucial to have \eqref{sym_even} and \eqref{sym_conj}, it does not affect any integral in the $x$ variable present in \cite{FLW17} and used here. However, we will explicitly mention the points throughout where this phase rotation give rise to significant differences with respect to \cite{FLW17}.
\end{remark}

\begin{remark}
\label{rem-beta2real}
Using \eqref{eq:beta2}, \eqref{sym_even} and the pseudo--periodicity of the Bloch waves at $k=\pi$, we have that $\beta_2\in\R$, since
    \[
    \begin{split}
    \beta_2&=\int^1_0\overline{\Phi_+}^2(x,\pi)\Phi_-^2(x,\pi)\,dx =\int^1_0\overline{\Phi_-}^2(-x,\pi)\Phi_+^2(-x,\pi)\,dx \\
    &=\int^0_{-1}\overline{\Phi_-}^2(y,\pi)\Phi_+^2(y,\pi)\, dy=\int^1_0\overline{\Phi_-}^2(z-1,\pi)\Phi_+^2(z-1,\pi)\,dz \\
    &=\int^1_0\overline{\Phi_-}^2(z,\pi)\Phi_+^2(z,\pi)\,dz=\overline{\beta_2}.
    \end{split}
    \]
\end{remark}


\section{Proof of Theorem \ref{thm:main2}}
\label{sec:limiteq}

In this section we prove Theorem \ref{thm:main2}, constructing a solution $\Psi$ of equation \eqref{NLD} in the form
\[
\Psi=\frac{1}{2}\begin{pmatrix}
    u+i\, v \\
    u-i\, v
\end{pmatrix}
\]
with $u:\R\to\R$ even and $v:\R\to\R$ odd. The proof is based on a dynamical system approach already exploited for Dirac equations in \cite{B22}.

\begin{proof}[Proof of Theorem \ref{thm:main2}]
For the sake of clarity, it is convenient to divide the proof in two steps.

\emph{Step(i): existence of a smooth solution of the form \eqref{eq:uvsol}$\&$\eqref{eq:uvsimm}.} Let us first prove the result in the case
    \[
    c_\sharp=1\,,\qquad \vartheta_\sharp=1\,.
    \]
    In this setting, denoting $\mu_\sharp$ simply by $\mu$ and by $f'$ the derivative of the one--variable function $f$, equation \eqref{NLD} for $\Psi=(\Psi_-,\Psi_+)^T$ reads
    \begin{equation}
    \label{eq:sist1}
    \begin{cases}
        i\Psi_-'+\Psi_+-\mu\Psi_-=(\beta_1|\Psi_-|^2+2\beta_1|\Psi_+|^2)\Psi_-+\beta_2\overline{\Psi_-}\Psi_+^2 & \\[.1cm]
        -i\Psi_+'+\Psi_--\mu\Psi_+=\beta_2\overline{\Psi_+}\Psi_-^2+(\beta_1|\Psi_+|^2+2\beta_1|\Psi_-|^2)\Psi_+ & 
    \end{cases}
    \end{equation}
    with $\mu\in(-1,1)$. Then, summing the two lines of \eqref{eq:sist1} and defining $(\varphi_1,\varphi_2)$ as
\begin{equation}\label{eq:varphi}
\begin{cases}
    \varphi_1:=\Psi_-+\Psi_+ \\[.1cm]
    \varphi_2:=\Psi_--\Psi_+
\end{cases} 
\iff 
\begin{cases}
 \Psi_-=(\varphi_1+\varphi_2)/2 \\[.1cm]
  \Psi_+=(\varphi_1-\varphi_2)/2
\end{cases},
\end{equation}
one gets
\begin{multline}
\label{eq:lin1}
i\varphi'_2+\varphi_1-\mu\varphi_1=\\[.1cm]
=\beta_1\big(\vert\Psi_-\vert^2\Psi_-+\vert\Psi_+\vert^2\Psi_+\big)+2\beta_1\big(\vert\Psi_+\vert^2\Psi_-+\vert\Psi_-\vert^2\Psi_+\big)+\beta_2\big(\overline{\Psi_-}\Psi_+^2+\overline{\Psi_+}\Psi_-^2\big)\,.
\end{multline}
Moreover, \eqref{eq:varphi} yields
\begin{align*}
    \vert\Psi_-\vert^2\Psi_-+\vert\Psi_+\vert^2\Psi_+ & = \frac{1}{4}\big(\vert\varphi_1\vert^2+\vert\varphi_2\vert^2\big)\varphi_1+\frac{1}{2}\Re(\varphi_1\overline{\varphi_2})\varphi_2\,,\\[.1cm]
    \vert\Psi_+\vert^2\Psi_-+\vert\Psi_-\vert^2\Psi_+ & = \frac{1}{4}\big(\vert\varphi_1\vert^2+\vert\varphi_2\vert^2\big)\varphi_1-\frac{1}{2}\Re(\varphi_1\overline{\varphi_2})\varphi_2\,,\\[.1cm]
    \overline{\Psi_-}\Psi_+^2+\overline{\Psi_+}\Psi_-^2 & = \frac{1}{8}\left(2\vert\varphi_1\vert^2\varphi_1-4\vert \varphi_2\vert^2\varphi_1+2\overline{\varphi_1}\varphi_2^2\right)\,.
\end{align*}
Now, making the following ansatz on $\varphi_1,\,\varphi_2$
\begin{equation}\label{eq:uv}
\varphi_1(y)= u(y)\,,\qquad \varphi_2(y)=i\, v(y)\,,\qquad y\in\R\,,
\end{equation}
with $u,v$ real--valued functions, one finds that
\begin{align*}
i\varphi'_2+\varphi_1-\mu\varphi_1 & =-v'+u-\mu u\,,\\[.1cm]
\beta_1\big(\vert\Psi_-\vert^2\Psi_-+\vert\Psi_+\vert^2\Psi_+\big) & =\frac{\beta_1}{4}\big(u^2+v^2\big)u\,,\\[.1cm] 2\beta_1\big(\vert\Psi_+\vert^2\Psi_-+\vert\Psi_-\vert^2\Psi_+\big) & =\frac{2\beta_1}{4}\big(u^2+v^2\big)u\,,\\[.1cm]
\beta_2\big(\overline{\Psi_-}\Psi_+^2+\overline{\Psi_+}\Psi_-^2\big) &=\frac{\beta_2}{4}\big(u^2-3v^2\big)u\,.
\end{align*}
Combining the above calculations we obtain that \eqref{eq:lin1} reads
\[
-v'+u-\mu u=\frac{3\beta_1}{4}(u^2+v^2)u+\frac{\beta_2}{4}(u^2-3v^2)u.
\]
Similar computations starting with the difference between the first and the second equation in \eqref{eq:sist1} lead to
\[
u'-v-\mu v=\frac{3\beta_1}{4}(u^2+v^2)v+\frac{\beta_2}{4}(v^2-3u^2)v\,.
\]
Summing up, suitably rearranging terms, we have that the pair $(u,v)$ has to satisfy
\[
    \begin{cases}
        \displaystyle u'=v+\mu v+\frac{3\beta_1-3\beta_2}{4}u^2v+\frac{3\beta_1+\beta_2}{4} v^3 \\[.2cm]
        \displaystyle v'= u-\mu u- \frac{3\beta_1+\beta_2}{4}u^3-\frac{3\beta_1-3\beta_2}{4}v^2u\,.
    \end{cases}
\]
Since by assumption $\beta_1\neq0,\,|\beta_2|\leq\beta_1$, we get
\[
\begin{cases}
4a:=3\beta_1-3\beta_2\geq 3(\beta_1-\vert \beta_2\vert)\geq 0 \\[.1cm]
4b:=3\beta_1+\beta_2\geq\beta_1-\beta_2\geq0\,,
\end{cases}
\]
so that
\begin{equation}\label{eq:ab}
a\geq0\qquad\mbox{and}\qquad b>0\,
\end{equation}
(notice that the second inequality is strict also when $3\beta_1+\beta_2=\beta_1-\beta_2=2\beta_1$).

We are thus led to perform a phase--plane analysis of the following system
\begin{equation}\label{eq:system}
    \begin{cases}
        u'=v+\mu v+a u^2v+b v^3 \\[.1cm]
        v'= u-\mu u- bu^3-av^2u
    \end{cases}\,,
\end{equation}
which turns out to be of Hamiltonian type, that is, of the form
\begin{equation}
  u'=\frac{\partial H}{\partial v} \,,\qquad v'=-\frac{\partial H}{\partial u}\,,
\end{equation}
with Hamiltonian function
\begin{equation}\label{eq:Ham}
H(u,v)=\frac{b}{4}(u^4+v^4)+\frac{a}{2}u^2v^2+\frac{\mu}{2}(u^2+v^2)+\frac{1}{2}(v^2-u^2)\,.
\end{equation}
Moreover, since the solutions of \eqref{eq:sist1} that we are looking for have to vanish at infinity, we require $(u,v)$ to satisfy the asymptotic conditions
\begin{equation}\label{eq:zerolimit}
\lim_{y\to\pm\infty}\big(u(y),v(y)\big)=(0,0)\,
\end{equation}
that is, we are looking for a homoclinic solution with respect to the origin. Since $H(0,0)=0$, the homoclinic we are interested in must lie on the zero--energy level curve
\[
H_0:=\left\{(u,v)\in\R^2\,:\,H(u,v)=0\right\}\,
\]
 \begin{figure}[!h]
        \centering
        \includegraphics[scale=.4]{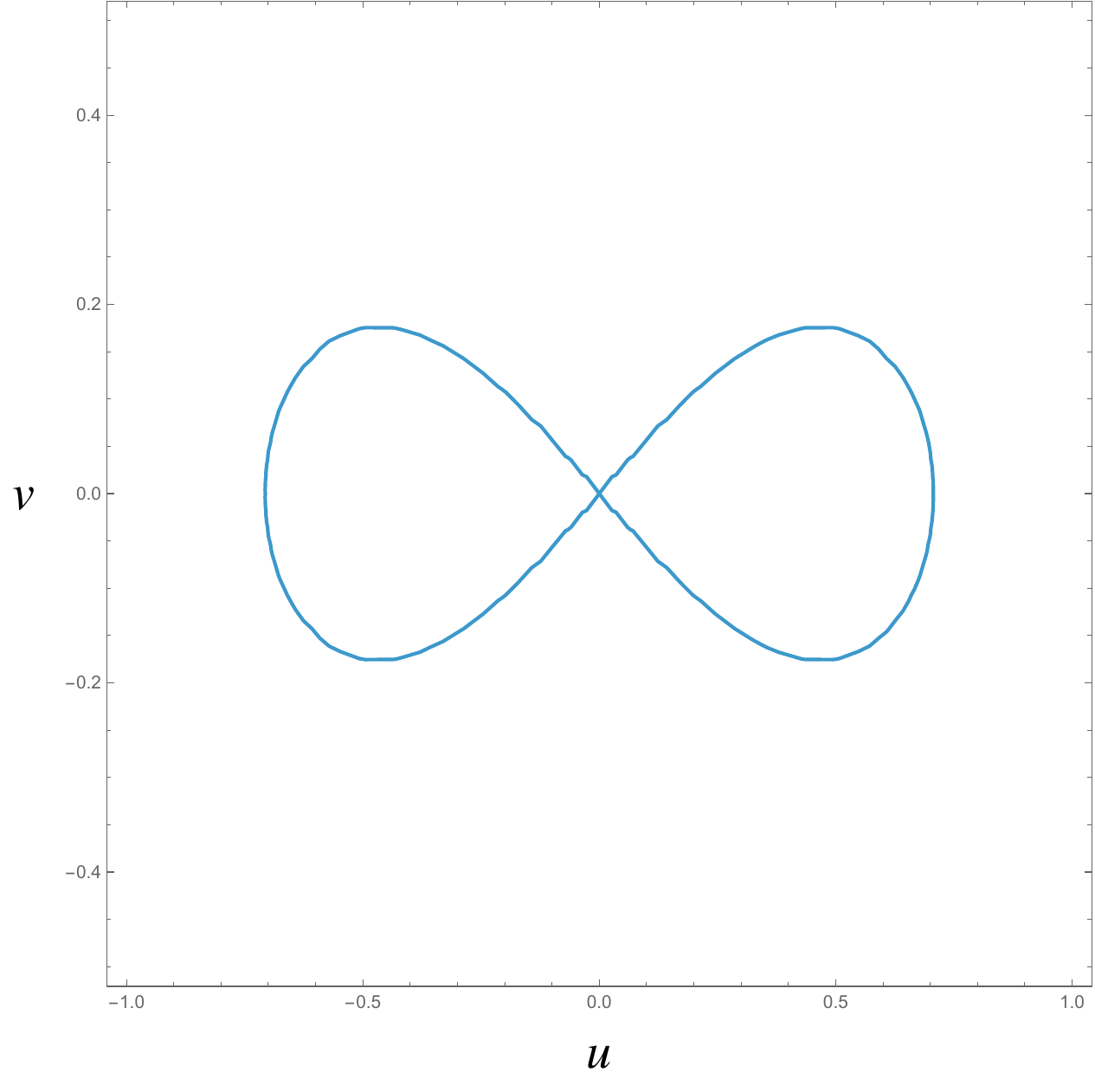}
 \caption{The energy level $H_0$, for $a=b=2$ and $\mu=0.5$.}
 \label{fig1}
   \end{figure}
(see  Figure \ref{fig1}). To find a solution $(u,v)$ of \eqref{eq:system}$\&$\eqref{eq:zerolimit}, we first note that $H(u,0)=0$ if and only if $u=0$ or $u=\pm\sqrt{2(1-\mu)/b}$. Hence, we consider the solution $(u,v)$ of \eqref{eq:system} satisfying
\begin{equation}
\label{eq-ci}
\big(u(0), v(0)\big) = (u_0,v_0) :=\left(\sqrt{2(1-\mu)/b},0\right)\,.
\end{equation}
The right hand side of \eqref{eq:system} being smooth, the local existence of this solution is guaranteed, as well as its uniqueness and smoothness. Since 
\[
\lim_{|(u,v)|\to\infty}H(u,v)=+\infty,
\]
the level set $H_0$ is compact, ensuring that the solution we are considering is actually global, i.e. it exists on the whole real line. Furthermore, a direct computation shows that, if $(u,v)$ satisfies \eqref{eq:system}$\&$\eqref{eq-ci}, then $\big(u(-\,\cdot),-v(-\,\cdot)\big)$ satisfies the same Cauchy problem, thus proving that $u$ is even and $v$ is odd. Therefore, we are left to prove that the couple $(u,v)$ thus constructed satisfies \eqref{eq:zerolimit}. To this aim, define the $\omega$-limit sets
\[
\Omega_\pm := \left\{ (\widetilde u,\widetilde v)\in\R^2:\ \big(u(y_n),v(y_n)\big)\to (\widetilde u,\widetilde v)\text{ for some}\ y_n\to\pm\infty\text{ as }n\to+\infty\right\} \,.
\]
Since, as observed, the set $H_0$ is compact, it is not hard to see that $\Omega_\pm$ is non-empty, compact and connected \cite[Lemma 6.6]{Teschl}. Then, by the Poincar\'e--Bendixson theorem one of the following alternatives holds for $\Omega_\pm$: 
\begin{enumerate}[label=\roman*)]
\item $\Omega_\pm$ is an equilibrium of the system; 
\item $\Omega_\pm$ is a regular periodic orbit (which is then a limit cycle);
\item $\Omega_\pm$ consists of equilibria and non-closed orbits connecting them.
\end{enumerate}
Note first that the continuity of the Hamiltonian gives
\begin{equation}
\label{omlimezeroen}
\Omega_\pm\subseteq H_0\,.
\end{equation}
Moreover, the only equilibrium of the system in $H_0$ is the origin $(0,0)$. Indeed, imposing the conditions
\begin{equation}\label{eq:equil}
    \frac{\partial H}{\partial u}=0\,,\qquad \frac{\partial H}{\partial v}=0\,,
\end{equation}
we obtain the system
\begin{equation}\label{eq:equileq}
\begin{cases}
(bv^2+au^2+\mu+1)v=0 \\[.1cm]
(bu^2+av^2+\mu-1)u=0\,.
\end{cases}
\end{equation}
Since here $\mu\in(-1,1)$ by assumption, by \eqref{eq:ab} the first equation yields
\[
v=0\,.
\]
Substituting into the second equation one finds either $u=0$ or 
\[
bu^2+\mu-1=0\implies u_{\pm}=\pm\sqrt{(1-\mu)/b}\,,
\]
so that the equilibria of the system are the points
\[
(0,0) \,, \quad (u_\pm,0)\,,
\]
and the origin is the only one in $H_0$, as
\[
H(u_\pm,0)= -\frac{(1-\mu)^2}{4b}<0
\]
using again that $\vert \mu\vert<1$. Hence, to prove \eqref{eq:zerolimit} we need to exclude the alternatives $ii)$ and $iii)$ above.

To this end, given our solution $(u,v)$ to \eqref{eq:system}, set
\[
\theta(y):=\arctan\frac{v(y)}{u(y)}\,,\quad y\in\R\,.
\]
First, note that $\theta(y)$ is well-defined for every $y\in\R$. Indeed, if $u(y)=0$ for some $y\in\R$, then the form of $H$ and the fact that $\big(u(y),v(y)\big)\in H_0$  would imply $v(y)=0$. However, since $(0,0)$ is an equilibrium, this would yield $u\equiv v\equiv 0$ on $\R$, which is impossible since by assumption $u(0)=\sqrt{2(1-\mu)/b}>0$. This also proves, in fact, that $u>0$, whence $\theta$ is smooth.

Moreover, using the condition $H\big(u(y),v(y)\big)=0$ for every $y\in\R$, a direct computation shows that
\[
\begin{split}
\theta'(y)=\frac{v'u-vu'}{u^2+v^2}&\,=\frac{u^2-\mu u^2-bu^4-au^2v^2-v^2-\mu v^2-au^2v^2-bv^4}{u^2+v^2}\\
&\,=-\frac{b(u^4+v^4)+2au^2v^2}{2(u^2+v^2)}<0,
\end{split}
\]
where we used \eqref{eq:ab} and the fact that $(u(y),v(y))\neq(0,0)$ for every $y$. Hence, $\theta$ is strictly decreasing on $\R$. Thus, since $u>0$, $\displaystyle \exists\lim_{y\to\pm\infty}\theta(y)=\theta_\pm\in[-\pi/2,\pi/2]$. Now, assume by contradiction that $\Omega_\pm$ is a limit cycle. As $(0,0)$ is an equilibrium, $(0,0)\not\in\Omega_\pm$. Now, let $(\widetilde u,\widetilde v)\in\Omega_\pm$. By the previous remark and \eqref{omlimezeroen} $\widetilde{u}\neq0$. In addition, one can check that $(\widetilde u,\widetilde v)$ lies on a straight line of angular coefficient $\theta_\pm$. Since this is true for any point in $\Omega_\pm$, it cannot be a limit cycle, thus ruling out alternative ii). Moreover, since $H_0$ contains the unique equilibrium $(0,0)$, no closed orbit connecting different equilibria belongs to $H_0$ and alternative iii) does not occur either. As a consequence, i) holds and \eqref{eq:zerolimit} is proved. Whence $\Psi=(\Psi_-,\Psi_+)^T$ given by \eqref{eq:varphi}$\&$\eqref{eq:uv} is a smooth solution of \eqref{NLD} of the desired form.

Although the above argument is developed under the assumption $c_\sharp=\vartheta_\sharp=1$, it clearly applies if this is relaxed to $c_\sharp,\,\vartheta_\sharp>0$ (with some inessential numerical changes). Moreover, it is easily seen that if $c_\sharp<0$, $\vartheta_\sharp>0$ the same argument works with no modifications, as it amounts to work with the Hamiltonian $-H(u,v)$ with $H$ given by \eqref{eq:Ham}, so that the zero--level set $H_0$ and the symmetries remain the same. Conversely, if $c_\sharp>0$, $\vartheta_\sharp<0$ we can argue again as above switching the role of $u$ and $v$, since a change of sign for $\vartheta_\sharp$ only reflects in a corresponding change of sign for the last term of \eqref{eq:Ham}, which in turn means we replace $H(u,v)$ with $H(v,u)$. However, this entails that, if $(u,v)$ solves the associated Cauchy problem, then $\big(-u(-\,\cdot),v(-\,\cdot)\big)$ solves the same one. Hence, here $u$ is odd and $v$ is even. Accordingly, if $c_\sharp<0$, $\vartheta_\sharp<0$, it is enough to work with $-H(v,u)$. Summing up, the existence of a solution $\Psi$ to \eqref{NLD} in the form \eqref{eq:uvsol}$\&$\eqref{eq:uvsimm} is guaranteed independently of the signs of $c_\sharp,\vartheta_\sharp$.

\emph{Step (ii): solutions of Step (i) satisfy \eqref{eq:expdecay}.} Let $\Psi$ be the smooth solution constructed at Step (i) for general $c_\sharp,\,\vartheta_\sharp\in\R\setminus\{0\}$. The equations solved by the associated smooth pair $(u,v)$ are
\begin{equation}
\label{eq:systemgeneral}
    \begin{cases}
        \displaystyle c_\sharp u'=\vartheta_\sharp v+\mu v+a u^2v+b v^3, \\[.1cm]
        \displaystyle c_\sharp v'= \vartheta_\sharp u-\mu u- bu^3-av^2u.
    \end{cases}
\end{equation}
Differentiating the first equation $u$ and further using \eqref{eq:systemgeneral}, one sees that $u$ satisfies
\begin{equation}
\label{eq:u''}
-u''+\frac{\vartheta_\sharp^2-\mu^2}{c_\sharp^2}u=G(y)u
\end{equation}
for a suitable smooth function $G$ such that $G(y)\to0$ as $|y|\to+\infty$ (since $G$ is a polynomial in $u$ and $v$). Recall also that $u>0$ on $(0,+\infty)$ (when $\theta>0$, in fact, on $\R$). Then, for any fixed $0<\varepsilon<(\vartheta_\sharp^2-\mu^2)/c_\sharp^2$, there exists $M_\varepsilon>0$ such that
\[
u''\geq\left(\frac{\vartheta_\sharp^2-\mu^2}{c_\sharp^2}-\varepsilon\right)u>0,\qquad\forall y\geq M_\varepsilon.
\]
Thus, the function
\[
U(y):=u(y)-u(M_\varepsilon)\,e^{-\sqrt{\frac{\vartheta^2-\mu^2}{c_\sharp^2}-\varepsilon\,}\,(y-M_\varepsilon)}
\]
is such that 
\[
U(M_\varepsilon)=0\,,\qquad \lim_{y\to+\infty}U(y)=0\,,
\]
and
\[
U''\geq \left(\frac{\vartheta_\sharp^2-\mu^2}{c_\sharp^2}-\varepsilon\right)U\qquad\forall y\geq M_\varepsilon.
\]
By the comparison principle one gets
\[
U(y)\leq0,\qquad\forall y\geq M_\varepsilon
\]
and thus
\[
u(y)\leq u(M_\varepsilon)e^{-\sqrt{\frac{\vartheta^2_\sharp-\mu^2}{c_\sharp^2}-\varepsilon\,}\,(y-M_\varepsilon)},\qquad\forall y\geq M_\varepsilon.
\]
Since $u$ is even or odd (and smooth), this means that there exists a constant $C_\varepsilon>0$ (depending on $u$) such that
\[
|u(y)|\leq C_\varepsilon e^{-\sqrt{\frac{\vartheta^2_\sharp-\mu^2}{c_\sharp^2}-\varepsilon\,}\,|y|},\qquad\forall y\in\R.
\]
Hence, one obtains the same estimate on $v$ with the same argument (with the only proviso of starting in a neighborhood of $-\infty$).  As a completely analogous argument works on the higher order derivatives just differentiating (and properly using) \eqref{eq:systemgeneral}, \eqref{eq:varphi} and \eqref{eq:uv} yield \eqref{eq:expdecay}.
\end{proof}

\begin{remark}
\label{rem:eqPsi'}
Observe that, by construction, the solution $\Psi$ to \eqref{NLD} given by Theorem \ref{thm:main2} is such that $\Psi_+=\overline{\Psi_-}$. Exploiting this relation, \eqref{NLD} can be rewritten as
\[
\begin{cases}
i\Psi_-'+\Psi_+-\mu_\sharp\Psi_-=3\beta_1|\Psi_-|^2\Psi_-+\beta_2\Psi_+^3 & \\
        -i\Psi_+'+\Psi_--\mu_\sharp\Psi_+=\beta_2\Psi_-^3+3\beta_1|\Psi_+|^2\Psi_+\,. & 
\end{cases}
\]
Differentiating the equation, we obtain that the spinor $\eta:=(\Psi_-',\Psi_+')^T$ satisfies
\begin{equation}
\label{eq-psiprime}
ic_\sharp\sigma_3\eta'+\vartheta_\sharp\sigma_1\eta-\mu_\sharp\eta=\begin{pmatrix}
    6\beta_1|\Psi_-|^2 & 3\left(\beta_1\Psi_-^2+\beta_2\overline{\Psi_-}^2\right) \\
    3\left(\beta_1\overline{\Psi_-}^2+\beta_2\Psi_-^2\right) & 6\beta_1|\Psi_-|^2
\end{pmatrix}\eta\,.
\end{equation}
\end{remark}


\section{Two-scale construction of Dirac solitons for (\ref{NLS})}
\label{sec:der_NLD}
In this section we begin the proof of Theorem \ref{thm:main1}. Precisely, here we develop a multiscale expansion that motivates the specific form of the solution to \eqref{NLS} given in Theorem \ref{thm:main1}, and we derive an equation for the corrector term. The solution of this latter equation will then be the object of the next section.

Let $V,\, W,\, c_\sharp,\, \vartheta_\sharp,\,,\mu_\sharp,\, \beta_1,\, \beta_2$ be as in the hypotheses of Theorem \ref{thm:main1}, with $\Phi_\pm(\cdot,\pi)$ the two Bloch waves given by Proposition \ref{prop-Blochsym}. In view of Remark \ref{rem-teorel}, Theorem \ref{thm:main2} ensures the existence of a smooth non--trivial solution $\Psi=(\Psi_-,\Psi_+)^T$ to \eqref{NLD}, in the form \eqref{eq:uvsol}$\&$\eqref{eq:uvsimm} and satisfying \eqref{eq:expdecay}. Fix, then, one such $\Psi$ and set
\begin{equation}
\label{eq:U0}
U_0(x,y):=\Psi_-(y)\Phi_-(x,\pi)+\Psi_+(y)\Phi_+(x,\pi)\,.
\end{equation}
At this point we make the following ansatz:
\begin{equation}
\label{eq:ansatz}
u_\delta(x):=\sqrt{\delta}\,\mathcal{U}_\delta(x,\delta x),\qquad\text{where}\qquad \mathcal{U}_\delta(x,y):=U_0(x,y)+\delta U_1(x,y)+\delta\eta_\delta(x),
\end{equation}
$U_0$ is as in \eqref{eq:U0} and $U_1(x,y),\,\eta_\delta(x)$ have to be determined so that $u_\delta$ is a solution of \eqref{NLS}, with $\mu_\delta$ as in \eqref{eq:muas} for any given $\mu_\sharp$ satisfying \eqref{eq-musciarp}. Moreover, for each fixed $y\in\R$ we require $U_0$ and $U_1$ to be $\pi$--pseudoperiodic in the variable $x$, i.e.
\begin{equation}
\label{eq-pseudoansatz}
U_j(x+1,y)=e^{i\pi} U_j(x,y)\,,\qquad j=0,1\,.
\end{equation}
In order to identify the functions $U_1(x,y)$ and $\eta_\delta(x)$, we plug \eqref{eq:ansatz} into \eqref{NLS}, treating
\begin{equation}
\label{eq:xdy}
x\quad\text{and}\quad y:=\delta x\quad\text{as independent variables}.
\end{equation}
This yields
 \begin{equation}
 \label{eq:mathU}
    \left( - (\partial_x + \delta \partial_y)^2 + V(x) + \delta W(x) - \mu_* - \delta  \mu_\sharp \right) \mathcal U_\delta(x,y) = \delta \lvert \mathcal U_\delta(x,y) \rvert^2 \mathcal U_\delta(x,y)\,.
\end{equation}
Now, the strategy proceeds by checking that $U_0$ is consistent with \eqref{eq:mathU} to zero-order in $\delta$, while $U_1$ can be appropriately selected to next order in the parameter $\delta$. Finally, the equation for the remainder term $\eta_\delta$ is determined, and its study will be the core of subsequent sections.


\subsection{Consistency of \texorpdfstring{$U_0$}{U}}
\label{subsec:U0}

Focusing on the zero--order terms in $\delta$ of \eqref{eq:mathU} (and recalling \eqref{eq-pseudoansatz} for $U_0$), one obtains that that $U_0$ has to satisfy
    \begin{equation}\label{eq:U0eq}
    (- \partial_x^2 + V(x) - \mu_*) U_0(x,y) = 0
    \end{equation}
    and the $\pi$--pseudoperiodic condition in $x$. However, these are automatically satisfied by the fact that the operator only affects the $x$ variable and by the choice of $U_0$ made in \eqref{eq:U0} (in view of \eqref{eq:L2pp}), since \eqref{eq:H}, \eqref{eq-ker} and \eqref{eq-eureka} entail
    \[
    \ker_{L^2_{\pi}(\R)} (- \partial_x^2 + V(\cdot)- \mu_*)= \Span \left\{ \Phi_+(\cdot,\pi), \Phi_-(\cdot,\pi) \right\}.
    \]


\subsection{Conditions on \texorpdfstring{$U_1$}{U}}
\label{subsec:U1}

 Notice that in \eqref{eq:mathU} there are order 1 terms in $\delta$ involving both $U_1(x,y)$ and $\eta_\delta(x)$. To determine separately these two unknown functions, we start by observing that a first group of order 1 terms in $\delta$ yields the following equation for $U_1(x,y)$
\begin{equation}
\label{eq-U1}
       (- \partial_x^2 + V(x) - \mu_*) U_1(x,y) = (2\partial_x\partial_y - W(x)+\mu_\sharp) U_0(x,y) + \lvert U_0(x,y) \rvert^2 U_0(x,y)\,.
\end{equation}
If the right hand side of \eqref{eq-U1} were identically zero, we could take $U_1(x,y)\equiv0$. From the technical point of view, this would amount to work directly with $U_0(x,y)+\delta\eta_\delta(x)$ in place of $U_0(x,y)+\delta U_1(x,y)+\delta\eta_\delta(x)$ in \eqref{eq:ansatz}. However, since the proof of Theorem \ref{thm:main1} in this case would be identical to the one we develop throughout for a non--trivial $U_1(x,y)$, with no loss of generality from now on we assume that $(2\partial_x\partial_y - W(x)+\mu_\sharp) U_0(x,y) + \lvert U_0(x,y) \rvert^2 U_0(x,y)\not\equiv0$. 

Now, the operator $- \partial_x^2 + V(\cdot)$ is closed and (as we mentioned in Section \ref{subsec-fundamentals}) has a compact resolvent on $L_\pi^2(\R)$. As a consequence, it is a Fredholm operator, so that \eqref{eq-U1} admits a solution $U_1(\cdot,y)$, for any fixed $y\in\R$, if and only if the right hand side is orthogonal to $\ker_{L^2_{\pi}(\R)} (- \partial_x^2 + V(\cdot)- \mu_*)$; that is, if and only if
     \begin{equation}
     \label{FAT}
     \left\langle (2\partial_x\partial_y - W+\mu_\sharp) U_0(\cdot,y)  + \lvert U_0(\cdot,y) \rvert^2 U_0(\cdot,y), \Phi_\pm(\cdot,\pi) \right\rangle_{L^2([0,1])} = 0\,,
    \end{equation}ƒ
where, here and throughout, the inner product is meant with respect to the sole $x$ variable. By \eqref{eq:U0}, \eqref{eq:c} and \eqref{eq:theta}
\begin{align*}
2 \left\langle \partial_x \partial_y U_0(\cdot,y), \Phi_\pm(\cdot,\pi)\right\rangle_{L^2([0,1])}= &\,2\partial_{y}\Psi_-(y) \left\langle \partial_x\Phi_-(\cdot,\pi), \Phi_\pm(\cdot,\pi)\right\rangle_{L^2([0,1])}\\
&\,+2\partial_{y}\Psi_+(y) \left\langle \partial_x\Phi_+(\cdot,\pi), \Phi_\pm(\cdot,\pi)\right\rangle_{L^2([0,1])}\\
=&\,\pm ic_\sharp\partial_y\Psi_\pm(y)
\end{align*}
(the fact that $\left\langle\partial_x\Phi_\pm(\cdot,\pi),\Phi_\mp(\cdot,\pi)\right\rangle_{L^2_x([0,1])}=0$ following by \eqref{sym_conj} and the $\pi$--pseudoperiodicity of $\Phi_\pm(\cdot,\pi)$), whereas
\begin{align*}
- \left\langle W U_0(\cdot,y), \Phi_\pm(\cdot,\pi) \right\rangle_{L^2([0,1])} = &\, - \Psi_-(y)\left\langle W \Phi_-(\cdot,\pi), \Phi_\pm(\cdot,\pi) \right\rangle_{L^2([0,1])} \\
        &\,- \Psi_+(y)\left\langle W \Phi_+(\cdot,\pi), \Phi_\pm(\cdot,\pi) \right\rangle_{L^2([0,1])}\\
        =&\,-\vartheta_\sharp\Psi_\mp(y)
\end{align*}
(the fact that $\left\langle W\Phi_\pm(\cdot,\pi),\Phi_\pm(\cdot,\pi)\right\rangle_{L^2([0,1])}=0$ following by \eqref{WF}, \eqref{eq-kergenerators}, \eqref{eq-eureka}, \eqref{eq:L2ppe}, \eqref{eq:L2ppo} and standard properties of trigonometric functions). Since the orthonormality of $\Phi_+(\cdot,\pi)$, $\Phi_-(\cdot,\pi)$ also gives
\[
\mu_\sharp \left\langle U_0(\cdot,y), \Phi_\pm(\cdot,\pi) \right\rangle_{L^2([0,1])} =  \mu_\sharp\Psi_\pm(y)
\]
the first part of \eqref{FAT} can be rewritten as
\begin{equation}
\label{eq:U11}
\left\langle (2\partial_x\partial_y - W+\mu_\sharp) U_0(\cdot,y), \Phi_\pm(\cdot,\pi) \right\rangle_{L^2([0,1])}=\pm ic_\sharp\partial_y\Psi_\pm(y)-\vartheta_\sharp\Psi_\mp(y)+\mu_\sharp\Psi_\pm(y).
\end{equation}
Let us now focus on the term $\left\langle |U_0(\cdot,y)|^2U_0(\cdot,y),\Phi_\pm(\cdot,\pi)\right\rangle_{L^2([0,1])}$. Since \eqref{eq:U0} gives
\begin{align*}
|U_0(x,y)|^2U_0(x,y)&=|\Psi_-(y)|^2|\Phi_-(x,\pi)|^2\Psi_-(y)\Phi_-(x,\pi)+\Psi_-^2(y)\overline{\Psi_+}(y)\Phi_-^2(x,\pi)\overline{\Phi_+}(x,\pi)\\[.1cm]
& \hspace{-.5cm}  + \Psi_+^2(y)\overline{\Psi_-}(y)\Phi_+^2(x,\pi)\overline{\Phi_-}(x,\pi)+|\Psi_+(y)|^2|\Phi_+(x,\pi)|^2\Psi_+(y)\Phi_+(x,\pi)\\[.1cm]
& \hspace{-.5cm}  +2|\Psi_-(y)|^2|\Phi_-(x,\pi)|^2\Psi_+(y)\Phi_+(x,\pi)+ 2|\Psi_+(y)|^2|\Phi_+(x,\pi)|^2\Psi_-(y)\Phi_-(x,\pi)
\end{align*}
taking the scalar product with $\Phi_\pm(\cdot,\pi)$ we have
\begin{align*}
\left\langle |U_0(\cdot,y)|^2U_0(\cdot,y),\Phi_\pm(\cdot,\pi)\right\rangle_{L^2([0,1])}=&\, |\Psi_-(y)|^2\Psi_-(y)\int_0^1|\Phi_-(x,\pi)|^2\Phi_-(x,\pi)\overline{\Phi_\pm}(x,\pi)\,dx\\[.1cm]
&\,+\Psi_-^2(y)\overline{\Psi_+}(y)\int_0^1\Phi_-^2(x,\pi)\overline{\Phi_+}(x,\pi)\overline{\Phi_\pm}(x,\pi)\,dx\\[.1cm]
&\,+\Psi_+^2(y)\overline{\Psi_-}(y)\int_0^1\Phi_+^2(x,\pi)\overline{\Phi_-}(x,\pi)\overline{\Phi_\pm}(x,\pi)\,dx\\[.1cm]
&\,+|\Psi_+(y)|^2\Psi_+(y)\int_0^1|\Phi_+(x,\pi)|^2\Phi_+(x,\pi)\overline{\Phi_\pm}(x,\pi)\,dx\\[.1cm]
&\,+2|\Psi_-(y)|^2\Psi_+(y)\int_0^1|\Phi_-(x,\pi)|^2\Phi_+(x,\pi)\overline{\Phi_\pm}(x,\pi)\,dx\\[.1cm]
&\,+2|\Psi_+(y)|^2\Psi_-(y)\int_0^1|\Phi_+(x,\pi)|^2\Phi_-(x,\pi)\overline{\Phi_\pm}(x,\pi)\,dx.
\end{align*}
Moreover, since \eqref{eq-kergenerators}, \eqref{eq-eureka}, \eqref{eq:L2ppe}, \eqref{eq:L2ppo}, \eqref{sym_conj} and standard properties of trigonometric functions imply
\begin{align}
\label{int_phi_j}
    \int_0^1 \lvert \Phi_\pm(x,\pi) \rvert^2 \Phi_\pm(x,\pi) \overline{\Phi_\mp}(x,\pi)\, dx & =\int_0^1\Phi_\pm^2(x,\pi)\overline{\Phi_\mp}(x,\pi)\overline{\Phi_\pm}(x,\pi)\,dx\nonumber\\[.1cm]
     & =\int_0^1 \lvert \Phi_\pm(x,\pi)\rvert^2 \Phi_\mp(x,\pi) \overline{\Phi_\pm}(x,\pi)\,dx=0,
\end{align}
from \eqref{eq:beta1}, \eqref{eq:beta2}, (again) \eqref{sym_conj} and Remark \ref{rem-beta2real} we obtain
\[
\left\langle|U_0(\cdot,y)|^2U_0(\cdot,y),\Phi_-(\cdot,\pi)\right\rangle_{L^2([0,1])}=\beta_1\left(|\Psi_-(y)|^2+2|\Psi_+(y)|^2\right)\Psi_-(y)+\beta_2\Psi_+^2(y)\overline{\Psi_-}(y)
\]
and
\[
\left\langle|U_0(\cdot,y)|^2U_0(\cdot,y),\Phi_+(\cdot,\pi)\right\rangle_{L^2([0,1])}=\beta_2\Psi_-^2(y)\overline{\Psi_+}(y)+\beta_1\left(|\Psi_+(y)|^2+2|\Psi_-(y)|^2\right)\Psi_+(y).
\]
Thus, combining with \eqref{eq:U11}, \eqref{FAT} gives rise to the following system (where we now omit the $y$ dependence of the functions as it is the unique one here)
\[
\left\{
\begin{array}{l}
\displaystyle -ic_\sharp\partial_y\Psi_--\vartheta_\sharp\Psi_++\mu_\sharp\Psi_-+\beta_1\left(|\Psi_-|^2+2|\Psi_+|^2\right)\Psi_-+\beta_2\Psi_+\overline{\Psi_-}\Psi_+=0\\[.2cm]
\displaystyle ic_\sharp\partial_y\Psi_+-\vartheta_\sharp\Psi_-+\mu_\sharp\Psi_++\beta_2\Psi_-\overline{\Psi_+}\Psi_-+\beta_1\left(|\Psi_+|^2+2|\Psi_-|^2\right)\Psi_+=0
\end{array}
\right.
\]
However, suitably rearranging terms and setting $\Psi:=(\Psi_-,\Psi_+)^T$, this system coincides with \eqref{NLD}. Hence, recalling that by assumption $\Psi$ is a solution of \eqref{NLD}, \eqref{FAT} is satisfied with $U_0$ given by \eqref{eq:U0}. As a consequence, denoting by $(-\partial^2_x+V(\cdot)-\mu_*)^{-1}$ the inverse on the orthogonal complement of $\ker_{L^2_{\pi}(\R)} (- \partial_x^2 + V(\cdot)- \mu_*)$, we take $U_1$ as
\begin{equation}\label{eq:U1}
    U_1(\cdot,y)=(-\partial^2_x+V(\cdot)-\mu_*)^{-1}G_1(\cdot,y;U_0),
\end{equation}
for every $y\in\R$, with
\begin{equation}\label{eq:G1}
G_1(x,y;U_0):=(2\partial_x\partial_y - W(x)+\mu_\sharp) U_0(x,y) + \lvert U_0(x,y) \rvert^2 U_0(x,y)\,.
\end{equation}

\begin{remark}
\label{rem:realU}
Note that, in view of the previous choices, the functions $U_0(x,\delta x),\,U_1(x,\delta x)$ identified above are real--valued. The fact that $U_0$ is real--valued follows directly by \eqref{eq:U0}, \eqref{eq:uvsol}, \eqref{eq:uvsimm} and \eqref{sym_conj}. As a consequence, the forcing term $G_1$ defined in \eqref{eq:G1} is real--valued too, so that the same holds for the function $U_1$ given by \eqref{eq:U1}. Moreover, if $\vartheta_\sharp>0$, then $U_0(x,\delta x)$ and $U_1(x,\delta x)$ are even. Indeed, by \eqref{eq:U0}, \eqref{eq:uvsol}, \eqref{eq:uvsimm} and \eqref{sym_even} we have
    \[
    \begin{split}
    U_0(-x,-\delta x)&=\Psi_-(-\delta x)\Phi_-(-x,\pi)+\Psi_+(-\delta x)\Phi_+(-x,\pi) \\
    &=\Psi_+(\delta x)\Phi_+(x,\pi)+\Psi_-(\delta x)\Phi_-(x,\pi) = U_0(x,\delta x)\,.
    \end{split}
    \]
    Since $U_0$ is even, this is true also for the function $G_1$ introduced by \eqref{eq:G1}, in turn ensuring that $U_1$ in \eqref{eq:U1} is even too (since, if $f(x)$ is any function such $(-\partial_x^2+V(x)-\mu_*)f(x)=g(x)$ for some even function $g$, then also $h(x):=f(-x)$ satisfies the same equation since $V$ is even too). Analogously, if $\vartheta_\sharp<0$, then $U_0(x,\delta x),\,U_1(x,\delta x)$ are odd.
\end{remark}


\subsection{The equation for \texorpdfstring{$\eta_\delta$}{eta}} 

Given $U_0$ and $U_1$ as in Sections \ref{subsec:U0} and \ref{subsec:U1}, respectively, in order to find a solution $u_\delta$ to \eqref{NLS} in the form \eqref{eq:ansatz} we are then left to construct a suitable correction term $\eta_\delta$. First, since $U_0$ and $U_1$ are real-valued, we may limit ourselves to search for a real--valued $\eta_\delta$. Then, to identify an equation for $\eta_\delta$ we collect all terms in \eqref{eq:mathU} we have not addressed so far. Since each of these terms comes multiplied by a power of $\delta$ of order not smaller than 1, we factor out $\delta$ and obtain the following equation (where we already used that $U_0,U_1$ and $\eta_\delta$ are real--valued)
\begin{equation}
\label{eq_eta}
\Big(- \partial^2_x + V(x) + \delta W(x) - \mu_\delta \Big)  \eta_\delta (x)= \delta F(\delta,\mu_\sharp, x) + \delta \Big ( L_\delta(\eta_\delta)(x) + \mathcal N_\delta (\eta_\delta)(x) \Big )\,,
\end{equation}
with, in view of \eqref{eq:xdy},
\begin{align}
\label{eq:F_delta}
    F(\delta,\mu_\sharp, x) := & \, \big(\partial_y^2U_0\big)(x,\delta x)+ \big((2\partial_x\partial_y - W(x) + \mu_\sharp) U_1\big)(x,\delta x) + 3 U_0^2(x,\delta x)U_1(x,\delta x)\nonumber\\[.2cm]
    & \,
    +\delta \left ( \big(\partial_y^2U_1\big)(x,\delta x) + 3U_0(x,\delta x) U_1^2(x,\delta x) \right) + \delta^2 U_1^3(x,\delta x),
\end{align}
which does not depend on $\eta_\delta$, and
\begin{align}
    \label{eq:L_delta}
    L_\delta(\eta_\delta)& := 3 U_0^2 \eta_\delta + 6 \delta  U_0U_1 \eta_\delta + 3 \delta^2 U_1^2 \eta_\delta,\\[.2cm]
    \label{eq:N_delta}
    \mathcal N_\delta (\eta_\delta)& := 3 \delta U_0 \eta_\delta^2 + 3 \delta^2 U_1 \eta_\delta^2 + \delta^2 \eta_\delta^3.
\end{align}
If for sufficiently small values of $\delta>0$ we manage to find a real--valued function $\eta_\delta$ (satisfying some suitable growth estimates with respect to $\delta$) satisfying \eqref{eq_eta}, then Theorem \ref{thm:main1} will be proved. This is the content of the next section.


\section{Proof of Theorem \ref{thm:main1}}
\label{sec:proofmain2}

In this section we complete the proof of Theorem \ref{thm:main1} constructing a solution $\eta_\delta$ to \eqref{eq_eta} for $\delta$ small enough. To lighten the notation, from now on we will suppress inessential subscripts, writing $\eta$ in place of $\eta_\delta$. Furthermore, we assume that $V,\, W,\, c_\sharp,\, \vartheta_\sharp,\,,\mu_\sharp,\, \beta_1,\, \beta_2$ fulfill the assumptions of Theorem \ref{thm:main1}. Finally, as already pointed out at the end of the previous section, recall that we are looking for a function $\eta$ that is real--valued, so that, by Remark \ref{rem:realU}, the same will be eventually true for the solution $u_\delta$ to \eqref{NLS} as in \eqref{eq:ansatz}. Moreover, since by Remark \ref{rem:realU} we also know that $U_0(x,\delta x),U_1(x,\delta x)$ are even/odd depending on the sign of $\vartheta_\sharp$, we will exploit these symmetries and actually look for $\eta$ even if $\vartheta_\sharp>0$ and odd if $\vartheta_\sharp<0$. Since the proof is essentially the same in the two cases, in what follows we first develop the argument in full details under the assumption $\vartheta_\sharp>0$, and we then highlight the main differences for $\vartheta_\sharp<0$ at the end of the section.

\medskip

Since the proof is rather long and technical, we briefly sketch the strategy we follow. We first use the Bloch decomposition to split the correction into near components, associated with the two bands close to the Dirac point, and far components, comprising the remaining modes. Exploiting spectral separation in the far--energy region, we solve the far equations by a contraction argument, obtaining the far correction as a function of the unknown near correction (Section \ref{sec5.1}). Substituting this function into the near--energy equations and rescaling the quasi--momentum, we obtain a perturbation of the effective Dirac equation linearized at the solitary wave $\Psi$ (Section \ref{sec5.2}). The key step is to invert this linearized operator on a suitable symmetry subspace, removing all non--trivial elements in the kernel (Section \ref{sec:near}). We then control the remaining linear perturbations (Section \ref{sec5.4}) and solve the reduced nonlinear problem by a second contraction argument (Section \ref{sec5.5}). Finally, we reconstruct $\eta$ and the standing wave $u_\delta$, and combine the resulting bounds with the estimates from Section \ref{sec:der_NLD} to obtain the approximation error \eqref{eq:uapp}.

\medskip
To begin with the analysis of \eqref{eq_eta}, recall that (see Section \ref{sec:fb}) every function in $L^2(\R)$ admits the decomposition \eqref{eq:L2dec} with respect to the family of Bloch waves given by Proposition \ref{prop-Blochsym}. Hence, for every fixed $k \in [0,2\pi]$, multiplying \eqref{eq_eta} by $\Phi_n(\cdot, k)$, $n\in I$ (with $I$ as in \eqref{eq-indices}), allows to rewrite the equation for $\eta$ as the infinite system of equations for the Bloch coefficient $\big( \widetilde \eta _n(k) \big)_{n \in I}$ of $\eta$ given by 
\begin{equation}
\label{eq_tilde_eta}
    \big ( \mu_n(k) - \mu_* \big ) \widetilde \eta_n(k) + \delta \langle W \eta, \Phi_n(\cdot,k) \rangle_{L^2 (\R)} -\delta\mu_\sharp \widetilde\eta_n(k)=  \delta \widetilde R_n(\eta)(k),
\end{equation}
with
\begin{equation}
    \label{eq_remainder}
    \widetilde R_n(\eta)(k) =  \widetilde F_n(k) + \langle L_\delta (\eta), \Phi_n(\cdot, k) \rangle_{L^2(\R)} + \langle \mathcal N_\delta (\eta), \Phi_n(\cdot, k) \rangle_{L^2(\R)}\,,
\end{equation}
where $\displaystyle \widetilde F_n(k):=\langle F(\delta,\mu_\sharp,\cdot), \Phi_n(\cdot, k) \rangle_{L^2(\R)}$.

Following the general strategy exploited in \cite{FLW17} in the linear setting, we now solve \eqref{eq_tilde_eta} dealing separately with regimes of $k$ close to and far from the Dirac point $(\pi,\mu_*)$. 
Precisely, for every $n\in I$ we define
\begin{equation}\label{eq:eta_near}
\widetilde \eta^{\mathrm{near}}_n(k) := 
\begin{cases}
    \chi (\lvert k - \pi \rvert < \delta^\tau ) \widetilde \eta_n (k) &\quad \text{ if } n = +,-\\[.1cm]
    0 &\quad \text{ otherwise};
\end{cases}
\end{equation}
\begin{equation}\label{eq:eta_far}
\widetilde \eta^{\mathrm{far}}_n(k) :=
\begin{cases}
     \chi (\lvert k - \pi \rvert \ge \delta^\tau ) \widetilde \eta_n (k) &\quad \text{ if } n=+,-\\[.1cm]
    \widetilde \eta_n(k) &\quad \text{ otherwise},
\end{cases}
\end{equation}
    for some $\tau >0$ to be chosen. Equivalently, with a shorter notation
    \[
    \begin{split}
    \widetilde \eta^{\mathrm{near}}_n(k) &= \chi \big (\lvert k-\pi\rvert < (\delta_{n,+} + \delta_{n,-} ) \delta^\tau \big ) \widetilde \eta_n(k),\\
     \widetilde \eta^{\mathrm{far}}_n(k) &= \chi \big (\lvert k-\pi\rvert \ge (\delta_{n, +} + \delta_{n, -} ) \delta^\tau \big ) \widetilde \eta_n(k),\\
    \end{split}
    \]
    where $\delta_{a,b}$ denotes the Kronecker delta with $(a,b)\in I\times I$. Clearly, this definition of near and far components for each $\widetilde{\eta}_n$ induces a decomposition on $\eta$:
    \begin{equation}
    \label{eq:deceta}
    \eta(x)= \eta^{\mathrm{near}} (x) + \eta^{\mathrm{far}} (x),
    \end{equation}
    where
    \begin{equation}
    \label{compact_not_near/far}
    \begin{split}
    \eta^{\mathrm{near}}(x) &= \frac 1{2\pi} \int_0^{2\pi} \widetilde \eta_+^{\mathrm{near}}(k) \Phi_+(x,k) dk+\frac 1{2\pi} \int_0^{2\pi} \widetilde \eta_-^{\mathrm{near}}(k) \Phi_-(x,k) dk\,,\\
     \eta^{\mathrm{far}}(x) &= \frac 1{2\pi} \sum_{n \in I} \int_0^{2\pi} \widetilde \eta_n^{\mathrm{far}}(k) \Phi_n(x,k) dk\,.
    \end{split}
      \end{equation}
   Accordingly, multiplying \eqref{eq_tilde_eta} first by $\chi_n^\mathrm{near}(k):=\chi \big (\lvert k-\pi\rvert < (\delta_{n,+} + \delta_{n,-} ) \delta^\tau \big )$ and then by $\chi_n^\mathrm{far}(k)= \chi \big (\lvert k-\pi\rvert \ge (\delta_{n, +} + \delta_{n, -} ) \delta^\tau \big )$, for every $n\in I$, we obtain the equivalent system 
   \begin{equation}
    \label{eq_near/far}
    \begin{cases}
     \big ( \mu_\pm(k) - \mu_* \big ) \widetilde \eta^{\mathrm{near}}_\pm(k) + \delta \chi_\pm^{\mathrm{near}} (k) \langle W \eta, \Phi_\pm(\cdot,k) \rangle_{L^2 (\R)} -\delta \mu_\sharp\widetilde \eta^{\mathrm{near}}_\pm(k) =  \delta \widetilde R^{\mathrm{near}}_\pm(\eta)(k) & \\[.2cm]
     \big ( \mu_n(k) - \mu_* \big ) \widetilde \eta^{\mathrm{far}}_n(k) + \delta \chi_n^{\mathrm{far}} (k) \langle W\eta, \Phi_n(\cdot,k) \rangle_{L^2 (\R)} -\delta \mu_\sharp\widetilde \eta^{\mathrm{far}}_n(k) =  \delta \widetilde R^{\mathrm{far}}_n(\eta)(k) &
         \end{cases}
    \end{equation}
   where the first line corresponds to the two equations for the only non--zero near components with $n=\pm$, and the second line corresponds to the infinitely many equations for the far components for every $n\in I$.

   We thus need to construct suitable functions $\widetilde\eta_n^{\mathrm{near}}, \widetilde\eta_n^{\mathrm{far}}$ solving \eqref{eq_near/far}. Since in the end our aim is to find a real--valued and even function $\eta$, we will look for solutions of \eqref{eq_near/far} in functional spaces with specific symmetries. Precisely, we set 
\begin{equation}
\label{eq:nfspaces}
   \begin{split}
   L^{2,\mathrm{near}}_s(\R)& := \left\{ f \in L^2(\R) \colon \widetilde f_n(k) \equiv \widetilde f_{n}^{\mathrm{near}}(k)\;\; \forall n\in I,\, \widetilde f_\pm(2\pi-k)=\overline{\widetilde f_\mp}(k)=\widetilde f_\mp(k)\right\}\\[.2cm]
   L^{2,\mathrm{far}}_s(\R) & := \left\{ f \in L^2(\R) \colon \widetilde f_n(k) \equiv \widetilde f_{n}^{\mathrm{far}}(k)\;\;\forall n\in I, \, \widetilde f_\pm(2\pi-k)=\overline{\widetilde f_\mp}(k)=\widetilde f_\mp(k)\,,\right.\\
   &\qquad\qquad\qquad\qquad\qquad\qquad\qquad\quad\left.\widetilde f_n(2\pi-k)=\overline{\widetilde f_n}(k)=\widetilde f_n(k)\;\;\forall n\in I_1\right\}
  \end{split}
\end{equation}
($I_1$ defined by \eqref{eq-indices}) and we analogously define the Sobolev spaces $H^{r,\mathrm{near}}_s(\R)$, $H^{r,\mathrm{far}}_s(\R)$, for $r>0$, endowing all these spaces with the usual inner product. 

\begin{proposition}
\label{prop:real}
    The functions in $L^{2,\mathrm{far}}_s(\R)$ and $L^{2,\mathrm{near}}_s(\R)$ are real--valued and even.
\end{proposition}
\begin{proof}
    Let $f\in L^{2,\mathrm{far}}_s(\R)$, so that by \eqref{eq:L2dec} and \eqref{eq:nfspaces}
    \[
    f(x) = \frac{1}{2\pi}\sum_{n\in I}\int_0^{2\pi} \widetilde f_{n}^{\mathrm{far}}(k)\Phi_n(x,k)\,dk.
    \]
    Then, by \eqref{eq:nfspaces} and \eqref{sym_conj} we have
    \[
    \begin{split}
        \overline{f}(x)&\,=\frac{1}{2\pi}\sum_{n\in I}\int^{2\pi}_0 \overline{\widetilde f_{n}^{\mathrm{far}}}(k)\overline{\Phi_n}(x,k)\,dk\\
        &\,=\frac{1}{2\pi}\sum_{n\in I_1}\int^{2\pi}_0 \overline{\widetilde f_{n}^{\mathrm{far}}}(k)\overline{\Phi_n}(x,k)\,dk\\
        &\quad\,+ \frac{1}{2\pi}\int^{2\pi}_0 \overline{\widetilde f_{+}^{\mathrm{far}}}(k)\overline{\Phi_+}(x,k)\,dk+ \frac{1}{2\pi}\int^{2\pi}_0 \overline{\widetilde f_{-}^{\mathrm{far}}}(k)\overline{\Phi_-}(x,k)\,dk\\
        &\,=\frac{1}{2\pi}\sum_{n\in I_1}\int^{2\pi}_0 \widetilde f_{n}^{\mathrm{far}}(2\pi-k)\Phi_n(x,2\pi-k)\,dk\\
        &\quad\,+ \frac{1}{2\pi}\int^{2\pi}_0 \widetilde f_{-}^{\mathrm{far}}(2\pi-k)\Phi_-(x,2\pi-k)\,dk+ \frac{1}{2\pi}\int^{2\pi}_0 \widetilde f_{+}^{\mathrm{far}}(2\pi-k)\Phi_+(x,2\pi-k)\,dk\\
        &\stackrel{\nu:=2\pi-k}{=}\frac{1}{2\pi}\sum_{n\in I}\int^{2\pi}_0 \widetilde f_{n}^{\mathrm{far}}(\nu)\Phi_n(x,\nu)\,d\nu = f(x)\,,
    \end{split}
    \]
    and
    \[
    \begin{split}
      f(-x)&\,=\frac{1}{2\pi}\sum_{n\in I_1}\int^{2\pi}_0 \widetilde f_{n}^{\mathrm{far}}(k) \Phi_n(-x,k)\,dk\\
      &\quad+\frac{1}{2\pi}\int^{2\pi}_0 \widetilde f_{-}^{\mathrm{far}}(k)\Phi_-(-x,k)\,dk+ \frac{1}{2\pi}\int^{2\pi}_0 \widetilde f_{+}^{\mathrm{far}}(k)\Phi_+(-x,k)\,dk\\
      &=\frac{1}{2\pi}\sum_{n\in I_1}\int^{2\pi}_0 \widetilde f_{n}^{\mathrm{far}}(2\pi-k)\Phi_n(x,2\pi-k)\,dk \\
      &\quad+\frac{1}{2\pi}\int^{2\pi}_0 \widetilde f_{+}^{\mathrm{far}}(2\pi-k)\Phi_+(x,2\pi-k)\,dk+ \frac{1}{2\pi}\int^{2\pi}_0 \widetilde f_{-}^{\mathrm{far}}(2\pi-k)\Phi_+(x,2\pi-k)\,dk\\
        &\stackrel{\nu:=2\pi-k}{=}\frac{1}{2\pi}\sum_{n\in I}\int^{2\pi}_0 \widetilde f_{n}^{\mathrm{far}}(\nu)\Phi_n(x,\nu)\,d\nu =f(x)\,,
    \end{split}
    \]
Analogous computations work for functions in $L^{2,\mathrm{near}}_s(\R)$.
\end{proof}

In view of Proposition \ref{prop:real}, if we find $\eta^\mathrm{near}\in H_s^{2,\mathrm{near}}(\R)$, $\eta^{\mathrm{far}}\in H_s^{2,\mathrm{far}}(\R)$ satisfying  \eqref{eq_near/far}, then the corresponding $\eta$ given by \eqref{eq:deceta} is a real--valued, even solution of \eqref{eq_eta} in $H^2(\R)$. Observe also that \eqref{eq_near/far} is invariant under the symmetries embodied in \eqref{eq:nfspaces}. Indeed, if $\eta^{\mathrm{near}}\in H_s^{2,\mathrm{near}}(\R)$ and $\eta^{\mathrm{far}}\in H_s^{2,\mathrm{far}}(\R)$, then by Proposition \ref{prop-Blochsym}, Proposition \ref{prop:real} and the fact that $W$ is real--valued and even by assumption, one has
\[
\begin{split}
    \overline{\left\langle W\eta,\Phi_n(\cdot,k)\right\rangle}_{L^2(\R)}&=\left\langle W\overline\eta,\overline{\Phi_n}(\cdot,k)\right\rangle_{L^2(\R)}=\left\langle W\eta,\Phi_n(\cdot,2\pi-k)\right\rangle_{L^2(\R)}\\
    \overline{\left\langle W\eta,\Phi_\pm(\cdot,k)\right\rangle}_{L^2(\R)}&=\left\langle W\overline\eta,\overline{\Phi_\pm}(\cdot,k)\right\rangle_{L^2(\R)}=\left\langle W\eta,\Phi_\mp(\cdot,2\pi-k)\right\rangle_{L^2(\R)}
\end{split}
\]
and
\[
\begin{split}
\left\langle W\eta,\Phi_n(\cdot,2\pi-k)\right\rangle_{L^2(\R)}&=\left\langle W\eta,\Phi_n(-\cdot,k)\right\rangle_{L^2(\R)} \\
&=\left\langle W(-\cdot)\eta(-\cdot),\Phi_n(-\cdot,k)\right\rangle_{L^2(\R)}=\left\langle W\eta,\Phi_n(\cdot,k)\right\rangle_{L^2(\R)},
\end{split}
\]
\[
\begin{split}
\left\langle W\eta,\Phi_\pm(\cdot,2\pi-k)\right\rangle_{L^2(\R)}&=\left\langle W\eta,\Phi_\mp(-\cdot,k)\right\rangle_{L^2(\R)} \\
&=\left\langle W(-\cdot)\eta(-\cdot),\Phi_\mp(-\cdot,k)\right\rangle_{L^2(\R)}=\left\langle W\eta,\Phi_\mp(\cdot,k)\right\rangle_{L^2(\R)},
\end{split}
\]
for every $n\in I_1$. The same argument applies to $\widetilde{R}_n(\eta)$ in \eqref{eq_near/far}, as all terms in \eqref{eq:F_delta}, \eqref{eq:L_delta} and \eqref{eq:N_delta} are real--valued and even by Remark \ref{rem:realU}. Recalling also the definition of $\mu_\pm$ given in \eqref{eq:lmeno}, \eqref{eq:lpiu}, this shows that replacing $k$ with $2\pi-k$ does not change \eqref{eq_near/far}, as it reduces to switch the equations for the Bloch coefficients $\widetilde\eta$ associated with $\Phi_\pm$. The same happens when one computes the conjugate of \eqref{eq_near/far} with $2\pi-k$ in place of $k$.

Hence, in order to prove Theorem \ref{thm:main1} we are left to construct one specific solution of \eqref{eq_near/far}, and will accomplish this exploiting the additional symmetries introduced in \eqref{eq:nfspaces}. We proceed by first focusing on the far components, solving the second line of \eqref{eq_near/far} with a suitable function $\eta^{\mathrm{far}}=\eta^{\mathrm{far}}(\eta^{\mathrm{near}},\delta)$, and then we use this solution (still dependent on $\eta^{\mathrm{near}})$ to simplify the first line of \eqref{eq_near/far}. This will in turn allow us to find a specific solution $\eta^{\mathrm{near}}$ for the near part of \eqref{eq_near/far} and thus, moving backward, to identify the corresponding $\eta^{\mathrm{far}}$ and complete the proof of Theorem \ref{thm:main1}.


   \subsection{Finding far energy components}\label{sec5.1}
   We begin by finding a solution for the second line of \eqref{eq_near/far} in the form 
   \begin{equation}\label{eq:far(near)}
   \eta^{\mathrm{far}}=  \eta^{\mathrm{far}}(\eta^{\mathrm{near}}, \delta)
   \end{equation}
   for any given $\eta^{\mathrm{near}}\in H_s^{2,\mathrm{near}}(\R)$ and for sufficiently small $\delta>0$. To this end, we define the map 
\[
H^{2,\mathrm{far}}(\R)\times H^{2,\mathrm{near}}(\R)\times\R_+\ni(\phi,\psi,\delta)\longmapsto\mathcal E (\phi, \psi, \delta),
\]
such that for every $n\in I$
  \begin{equation}\label{eq:En}
  \begin{split}
  \widetilde {\mathcal E}_n (\phi, \psi, \delta)(k):=&\,\langle \mathcal E (\phi, \psi,\delta), \Phi_n(\cdot, k) \rangle_{L^2(\R)}\\
 =&- \delta \frac{\chi_n^{\mathrm{far}}(k)}{\mu_n(k) -\mu_* } \langle W ( \phi +\psi), \Phi_n(\cdot,k) \rangle_{L^2(\R)}\\
&+ \delta \frac{\mu_\sharp}{\mu_n(k) - \mu_*} \widetilde \phi_n^{\mathrm{far}}(k) + \delta \frac{\widetilde R_n ^{\mathrm{far}}(\phi+\psi)(k)}{\mu_n(k) - \mu_*}\,.
  \end{split}
  \end{equation}
  In this way, the second line of \eqref{eq_near/far} can be written as
  \begin{equation}
  \label{fp_E}
\eta^{\mathrm{far}}=  \mathcal E(\eta^{\mathrm{far}}, \eta^{\mathrm{near}},\delta)
  \end{equation}
and by the symmetries of the system we have $\mathcal{E}(\phi,\psi,\delta)\in H_s^{2,\mathrm{far}}(\R)$ for every $\phi\in H_s^{2,\mathrm{far}}(\R)$, $\psi\in H_s^{2,\mathrm{near}}(\R)$ and $\delta>0$.

To perform a fixed point argument for $\mathcal{E}$, we first derive the next preliminary estimates. Note that, given the properties of $\Psi$ obtained in Theorem \ref{thm:main2}, the proof below follows along the same line as in \cite[Appendix G]{FLW17}. We briefly sketch it for the sake of completeness.

\begin{lemma}
\label{lem:estU}
As $\delta\to0^+$, in view of \eqref{eq:xdy}, one has
\begin{equation}\label{eq:Uest}
    \begin{split}
        \Vert U_i(\cdot,\delta\cdot)\Vert_{L^2(\R)}&\lesssim \delta^{-1/2} \\[.1cm]
        \Vert (\partial^2_y U_i)(\cdot,\delta\cdot)\Vert_{L^2(\R)}&\lesssim\delta^{-1/2} \\[.1cm]
        \Vert (\partial_x\partial_y U_i)(\cdot,\delta\cdot)\Vert_{L^2(\R)}&\lesssim\delta^{-1/2}\,
    \end{split}\qquad\text{for}\quad i=0,\,1,
\end{equation}
where $U_0$ is as in \eqref{eq:U0} and $U_1$ as in \eqref{eq:U1}.
\end{lemma}

\begin{proof}
The estimates for $U_0(\cdot,\delta\cdot)$ directly follow by the explicit formula \eqref{eq:U0} and the regularity of the Bloch waves and of the solution $\Psi$ of \eqref{NLD} given by Theorem \ref{thm:main2}. Let us then prove \eqref{eq:Uest} for $U_1(\cdot,\delta\cdot)$. To this end, observe that by \eqref{eq:U1} the function $U_1(x,y)$ satisfies an equation of the form
\begin{equation}
\label{eq:UG}
(-\partial_x^2 + V(x) - \mu_*) U_1(x,y) = G_1(x,y)
\end{equation}
with
\begin{equation}
\label{eq:G}
G_1(x,y) = \sum_{j=1}^N f_j(x) g_j(y), \quad  f_j \in L^2_{\pi}(\R)\cap C^\infty(\R),\,g_j \in \mathcal S(\R),
\end{equation}
for some $N \in \N$. Indeed, $G_1$ is given by \eqref{eq:G1} and expanding all terms therein shows that we can rewrite
 \begin{equation}
 \label{eq:sumG}
 G_1(x,y)=\sum_{j=1}^{10} f_j(x) g_j(y)
 \end{equation}
 where $f_j$ are combinations of the Bloch waves $\Phi_\pm(\cdot,\pi)$ and of their derivatives (possibly multiplied by $W(x)$ and $\mu_\sharp$), and $g_j$ of the components of the solution $\Psi(y)$ of \eqref{NLD} given by Theorem \ref{thm:main2} and of their derivatives. Now, by \cite[Appendix G]{FLW17}, the first estimate in \eqref{eq:Uest} holds for every function satisfying an equation in the form \eqref{eq:UG} with a right hand side as in \eqref{eq:G}. Hence, the first line of \eqref{eq:Uest} is proved for $U_1$. Furthermore, for any $\alpha \in \N$, $\partial_y^\alpha U_1$ solves
\[
(-\partial_x^2 + V(\cdot) -\mu_*) \partial_y^\alpha U_1 = \partial_y^\alpha G_1
\]
and $\partial_x\partial_y U_1$ satisfies
\[
(-\partial_x^2 + V(\cdot) -\mu_*) \partial_x \partial_yU_1 = \partial_x\partial_y G_1 - V'(x)\partial_yU_1.
\]
Since the right hand side of both equations can be written as in \eqref{eq:G} (observe, in particular, that $\displaystyle U_1(x,y)=\sum_{j=1}^{10}g_j(y)h_j(x)$, with $h_j\in L^2_{\pi}(\R)\cap C^\infty(\R)$, by \eqref{eq:U1}, \eqref{eq:G} and \eqref{eq:sumG}), we obtain the remaining estimates in \eqref{eq:Uest} for $U_1$ again by \cite[Appendix G]{FLW17}.
\end{proof}

In the next proposition we prove the estimates necessary to solve the far equations in the second line of \eqref{eq_near/far}, i.e. \eqref{fp_E}. In what follows we denote by $B^{\mathrm{near}}(R)$, $B^{\mathrm{far}}(R)$ the balls centered at the origin with radius $R$ in $H_s^{2,\mathrm{near}}(\R)$, $H_s^{2,\mathrm{far}}(\R)$ respectively. 

  \begin{proposition}[Lipschitz bounds for $\mathcal E$]
  \label{prop:lipbound}
  Let $\tau\in(0,\frac12)$. Then
  \begin{align}
  \label{eq:farbound}
  \lVert \mathcal E(\phi, \psi, \delta) \rVert_{H^2(\R)}^2\lesssim &\, \delta^{1-2\tau}+\delta^{2-2\tau}\left(\lVert \phi \rVert_{L^2(\R)}^2 + \lVert \psi \rVert_{L^2(\R)}^2\right)\nonumber \\[.1cm]
   &\,+\delta^{4-2\tau}\left(\lVert \phi \rVert_{H^1(\R)}^2 + \lVert \psi \rVert_{H^1(\R)}^2\right)\left(\lVert \phi \rVert_{L^2(\R)}^2 + \lVert \psi \rVert_{L^2(\R)}^2\right)\nonumber\\[.1cm]
   &\,+ \delta^{6-2\tau}\left(\lVert \phi \rVert_{H^1(\R)}^4 + \lVert \psi \rVert_{H^1(\R)}^4\right)\left(\lVert \phi \rVert_{L^2(\R)}^2 + \lVert \psi \rVert_{L^2(\R)}^2\right)
  \end{align}
  for every $(\phi,\psi,\delta)\in H_s^{2,\mathrm{far}}(\R)\times H_s^{2,\mathrm{near}}(\R)\times(0,1)$, and
  \begin{multline}
  \label{eq:Efarlip}
  \lVert \mathcal E (\phi_1, \psi_1 , \delta)\,- \mathcal E(\phi_2, \psi_2 , \delta) \rVert_{H^2(\R)}^2\lesssim \delta^{2-2\tau}\left(\lVert \phi_1-\phi_2 \rVert_{L^2(\R)}^2 + \lVert \psi_1-\psi_2 \rVert_{L^2(\R)}^2\right) \\[.1cm]
   +\delta^{4-2\tau}\left(\lVert \phi_1 \rVert_{H^1(\R)}^2+\lVert \phi_2 \rVert_{H^1(\R)}^2 + \lVert \psi_1 \rVert_{H^1(\R)}^2+ \lVert \psi_2 \rVert_{H^1(\R)}^2\right)\,\cdot\\[.1cm]
   \hspace{7cm}\cdot\,\left(\lVert \phi_1-\phi_2 \rVert_{L^2(\R)}^2 + \lVert \psi_1-\psi_2 \rVert_{L^2(\R)}^2\right)\\[.1cm]
   \,+ \delta^{6-2\tau}\left(\lVert \phi_1 \rVert_{H^1(\R)}^4+\lVert \phi_2 \rVert_{H^1(\R)}^4 + \lVert \psi_1 \rVert_{H^1(\R)}^4+ \lVert \psi_2 \rVert_{H^1(\R)}^4\right)\,\cdot\\[.1cm]
   \cdot\,\left(\lVert \phi_1-\phi_2 \rVert_{L^2(\R)}^2 + \lVert \psi_1-\psi_2 \rVert_{L^2(\R)}^2\right)
  \end{multline}
  for every $(\phi_1,\psi_1,\delta),\,(\phi_2,\psi_2,\delta)\in H_s^{2,\mathrm{far}}(\R)\times H_s^{2,\mathrm{near}}(\R)\times(0,1)$. In particular, for every $R>0$, there exists $\widetilde\delta=\widetilde\delta(\tau,R)\in(0,1)$ such that, if $\rho_\delta := \gamma\,\delta^{\frac 12 - \tau}$, with $\gamma>0$ large enough, and $\psi \in B^{\mathrm{near}}(R)$, then
    \begin{equation}
  \label{eq-ballpreserve}
  \phi \in B^{\mathrm{far}}(\rho_\delta) \implies \mathcal E (\phi,\psi,\delta) \in B^{\mathrm{far}} (\rho_\delta),
  \end{equation}
  and
      \begin{equation}\label{eq:lipbound}
      \lVert \mathcal E (\phi_1, \psi, \delta) - \mathcal E (\phi_2, \psi, \delta) \rVert_{H^2(\R)} <\lVert \phi_1 - \phi_2 \rVert_{L^2(\R)},\qquad\forall \phi_1,\, \phi_2 \in B^{\mathrm{far}}(\rho_\delta)\subseteq H_s^{2,\mathrm{far}}(\R),
      \end{equation}
      for all $\delta\in(0,\widetilde\delta)$.
  \end{proposition}
  
  \begin{proof}
 We start by estimating the $H^2(\R)$ norm of $\mathcal E$. Recalling the equivalence \eqref{eq:sobeq} between the norm induced by the Bloch decomposition and the standard Sobolev norm, suitably adapted to the Bloch waves given by Proposition \ref{prop-Blochsym}, we have
  \begin{equation}
  \label{Hs_E}
      \lVert\mathcal  E(\phi, \psi, \delta) \lVert_{H^2(\R)}^2 \sim \sum_{n\in I} \ell(n)^2 \int_0^{2\pi} \lvert \widetilde {\mathcal E}_n(\phi, \psi,\delta) (k) \rvert ^2 \,dk,
    \end{equation}
    with
    \[
    \ell(n):=\left\{\begin{array}{ll}1+n^2,&n\neq+,-,\\[.1cm]1,& n=+,-.\end{array}\right.
    \]
  Now,  by \cite[Lemma 6.2]{FLW17}, if $n \ne +,-$, then we have
  \begin{multline*}
    \big  \lvert \widetilde {\mathcal E}_n (\phi, \psi, \delta)\big \rvert^2 \lesssim \frac{\delta ^2}{(1+n^2)^2}\, \cdot\\[.1cm]
    \cdot\,\Big \{ \big \lvert \langle W (\phi +\psi), \Phi_n(\cdot,k) \rangle_{L^2(\R)} \big \rvert^2 + \mu_\sharp^2 \lvert \widetilde \phi_n^{\mathrm{far}}(k) \rvert^2 + \lvert \widetilde R_n ^{\mathrm{far}}(\phi+\psi)(k)\rvert^2 \Big \},
  \end{multline*}
  while
  \[
  \big  \lvert \widetilde {\mathcal E}_\pm (\phi, \psi, \delta)\big \rvert^2 \lesssim \delta ^{2(1-\tau)} \Big \{ \big \lvert \langle W (\phi +\psi), \Phi_\pm(\cdot,k) \rangle_{L^2(\R)} \big \rvert^2 + \mu_\sharp^2 \lvert \widetilde \phi_\pm^{\mathrm{far}}(k) \rvert^2 + \lvert \widetilde R_\pm ^{\mathrm{far}}(\phi+\psi)(k)\rvert^2 \Big \}\,.
  \]
  Therefore, from \eqref{Hs_E}
  \[
  \begin{split}
    \lVert\mathcal  E(\phi, \psi, \delta) \lVert_{H^2(\R)}^2 \lesssim &\, \delta^{2(1-\tau)}  \sum_{n\in I} \int_0^{2\pi}  \left ( \big \lvert \langle W \phi , \Phi_n(\cdot,k) \rangle_{L^2(\R)} \big \rvert^2+\right. \\[.1cm]
    &\, \left.+ \big \lvert \langle W \psi , \Phi_n(\cdot,k) \rangle_{L^2(\R)} \big \rvert^2+ \mu_\sharp^2 \lvert \widetilde \phi_n^{\mathrm{far}}(k) \rvert^2 + \lvert \widetilde R_n ^{\mathrm{far}}(\phi+\psi)(k) \rvert^2 \right) dk \\[.1cm]
    \lesssim &\, \delta^{2(1-\tau)} \left( \lVert W \phi \rVert_{L^{2}(\R)}^2 +\lVert W \psi \rVert_{L^{2}(\R)}^2 + \right.\\[.1cm]
    &\, \left.+\mu_\sharp^2 \lVert \phi \rVert_{L^{2}(\R)}^2 + \sum_{n\in I} \int_0^{2\pi} \lvert \widetilde R^{\mathrm{far}}_n(\phi+\psi) (k)\rvert ^2 dk \right).
  \end{split}
  \]
Now, from \eqref{eq_remainder} we have 
\[
\begin{split}
\sum_{n\in I}\int_0^{2\pi} \lvert \widetilde R_n^{\mathrm{far}}(\phi+\psi) (k) \rvert^2 dk & \\
& \hspace{-4cm}\lesssim \sum_{n\in I}\int_0^{2\pi} \Big (\lvert \widetilde F_n(k) \rvert^2 + \lvert \langle L_\delta (\phi+\psi), \Phi_n(\cdot, k) \rangle_{L^2(\R)}  \rvert^2 + \lvert \langle \mathcal N_\delta (\phi+\psi), \Phi_n(\cdot, k) \rangle_{L^2(\R)} \rvert^2 \Big  )dk\\[.1cm]
& \hspace{-4cm}\lesssim \lVert F (\delta, \mu_\sharp, \cdot )\rVert _{L^2(\R)}^2 + \lVert L_\delta (\phi + \psi) \rVert_{L^2(\R)}^2 + \lVert \mathcal N_\delta (\phi+\psi) \rVert_{L^2(\R)}^2\,.
\end{split}
\]
Using Lemma \ref{lem:estU}, the fact that $W \in L^\infty (\R)$ and the fact that $U_0(\cdot, \delta\cdot)$, $U_1(\cdot, \delta\cdot)$ are bounded in $L^\infty(\R)$ uniformly in $\delta$ by construction, we obtain the following bounds
\begin{equation}
\label{est_F}
\lVert F(\delta, \mu_\sharp, \cdot) \Vert_{L^2(\R)}^2 \lesssim \delta^{-1},
\end{equation}
\[
\lVert L_\delta (\phi + \psi) \rVert_{L^2(\R)}^2 \lesssim  \lVert \phi \rVert^2_{L^2(\R)} + \lVert \psi \rVert_{L^2(\R)}^2, 
\]
\begin{align}
\label{est_N}
\lVert \mathcal N_\delta (\phi + \psi) \rVert_{L^2(\R)}^2 \lesssim &\,\delta^2\|\phi+\psi\|_{L^\infty(\R)}^2\|\phi+\psi\|_{L^2(\R)}^2\nonumber\\
&+ \delta^4\lVert \phi + \psi \rVert_{L^\infty(\R)}^2 \lVert \phi + \psi \rVert_{L^2(\R)}^2 + \delta^4\lVert \phi + \psi \rVert_{L^\infty(\R)}^4 \lVert \phi + \psi \rVert_{L^2(\R)}^2\,. 
\end{align}
Finally, using Sobolev embeddings, again that $W \in L^\infty (\R)$, and suitably rearranging terms we get \eqref{eq:farbound}.

Let us, now, estimate the quantity $\lVert \mathcal E(\phi_1, \psi_1,\delta ) - \mathcal E(\phi_2, \psi_2, \delta) \rVert _{H^2(\R)} $. We report here the argument in full details for the nonlinear term only, the linear one being analogous (and simpler). By \eqref{eq:N_delta},
\[
\begin{split}
\lvert \mathcal N_\delta (\phi_1 + \psi_1) - \mathcal N_\delta (\phi_2 + \psi_2) \rvert ^2\lesssim \delta^2 \Big ( & (|U_0|+\delta |U_1|)\big \lvert \lvert \phi_1 + \psi_1 \rvert^2 - \lvert \phi_2 + \psi_2 \rvert^2 \big \rvert \\
& + \delta\big \lvert (\phi_1 + \psi_1) \lvert \phi_1 + \psi_1 \rvert ^2 - (\phi_2 + \psi_2) \lvert \phi_2 + \psi_2 \rvert ^2 \big \rvert  \Big )^2\,.
\end{split}
\]
To estimate the terms on the right hand side observe that 
    \[
    \begin{split}
       \lvert \phi_1 + \psi_1 \rvert^2 - \lvert \phi_2 + \psi_2 \rvert^2  &=  (\lvert \phi_1 + \psi_1 \rvert + \lvert \phi_2 + \psi_2 \rvert) (\lvert \phi_1 + \psi_1 \rvert - \lvert \phi_2 + \psi_2 \rvert)  \\
        & \le  (\lvert \phi_1 + \psi_1 \rvert + \lvert \phi_2 + \psi_2 \rvert)  ( \lvert \phi_1 - \phi_2 \rvert + \lvert \psi_1 - \psi_2 \rvert )\,.
    \end{split}
    \]
Moreover, since for every $\alpha \in \N$ the function $f(u) = \lvert u \rvert^{\alpha-1} u$ satisfies $\lvert f(z) - f(w) \rvert \lesssim_\alpha (\lvert z \rvert + \lvert w \rvert)^{\alpha -1} \lvert z-w \rvert$, $\forall z,w \in \C$, we also have
    \[
     \lvert (\phi_1 + \psi_1) \lvert \phi_1 + \psi_1 \rvert ^2 - (\phi_2 + \psi_2) \lvert \phi_2 + \psi_2 \rvert ^2 \big \rvert  \le (\lvert \phi_1 + \psi_1 \rvert + \lvert \phi_2 + \psi_2 \rvert )^2( \lvert \phi_1 - \phi_2 \rvert + \lvert \psi_1 - \psi_2 \rvert ).
    \]
Arguing similarly for the linear terms and combining all estimates together finally yields
\[
\begin{split}
      \lVert \mathcal E (\phi_1, \psi_1 , \delta)&\,- \mathcal E(\phi_2, \psi_2 , \delta) \rVert_{H^2(\R)}^2 \\
      &\,\lesssim \delta^{2(1-\tau)}  \Big ( \lVert W \rVert_{L^\infty(\R)}^2 \big ( \lVert \phi_1 - \phi_2 \rVert_{L^2(\R)}^2  + \lVert \psi_1 - \psi_2 \rVert_{L^2(\R)}^2 \big ) \\
    & + \mu_\sharp^2 \lVert \phi_1 - \phi_2 \rVert_{L^2(\R)}^2 + \lVert \phi_1 - \phi_2 \rVert_{L^2(\R)}^2 + \lVert \psi_1 - \psi_2 \rVert_{L^2(\R)}^2  \\
    & +  \delta^2\big ( \lVert \phi_1 + \psi_1 \rVert_{L^\infty(\R)}^2 + \lVert \phi_2 + \psi_2 \rVert_{L^\infty(\R)}^2 \big ) \big ( \lVert \phi_1 - \phi_2 \rVert_{L^2(\R)}^2 + \lVert \psi_1 - \psi_2 \rVert_{L^2(\R)}^2 \big ) \\
    & + \delta^4\big ( \lVert \phi_1 + \psi_1 \rVert_{L^\infty(\R)}^4 + \lVert \phi_2 + \psi_2 \rVert_{L^\infty(\R)}^4 \big )  \big ( \lVert \phi_1 - \phi_2 \rVert_{L^2(\R)}^2 + \lVert \psi_1 - \psi_2 \rVert_{L^2(\R)}^2 \big )\Big )\,,
\end{split}
\]
that by Sobolev embeddings completes the proof of \eqref{eq:Efarlip}. for a suitable choice of $\delta_0=\delta_0(\tau,R)$.

Finally, the proofs of \eqref{eq-ballpreserve} and \eqref{eq:lipbound} are straightforward consequences of \eqref{eq:farbound} and \eqref{eq:Efarlip}
\end{proof}

In view of Proposition \ref{prop:lipbound}, it is immediate to solve the second line of \eqref{eq_near/far} with a fixed point argument and thus define $\eta^{\mathrm{far}}$ as a function of $\eta^{\mathrm{near}}$.

\begin{corol}
\label{cor:solfar}
Let $\tau\in(0,\frac12)$. For every $\eta^{\mathrm{near}} \in H^{2,\mathrm{near}}_s(\R)$, there exists $\widetilde\delta=\widetilde\delta(\tau,\eta^{\mathrm{near}})\in(0,1)$ such that $\mathcal{E}(\cdot,\eta^{\mathrm{near}},\delta)$ is a contraction on $B^{\mathrm{far}}(\rho_\delta) \subseteq H^{2,\mathrm{far}}_s(\R)$, with $\rho_\delta$ defined in Proposition \ref{prop:lipbound}, for all $\delta\in(0,\widetilde\delta)$. In particular, for every $\eta^{\mathrm{near}} \in H^{2,\mathrm{near}}_s(\R)$ there exists a unique solution $\eta^{\mathrm{far}} = \eta^{\mathrm{far}}(\eta^{\mathrm{near}}, \delta) \in B^{\mathrm{far}}(\rho_\delta) \subseteq H^{2,\mathrm{far}}_s(\R)$ of \eqref{fp_E}, i.e. of the second line of \eqref{eq_near/far}, for all $\delta\in(0,\widetilde\delta)$.
\end{corol}

\begin{proof}
    This is a straightforward consequence of Proposition \ref{prop:lipbound}.  
\end{proof}


\subsection{The Dirac--type equation for near energy components}\label{sec5.2}

In view of the previous subsection, to complete the proof of Theorem \ref{thm:main1} we are left to find a suitable solution $\eta^{\mathrm{near}}$ to the first line of \eqref{eq_near/far}, that is to solve the two equations
    \begin{multline}
    \label{eq:rewneareq1}
        (\mu_+(k)-\mu_*)\,\widetilde\eta^{\mathrm{near}}_{+}(k)+\delta\,\chi(\vert k-\pi\vert<\delta^\tau)\langle W(\eta^{\mathrm{near}}+\eta^{\mathrm{far}}), \Phi_+(\cdot,k) \rangle_{L^2(\R)}\\-\delta\mu_\sharp\widetilde\eta^{\mathrm{near}}_{+}(k)
        =\delta\,\chi(\vert k-\pi\vert<\delta^\tau)\widetilde R_{+}(\eta^{\mathrm{near}}+\eta^{\mathrm{far}})(k)
    \end{multline}
  \begin{multline}\label{eq:rewneareq2}
        (\mu_-(k)-\mu_*)\,\widetilde\eta^{\mathrm{near}}_-(k)+\delta\,\chi(\vert k-\pi\vert<\delta^\tau)\langle W(\eta^{\mathrm{near}}+\eta^{\mathrm{far}}), \Phi_-(\cdot,k) \rangle_{L^2(\R)}\\-\delta\mu_\sharp\widetilde\eta^{\mathrm{near}}_-(k) 
        =\delta\,\chi(\vert k-\pi\vert<\delta^\tau)\widetilde R_-(\eta^{\mathrm{near}}+\eta^{\mathrm{far}})(k)\,,
    \end{multline}
Recall that our aim is to find a solution  $\eta^{\mathrm{near}}$ of the first line of \eqref{eq_near/far} in the space $H_s^{2,\mathrm{near}}(\R)$ embodying the symmetries introduced in \eqref{eq:nfspaces}, i.e. we want to construct $\widetilde\eta_\pm^{\mathrm{near}}$ solving \eqref{eq:rewneareq1}, \eqref{eq:rewneareq2} and satisfying the additional constraints
 \begin{equation}
     \label{eq:sympm}
      \begin{cases}
    \widetilde \eta_{\pm}^{\mathrm{near}}(2\pi-k)=\overline{\widetilde\eta_\mp^{\mathrm{near}}(k)} \\[.1cm]
    \widetilde \eta_{\pm}^{\mathrm{near}}(2\pi-k)=\widetilde\eta_\mp^{\mathrm{near}}(k)\,.
    \end{cases}
 \end{equation}
To do this, let us first rewrite \eqref{eq:rewneareq1}, \eqref{eq:rewneareq2} properly {\bf (note that to simplify notations we omit the superscript ``near'' as much as possible)}. Following the same strategy exploited in \cite{FLW17} in the linear setting, we look for a solution in the form 
\begin{equation}
    \label{eq:tilde_eta_pm}
\widetilde\eta_\pm(k)=:\widehat\eta_\pm \left(\frac{k-\pi}{\delta}\right)=\chi\left(\frac{\vert k-\pi\vert}{\delta}<\delta^{\tau-1}\right)\widehat\eta_\pm \left(\frac{k-\pi}{\delta}\right)\,.
\end{equation}
namely we search for solutions that are localized Fourier transforms around $k=\pi$ of another function, to be determined. Equivalently, in terms of the new variable 
\begin{equation}\label{eq:xi}
\xi:=\frac{k-\pi}{\delta}
\end{equation}
one has
\begin{equation}\label{eq:bl}
\widetilde\eta_\pm(\pi+\delta\xi)=\widehat\eta_\pm(\xi)=\chi\left(\vert\xi\vert <\delta^{\tau-1}\right)\widehat\eta_\pm \left(\xi\right)\,.
\end{equation}
Recall also that $\Vert \widehat\zeta\Vert^2_{L^{2,1}(\R)}=\Vert \zeta \Vert^2_{H^1(\R)}$, for every function $\zeta:\R\to\C$ that admits Fourier transform. Now, since \eqref{eq:rewneareq1}, \eqref{eq:rewneareq2} involve the eigenvalues $\mu_\pm(k)$ for $k$ in a small neighbourhood of $\pi$, exploiting the smooth reparametrization \eqref{eq:linband} we can write 
\begin{equation}\label{eq:bandtaylor}
\begin{split}
    \mu_\pm(k)-\mu_*=\delta\mu'_\pm(\pi)\xi+\frac{1}{2}(\delta\xi)^2\,\mu''_\pm(\xi^{\delta}_\pm) =\pm\delta\, c_\sharp\,\xi+\frac{1}{2}(\delta\xi)^2\,\mu''_\pm(\xi^{\delta}_\pm)
    \end{split}
\end{equation}
with $\xi^\delta_\pm\in\left(\min\{\pi,\pi+\delta\xi\},\max\{\pi,\pi+\delta\xi\}\right)$ and $c_\sharp$ as in \eqref{eq:c}. Substituting these expressions in \eqref{eq:rewneareq1}, \eqref{eq:rewneareq2} (and factoring out a common $\delta$) we get
\begin{equation}\label{eq:xinear1}
\begin{split}
    c_\sharp\,\xi\,\widehat\eta_+(\xi)&+\chi(\vert\xi\vert\leq \delta^{\tau-1})\langle W\eta^{\mathrm{near}},\Phi_+(\cdot,\pi+\delta\xi)\rangle_{L^2(\R)}-\mu_\sharp\widehat\eta_+(\xi) \\
&=-\chi(\vert\xi\vert\leq \delta^{\tau-1})\langle W\,\eta^{\mathrm{far}}(\eta^{\mathrm{near}},\delta),\Phi_+(\cdot,\pi+\delta\xi)\rangle_{L^2(\R)} \\
&\quad\,+\chi(\vert\xi\vert\leq \delta^{\tau-1})\widetilde R_+(\eta^{\mathrm{near}}+\eta^{\mathrm{far}}(\eta^{\mathrm{near}},\delta))(\pi+\delta\xi)-\frac{1}{2}\delta\,\mu''_+(\xi^\delta_+)\,\xi^2\,\widehat{\eta}_+(\xi)
\end{split}
\end{equation}
and
\begin{equation}\label{eq:xinear2}
    \begin{split}
     -c_\sharp\,\xi\,\widehat\eta_-(\xi)&+\chi(\vert\xi\vert\leq \delta^{\tau-1})\langle W\eta^{\mathrm{near}},\Phi_-(\cdot,\pi+\delta\xi)\rangle_{L^2(\R)} -\mu_\sharp\widehat\eta_-(\xi)\\
&=-\chi(\vert\xi\vert\leq \delta^{\tau-1})\langle W\,\eta^{\mathrm{far}}(\eta^{\mathrm{near}},\delta),\Phi_-(\cdot,\pi+\delta\xi)\rangle_{L^2(\R)} \\
&\quad\,+\chi(\vert\xi\vert\leq \delta^{\tau-1})\widetilde R_-(\eta^{\mathrm{near}}+\eta^{\mathrm{far}}(\eta^{\mathrm{near}},\delta))(\pi+\delta\xi) -\frac{1}{2}\delta\,\mu''_-(\xi^\delta_-)\,\xi^2\,\widehat{\eta}_-(\xi)\,.
    \end{split}
\end{equation}
Note that we are already using Corollary \ref{cor:solfar}, since in the two equations above $\eta^{\mathrm{far}}$ is always considered as the solution of the second line of \eqref{eq_near/far} found in Corollary \ref{cor:solfar} corresponding to the unknown function $\eta^{\mathrm{near}}$ we are looking for. We will explain why this is rigorous, for $\delta$ sufficiently small, in the proof of Proposition \ref{lem:F}.

To handle \eqref{eq:xinear1}, \eqref{eq:xinear2}, some simplifications are needed. First, adapting verbatim the computations to derive \cite[Eqs. (6.71)$\&$(6.73)]{FLW17}, we obtain
\begin{equation}\label{eq:massterm}
\chi(\vert\xi\vert\leq \delta^{\tau-1})\langle W\eta^{\mathrm{near}},\Phi_\pm(\cdot,\pi+\delta\xi)\rangle_{L^2(\R)}= \vartheta_\sharp \widehat\eta_\mp(\xi)+\sum^8_{j=1}I^j_\pm(\xi\,;\eta^{\mathrm{near}})\,,
\end{equation}
for suitable terms $I^j_\pm(\xi\,;\eta^{\mathrm{near}})$ that are linear in the components of $\eta^{\mathrm{near}}$. Moreover, we recall that
\begin{align*}
    \widetilde R _\pm(\eta^{\mathrm{near}}+\eta^{\mathrm{far}}(\eta^{\mathrm{near}},\delta))(\pi+\delta\xi) = &\, \widetilde F_\pm (\pi+\delta \xi) \\
    & \, + \langle L_\delta (\eta^{\mathrm{near}}+\eta^{\mathrm{far}}(\eta^{\mathrm{near}},\delta)), \Phi_\pm (\cdot, \pi + \delta \xi) \rangle_{L^2(\R)} \\
    & \, + \langle \mathcal N_\delta (\eta^{\mathrm{near}}+\eta^{\mathrm{far}}(\eta^{\mathrm{near}},\delta)), \Phi _\pm (\cdot, \pi + \delta \xi) \rangle_{L^2(\R)}
    \end{align*}
    with $L_\delta (\eta), \mathcal N_\delta (\eta)$ given by \eqref{eq:L_delta}, \eqref{eq:N_delta}, and, for the sake of simplicity, we split $L_\delta (\eta)$ as
    \begin{equation}\label{eq:L_split}
    L_\delta (\eta) = L^1 (\eta) + L^2_\delta (\eta)
    \end{equation}
    with 
    \begin{equation}\label{eq:L1}
    L^1(\eta)(x) := 3U_0^2(x,\delta x)  \eta(x)\,.
    \end{equation}
    Then, we obtain the following result. 
    
\begin{lemma}\label{lem:L^1_delta}
Let $\tau\in(0,\frac{1}{2}). \,$For $\delta>0$ sufficiently small
   \begin{equation}\label{eq:L1_delta_proj}
   \hspace{-.5cm} \begin{cases}
       \langle L^1 (\eta^{\mathrm{near}}), \Phi_+ (\cdot, \pi + \delta \xi) \rangle_{L^2(\R)} = 3\mathcal{F}\left(2\beta_1|\Psi_+|^2\eta_++\left(\beta_1\Psi_+^2+\beta_2\Psi_-^2\right)\eta_-\right) + I_\delta^+(\eta^{\mathrm{near}}) & \\[.1cm]
       \langle L^1 (\eta^{\mathrm{near}}), \Phi_- (\cdot, \pi + \delta \xi) \rangle_{L^2(\R)} = 3\mathcal{F}\left(\left(\beta_2\Psi_+^2+\beta_1\Psi_-^2\right)\eta_++2\beta_1|\Psi_+|^2\eta_-\right) + I_\delta^-(\eta^{\mathrm{near}}) & 
   \end{cases}
   \end{equation}
   where $\eta_\pm$ are defined by \eqref{eq:bl} and $I_\delta^\pm (\eta^{\mathrm{near}})$ are residual terms satisfying
   \begin{equation}\label{eq:I_delta_est}
   \lVert I_\delta^\pm (\eta^{\mathrm{near}}) \rVert_{L^2(\R)} \lesssim \delta^\tau\left( \lVert \widehat \eta_+ \rVert_{L^{2,1}(\R)}+\lVert \widehat \eta_-\rVert_{L^{2,1}(\R)}\right)\,.
   \end{equation}
\end{lemma}
\begin{proof}
Recall that, by Remark \ref{rem-pseudotoper} and Remark \ref{rem:regP}, we have
\[
\Phi_\pm(x,k)=e^{i k x}p_\pm(x,k)
\]
for suitable smooth functions $p_\pm(x,k)$ 1--periodic in the first variable. Smoothness allows to write
\begin{equation}
\label{eq:expandP}
\Phi_\pm(x,\pi+\delta\xi)=e^{i\pi x}e^{i\delta\xi x}\left(p_\pm(x,\pi)+\delta\xi\partial_k p_\pm(x,\widetilde k_\pm(x;\delta\xi))\right)
\end{equation}
for some $\widetilde k(x;\delta\xi)\in \left(\min\{\pi,\pi+\delta\xi\},\max\{\pi,\pi+\delta\xi\}\right)$ so that
\begin{equation}
\label{eq:lowerord}
\sup_{x\in[0,1]}|\delta\xi\partial_k p_\pm(x,\widetilde k_\pm(x;\delta\xi))|\leq\delta^\tau
\end{equation}
whenever $|\xi|\leq\delta^{\tau-1}$. By \eqref{compact_not_near/far}, \eqref{eq:tilde_eta_pm}, \eqref{eq:xi} and \eqref{eq:expandP} we then have
\begin{align}
\label{eq-dinuovonear}\eta^{\mathrm{near}}(x)&\,=\frac{1}{2\pi}\int_0^{2\pi}\widetilde\eta_+(k)\Phi_+(x,k)\,dk+\frac{1}{2\pi}\int_0^{2\pi}\widetilde\eta_-(k)\Phi_-(x,k)\,dk\\
&\,=\frac{\delta}{2\pi}\int_{\left\{|\xi|\leq\delta^{\tau-1}\right\}}\widehat{\eta}_+(\xi)e^{i\pi x}e^{i\delta\xi x}\left(p_+(x,\pi)+\delta\xi\partial_k p_+(x,\widetilde k_+(x;\delta\xi))\right)\,d\xi\nonumber\\
&\quad+\frac{\delta}{2\pi}\int_{\left\{|\xi|\leq\delta^{\tau-1}\right\}}\widehat{\eta}_-(\xi)e^{i\pi x}e^{i\delta\xi x}\left(p_-(x,\pi)+\delta\xi\partial_k p_-(x,\widetilde k_-(x;\delta\xi))\right)\,d\xi\nonumber\\
&\,=e^{i\pi x}\frac{\delta}{2\pi}\left(p_+(x,\pi)\eta_+(\delta x)+\rho_+(x,\delta x) +p_-(x,\pi)\eta_-(\delta x)+\rho_-(x,\delta x)\right)\nonumber
\end{align}
with
\[
\rho_\pm(x,\delta x):=\delta\int_{\left\{|\xi|\leq\delta^{\tau-1}\right\}}\widehat\eta_\pm(\xi)\xi e^{i\xi\delta x}\partial_k p_\pm(x;\widetilde k_\pm(x,\delta\xi))\,d\xi\,.
\]
We now use this expression for $\eta^{\mathrm{near}}$ to estimate $\langle3U_0(\cdot,\delta\cdot)^2\eta^{\mathrm{near}},\Phi_\pm(\cdot,\pi+\delta\xi)\rangle_{L^2(\R)}$. Consider first the case of $\Phi_+$, for which the previous formula yields
\begin{align}
\label{I}
  \langle  U_0^2(\cdot,\delta\cdot) \eta^{\mathrm{near}}, &\Phi_+ (\cdot, \pi + \delta \xi) \rangle_{L^2(\R)} = \\
\label{I.1}
& \frac{\delta}{2\pi} \langle p_+(\cdot, \pi)  U_0^2(\cdot, \delta \cdot) \eta_{+} (\delta \cdot), e^{i \xi \delta \cdot} p_+ (\cdot, \pi + \delta \xi) \rangle_{L^2(\R)}  \\
\label{I.2}
& + \frac{\delta}{2\pi} \langle p_-(\cdot, \pi)  U_0^2(\cdot, \delta \cdot) \eta_{-} (\delta \cdot), e^{i \xi \delta \cdot} p_+ (\cdot, \pi + \delta \xi) \rangle_{L^2(\R)}  \\
\label{I.3}
& +  \frac{\delta}{2\pi} \langle   U_0^2(\cdot, \delta \cdot) \rho_+(\cdot, \delta \cdot), e^{i \xi \delta \cdot} p_+ (\cdot, \pi + \delta \xi) \rangle_{L^2(\R)}  \\
\label{I.4}
 & +  \frac{\delta}{2\pi} \langle  U_0^2(\cdot, \delta \cdot) \rho_- (\cdot, \delta \cdot), e^{i \xi \delta \cdot} p_+ (\cdot, \pi + \delta \xi) \rangle_{L^2(\R)}\,.
 \end{align}
Since by \eqref{eq:U0} 
\[
    U^2_0(x, \delta x) = \Psi_+ (\delta x)^2 \Phi_+ (x,\pi) ^2 + \Psi_-(\delta x)^2 \Phi_-(x,\pi)^2 + 2 \Psi_+(\delta x) \Psi_- (\delta x) \Phi_+(x,\pi) \Phi_-(x,\pi)
\]
with $\Psi_\pm$ as in Theorem \ref{thm:main2}, note that setting
\begin{equation}\label{eq:W_k}
    W_{ij}(x) \coloneqq \Phi_i (x,\pi) \overline{\Phi_j}(x,\pi)\quad 
    \mbox{and} \quad \kappa_{ij} (\delta x) \coloneqq \Psi_i(\delta x) \overline{\Psi_j}(\delta x)
    \end{equation}
for $i,j \in \{ +,-\}$, one has (recall also Proposition \ref{prop-Blochsym})
\[
U^2_0(x, \delta x)=\sum_{i,j\in\left\{+,-\right\}}\kappa_{ij}(\delta x)W_{ij}(x)\,,
\]
and $W_{ij}$ are smooth 1--periodic functions and $\kappa_{ij} \in \mathcal{S}(\R)$.
Therefore, to estimate \eqref{I} we can repeat the analogous argument developed in \cite{FLW17} combining \cite[Lemma 6.5, Lemma 6.11, proof of Proposition 6.13]{FLW17}. In particular, it holds
    \begin{equation}\label{eq:51d_est}
    \lVert \chi (\lvert \xi \rvert \le \delta^{\tau -1})\eqref{I.3} \rVert_{L^2(\R)} \le C_1 \delta^\tau \lVert \widehat \eta_+ \rVert_{L^{2,1} (\R)},
    \end{equation}
    and
    \begin{equation}\label{eq:51e_est}
    \lVert \chi (\lvert \xi \rvert \le \delta^{\tau -1}) \eqref{I.4} \rVert_{L^2(\R)} \le C_2 \delta^\tau \lVert \widehat \eta_- \rVert_{L^{2,1} (\R)},
    \end{equation}
    for suitable constants $C_1, C_2 $ depending on the $L^\infty$-norm of $\Psi_\pm, \Phi_\pm$ only. 

    Let us now focus on \eqref{I.1}. By \eqref{eq:expandP} and \cite[Lemma 6.5]{FLW17} we obtain
     \begin{align*}
    \eqref{I.1} = & \sum_{i,j\in\{+,-\}}\left\{\sum_{m \in \Z} \mathcal F\left(\kappa_{ij} \eta_+\right) \Big (\frac{2 \pi m}{\delta} + \xi \Big ) \int_0 ^1 e^{2\pi i mx} \lvert p_+ (x, \pi) \rvert^2 W_{ij}(x) dx\right. \\
    & \left.+ \sum_{m \in \Z} \mathcal F\left(\kappa_{ij} \eta_+\right) \Big (\frac{2 \pi m}{\delta} + \xi \Big ) \int_0 ^1 e^{2\pi i mx} \Delta p_+ (x, \delta \xi)  W_{ij}(x) dx\right\} \\
    = & \sum_{i,j\in\{+,-\}}\mathcal F\left(\kappa_{ij} \eta_+\right) ( \xi ) \int_0 ^1 \lvert p_+ (x, \pi) \rvert^2 W_{ij}(x) dx + A_1^+,
    \end{align*}
    where we wrote $\Delta p_+ (x, \delta \xi):=\delta\xi\partial_k p_+(x,\widetilde k_+(x;\delta\xi))$. Now, by \eqref{eq:lowerord} and \cite[Lemma 6.11]{FLW17} one has
    \begin{equation}\label{eq:R^+_1_est}
    \lVert \chi (\lvert \xi \rvert \le \delta^{\tau -1}) A_1^+ \rVert_{L^2(\R)} \le \delta^\tau\lVert \widehat \eta_+ \rVert_{L^{2,1} (\R)}.
    \end{equation}
    Moreover, by Proposition \ref{prop-Blochsym}, \eqref{int_phi_j}, \eqref{eq:W_k}, and recalling the definition \eqref{eq:beta1}, \eqref{eq:beta2} of $\beta_1$ and $\beta_2$, we see that
    \[
    \begin{split}
    &\int_0^1|p_+(x,\pi)|^2W_{++}(x)\,dx=\beta_1\,,\qquad\int_0^1|p_+(x,\pi)|^2W_{--}(x)\,dx=\beta_1\,,\\
    &\int_0^1|p_+(x,\pi)|^2W_{+-}(x)\,dx=\int_0^1|p_+(x,\pi)|^2W_{-+}(x)\,dx=0\,,
    \end{split}
    \]
    so that, again by \eqref{eq:W_k}, 
    \begin{equation}
    \label{eq:p+1}
    \begin{split}
    \sum_{i,j\in\{+,-\}}\mathcal F\left(\kappa_{ij} \eta_+\right)( \xi ) &\int_0 ^1 \lvert p_+ (x, \pi) \rvert^2 W_{ij}(x)  dx\\
    &\,=\beta_1\mathcal F\left(|\Psi_+|^2\eta_+\right)(\xi)+\beta_1\mathcal F\left(|\Psi_-|^2\eta_+\right)(\xi)\,.
    \end{split}
    \end{equation}
    In the very same way, we also have
     \[
     \eqref{I.2} =\sum_{i,j\in\{+,-\}}\mathcal F\left(\kappa_{ij} \eta_- \right)(\xi) \int_0^1 \overline{p_+}(x,\pi) p_-(x,\pi) W_{ij}(x) dx + A_2^+,
     \]
     where the last term satisfies
     \begin{equation}\label{eq:R^+_2_est}
     \lVert \chi (\lvert \xi \rvert \le \delta^{\tau -1}) A_2^+ \rVert_{L^2(\R)} \le \delta ^\tau \lVert \widehat \eta_- \rVert_{L^{2,1} (\R)}
     \end{equation}
    and 
    \[
    \begin{split}
    &\int_0^1 \overline{p_+}(x,\pi) p_-(x,\pi) W_{++}(x)\,dx=\int_0^1\overline{p_+}(x,\pi) p_-(x,\pi)W_{--}(x)\,dx=0\,,\\
    &\int_0^1\overline{p_+}(x,\pi) p_-(x,\pi)W_{+-}(x)\,dx=\beta_1\,,\qquad\int_0^1\overline{p_+}(x,\pi) p_-(x,\pi)W_{-+}(x)\,dx=\beta_2\,,
    \end{split}
    \]
    so that
    \begin{equation}
        \label{eq:p+2}
        \begin{split}
            \sum_{i,j\in\{+,-\}}\mathcal F\left(\kappa_{ij} \eta_- \right)(\xi) &\int_0^1 \overline{p_+}(x,\pi) p_-(x,\pi) W_{ij}(x) dx \\
           &\, = \beta_1\mathcal F\left(\Psi_+\overline{\Psi_-}\eta_-\right)(\xi)+\beta_2\mathcal{F}\left(\Psi_-\overline{\Psi_+}\eta_-\right)(\xi).
        \end{split}
    \end{equation}
Recalling the relation between $\Psi_\pm$ given by Theorem \ref{thm:main2} and combining \eqref{eq:51d_est}, \eqref{eq:51e_est}, \eqref{eq:R^+_1_est}, \eqref{eq:p+1}, \eqref{eq:R^+_2_est} and \eqref{eq:p+2}, we end up with 
\[
\begin{split}
    \langle  U_0^2(\cdot,\delta\cdot) \eta^{\mathrm{near}}, &\,\Phi_+ (\cdot, \pi + \delta \xi) \rangle_{L^2(\R)} \\
    &\,=\mathcal{F}\left(2\beta_1|\Psi_+|^2\eta_++\left(\beta_1\Psi_+^2+\beta_2\Psi_-^2\right)\eta_-\right)+I_\delta^+(\eta^{\mathrm{near}})
\end{split}
\]
with $\|I_+^\delta(\eta^{\mathrm{near}})\|_{L^2(\R)}\lesssim\delta^\tau\|\widehat\eta^{\mathrm{near}}\|_{L^{2,1}(\R)}$. This proves the first line of \eqref{eq:L1_delta_proj}. Since the second one follows analogously, we conclude.
\end{proof}

In view of \eqref{eq:massterm} and Lemma \ref{lem:L^1_delta}, if we set
\begin{equation}\label{eq:zeta}
\widehat\zeta(\xi):=\begin{pmatrix}
    \widehat\eta_-(\xi) \\
    \widehat\eta_+(\xi)
\end{pmatrix}
\end{equation}
the system of equations \eqref{eq:xinear1}, \eqref{eq:xinear2} can be written in the form
\begin{equation}\label{eq:diracnear}
       ( \widehat\D^\delta(\xi) +\LL^\delta(\xi))\widehat\zeta(\xi)= \RR^\delta(\widehat\zeta)(\xi)\,,
\end{equation}
where
\begin{equation}\label{eq:Ddelta}
        \widehat\D^\delta(\xi)\widehat\zeta(\xi):=ic_\sharp\,\sigma_3\,\widehat{\zeta'}(\xi) +\vartheta_\sharp\, \sigma_1 \widehat \zeta(\xi)\, - \mu_\sharp\widehat\zeta(\xi)- \chi(\lvert \xi \rvert \le \delta^{\tau -1})  L(\widehat \zeta)(\xi)
\end{equation}
with
\begin{equation}
\label{eq:defLzeta}
\mathcal F^{-1}\left(L(\widehat\zeta)\right)(x):=\begin{pmatrix}
  6\beta_1|\Psi_+(x)|^2 & 3\left(\beta_2\Psi_+^2(x)+\beta_1\Psi_-^2(x)\right)\\
  3\left(\beta_1\Psi_+^2(x)+\beta_2\Psi_-^2(x)\right) & 6\beta_1|\Psi_+(x)|^2\,,
\end{pmatrix}\zeta(x)\,,
\end{equation}
and
\begin{equation}\label{eq:Ldelta}
    \begin{split}
        \LL^\delta(\xi)\widehat\zeta(\xi):= &\, \frac{1}{2}\delta \chi(\vert\xi\vert\leq \delta^{\tau-1})
        \begin{pmatrix}
            \mu''_-(\xi^\delta_-) \\
            \mu''_+(\xi^\delta_+)
        \end{pmatrix}
        \xi^2\widehat\zeta(\xi)\\
        &+ \chi(\vert\xi\vert\leq \delta^{\tau-1})\sum^8_{j=1}
        \begin{pmatrix}
            I^j_-(\xi\,;\eta^{\mathrm{near}})\\
            I^j_+(\xi\,;\eta^{\mathrm{near}})
        \end{pmatrix}
         - \begin{pmatrix}
            I^-_\delta(\xi\,;\eta^{\mathrm{near}})\\
            I^+_\delta(\xi\,;\eta^{\mathrm{near}})
        \end{pmatrix}\\
        &- \chi(\vert\xi\vert\leq \delta^{\tau-1})\begin{pmatrix}
        \langle L^2_\delta(\zeta,\Phi_-(,\pi+\delta\xi)\rangle_{L^2(\R)}\\
        \langle L^2_\delta(\zeta,\Phi_+(,\pi+\delta\xi)\rangle_{L^2(\R)}
        \end{pmatrix} 
        \end{split}
    \end{equation}
    with $I_\delta^\pm(\widehat\eta), L^2_\delta(\zeta)$ as in \eqref{eq:L1_delta_proj} and \eqref{eq:L_split}, respectively, and
    \begin{equation}\label{eq:Rdelta}
    \begin{split}
        \RR^\delta(\widehat\zeta)(\xi):=&\,-\chi(\vert\xi\vert\leq \delta^{\tau-1})
        \begin{pmatrix}
        \langle W\eta^{\mathrm{far}}(\eta^{\mathrm{near}},\delta),\Phi_-(\cdot,\pi+\delta\xi)\rangle_{L^2(\R)}\\
        \langle W\eta^{\mathrm{far}}(\eta^{\mathrm{near}},\delta),\Phi_+(\cdot,\pi+\delta\xi)\rangle_{L^2(\R)}
        \end{pmatrix} 
        \\[.2cm]
        &\,+\chi(\vert\xi\vert\leq \delta^{\tau-1})
        \begin{pmatrix}
        \langle L_\delta(\eta^{\mathrm{far}}(\eta^{\mathrm{near}},\delta)),\Phi_-(\cdot,\pi+\delta\xi)\rangle_{L^2(\R)}\\
        \langle L_\delta(\eta^{\mathrm{far}}(\eta^{\mathrm{near}},\delta)),\Phi_+(\cdot,\pi+\delta\xi)\rangle_{L^2(\R)}
        \end{pmatrix} 
        \\[.2cm]
        &\,+\chi(\vert\xi\vert\leq \delta^{\tau-1})
        \begin{pmatrix}
            \widetilde F_- (\delta \xi + \pi)\\
            \widetilde F_+ (\delta \xi + \pi)
        \end{pmatrix}
        \\[.2cm]
        &\,+\chi(\vert\xi\vert\leq \delta^{\tau-1})
        \begin{pmatrix}
            \langle \mathcal N_\delta (\eta), \Phi _- (\cdot, \pi + \delta \xi) \rangle_{L^2(\R)} \\
            \langle \mathcal N_\delta (\eta), \Phi _+ (\cdot, \pi + \delta \xi) \rangle_{L^2(\R)}
        \end{pmatrix}.
        \end{split}
    \end{equation}
Equivalently, \eqref{eq:diracnear} can be further rewritten as
\begin{equation}
\label{eq:diracnear_2}
\left( \widehat\D_0(\xi)+(\widehat\D^\delta(\xi)-\widehat\D_0(\xi) )+\LL^\delta(\xi)\right)\widehat\zeta(\xi)= \RR^\delta(\widehat\zeta)(\xi)\,,
\end{equation}
where 
\begin{equation}\label{eq:D}
    \widehat\D_0(\xi)\widehat\zeta(\xi)= ic_\sharp\,\sigma_3\,\widehat{\zeta'}(\xi) +\vartheta_\sharp\, \sigma_1\widehat \zeta(\xi)-\mu_\sharp\widehat\zeta(\xi) \, -   L(\widehat \zeta)(\xi) 
\end{equation} 
and
\[
(\widehat\D^\delta(\xi)-\widehat\D_0(\xi))\widehat\zeta(\xi)= \chi(\lvert \xi \rvert > \delta^{\tau -1})  L(\widehat \zeta)(\xi)\,.
\]

\begin{remark}
\label{rem:local}
Recall that, by \eqref{eq:bl}, we look for a band--limited $\zeta$ solving \eqref{eq:diracnear}, that is a solution in the form
\[
\widehat\zeta(\xi)=\chi(\vert \xi\vert \leq\delta^{\tau-1})\widehat\zeta(\xi)\,.
\]
On the other hand, we rewrote \eqref{eq:diracnear} as in \eqref{eq:diracnear_2} in terms of the operator $\widehat\D_0(\xi)$, which does not preserve band--limited functions. However, since \eqref{eq:diracnear} and \eqref{eq:diracnear_2} are equivalent, $\widehat\zeta\in L^{2,1}(\R)$ is a solution to the latter equation if and only if it is a solution to the former. Hence, we can apply the projector $\chi(\vert\xi\vert> \delta^{\tau-1})$ to the former equation, obtaining that $\widehat\zeta$ also satisfies
\[
\chi(\vert\xi\vert> \delta^{\tau-1})\left(-c_\sharp\sigma_3\xi\widehat\zeta(\xi)+\theta_\sharp\sigma_1\widehat\zeta(\xi)-\mu_\sharp\widehat\zeta(\xi)\right)=0\,,
\]
so that $\widehat\zeta$ is indeed supported on $\{\vert\xi\vert\leq\delta^{\tau-1}\}$ (since $\mu_\sharp$ is in the spectral gap of the operator $\D$ defined by \eqref{eq-diracop}).
\end{remark}
The strategy to perform our desired fixed point argument is now the following:
\begin{enumerate}[label=\alph*)]
    \item we show that $\widehat{\mathcal D}_0$ is invertible;

    \item this allows to rewrite \eqref{eq:diracnear_2} in the form
\[
(\mathbb{I} + C^\delta (\xi)) \hat \zeta (\xi) = \left(\widehat\D_0(\xi)^{-1} \RR^\delta(\widehat\zeta)\right)(\xi)
\]
with $C^\delta(\xi):=\widehat{\mathcal D}_0(\xi)^{-1}((\widehat\D^\delta(\xi)-\widehat\D_0(\xi) )+\LL^\delta(\xi))$, and prove that $\mathbb{I}+C^\delta$ is invertible too for $\delta$ small enough;

    \item we conclude by a contraction argument for the map $(\mathbb{I}+C^\delta)^{-1}\widehat\D^{-1}\RR^\delta(\cdot)$.
\end{enumerate}
These three steps are accomplished in the next subsections.
   

\subsection{Inversion of \texorpdfstring{$\widehat\D_0$}{D} through a symmetry reduction}
\label{sec:near}

To prove that $\widehat\D_0$ can be inverted it is crucial to exploit the symmetries \eqref{eq:sympm} we require on $\widetilde\eta_\pm$, that in terms of the variable $\widehat\zeta$ defined by \eqref{eq:zeta} read
\begin{equation}\label{eq:xisym}
\begin{cases}
\overline{\widehat{\zeta}(\xi)}=\sigma_1 \widehat{\zeta}(-\xi)\\
\widehat{\zeta}(\xi)=\sigma_1 \widehat{\zeta}(-\xi)\,.
\end{cases}
\end{equation}
Observe that, in physical space, the first line of \eqref{eq:xisym} gives
\[
\overline{\zeta(x)}=\int^{+\infty}_{-\infty}\overline{\widehat{\zeta}(\xi)}e^{i(-\xi) x}\,d\xi =\int^{+\infty}_{-\infty}\sigma_1 \widehat{\zeta}(-\xi)e^{i(-\xi)x}\,d\xi\stackrel{\alpha=-\xi}{=}\int^{+\infty}_{-\infty}\sigma_1\widehat{\zeta}(\alpha)e^{i\alpha x}\,d\alpha\,,
\]
that is
\begin{equation}\label{eq:xsym1}
 \overline{\zeta(x)}=\sigma_1\zeta(x)\,,\qquad x\in\R\,.
\end{equation}
Analogously, the second line of \eqref{eq:xisym} corresponds to
\begin{equation}
\begin{split}
\zeta(x)&=\int^{+\infty}_{-\infty}\widehat\zeta(\xi) e^{i\xi x}\,d\xi=\sigma_1\int^{+\infty}_{-\infty}\widehat\zeta(-\xi)e^{i\xi x}\,d\xi \\
&=\sigma_1\int^{+\infty}_{-\infty}\widehat\zeta(-\xi)e^{i(-\xi) (-x)}\,d\xi\stackrel{\alpha=-\xi}{=}\sigma_1\int^{+\infty}_{-\infty}\widehat\zeta(\alpha)e^{i\alpha (-x)}\,d\alpha \,,
\end{split}
\end{equation}
that is
\begin{equation}\label{eq:xsym2}
\zeta(x)=\sigma_1\zeta(-x)\,.
\end{equation}
Hence, to prove that $\widehat\D_0$ is invertible when restricted to the subset of functions in $L^{2,1}(\R)$ satisfying \eqref{eq:xisym} is equivalent to show that $\D_0$, which can be immediately deduced by \eqref{eq:D}, is invertible when restricted to the subset of functions in $H^1(\R)$ satisfying \eqref{eq:xsym1}, \eqref{eq:xsym2}. The invertibility of the latter is guaranteed by the next proposition.
\begin{proposition}
\label{prop:kernelD}
Let $Y=\{\zeta\in H^1(\R)\,:\,\overline\zeta=\sigma_1\zeta\,,\,\zeta(\cdot)=\sigma_1\zeta(-\cdot)\}$. Then, $\ker_{\vert_Y}\D_0=\left\{0\right\}$.
\end{proposition}
\begin{proof}
    Observe first that the set of solutions in $H^1(\R)$ of
    \begin{equation}\label{eq:Dzero}
\begin{cases}
    \D_0\zeta=0 & \\
    \overline\zeta=\sigma_1\zeta &
\end{cases}
\end{equation}
is given by $\Span\{\Psi'\}$, where $\Psi$ is the solution of \eqref{NLD} given by Theorem \ref{thm:main2} used in the definition \eqref{eq:U0} of the function $U_0$. In order to see this, note first that, by Theorem \ref{thm:main2}, $\overline{\Psi'}=\sigma_1\Psi'$; while the fact that $\D_0\Psi'=0$ is exactly the content of Remark \ref{rem:eqPsi'}. However, since the first line of \eqref{eq:Dzero} defines a $2\times 2$ ODE system with complex coefficients, a priori there may exist another solution in $H^1(\R)$, say $\varphi:=(\varphi_-,\varphi_+)^T$, linearly independent with respect to $\Psi'$. Defining the Wronskian 
\[
W(\Psi',\varphi):=\Psi_-'\varphi_+-\Psi_+'\varphi_-\,,
\]
we see that the second line of \eqref{eq:Dzero} entails $W(\Psi'(x),\varphi(x))\in i\R$ for every $x\in\R$. Moreover, using that both $\Psi'$ and $\varphi$ satisfy \eqref{eq:Dzero} we obtain
\[
    \frac{d}{dx}W(\Psi'(x),\varphi(x))=\Psi_-''(x)\varphi_+(x)+\Psi_-'(x)\varphi_+'(x)-\Psi_+''(x)\varphi_-(x)-\Psi_+'(x)\varphi_-'(x)=0\qquad\forall x\in\R\,,
\]
so that 
\begin{equation}\label{eq:W}
W(\Psi'(x),\varphi(x))\equiv ic\,,\quad \mbox{for some $c\in\R$.}
\end{equation}
On the one hand, there holds $c\neq0$, since the Wronskian cannot vanish as the two solutions $\Psi',\,\varphi$ must be linearly independent. On the other hand, since $\Psi', \varphi\in H^1(\R)$ implies that both functions tend to zero as $|x|\to+\infty$, we end up with
\[
0=\lim_{x\to\pm\infty}W(\Psi'(x),\varphi(x))=ic\neq0\,.
\]
Hence, there is no $H^1$--solution of \eqref{eq:Dzero} that is linearly independent with respect to $\Psi'$.

To conclude the proof it is then enough to note that $\ker_{\vert_Y}\D_0\subseteq\Span\{\Psi'\}$, but $\Psi'\not\in\ker_{\vert_Y}\D_0$, since $\Psi'(x)=-\sigma_1\Psi'(-x)$ by \eqref{eq:psisimm}.
\end{proof}


\subsection{Inversion of \texorpdfstring{$\mathbb{I}+C^\delta$}{I+C}}\label{sec5.4}

In view of Proposition \ref{prop:kernelD}, we rewrite \eqref{eq:diracnear_2} as 
\begin{equation}
\label{eq:D_perp_2}
(\mathbb{I} + C^\delta (\xi)) \hat \zeta (\xi) = \left(\widehat\D_0(\xi)^{-1} \RR^\delta(\widehat\zeta)\right)(\xi)\,
\end{equation}
where
\begin{equation}
\label{eq:C_delta}
C^\delta(\xi)\widehat\zeta(\xi):=\widehat\D_0(\xi)^{-1}  \left( (\widehat\D^\delta(\xi)-\widehat\D_0(\xi)) +\LL^\delta(\xi)\right)\widehat\zeta(\xi)\,.
\end{equation}
We now prove that, for $\delta$ small enough, the operator $\mathbb{I} + C^\delta$ is invertible 

\begin{proposition}\label{prop:Cbound}
   Let $\tau\in(0,\frac{1}{2})$. There exists $\delta^\dagger>0$ such that, for every $\delta\in(0,\delta^\dagger)$, there holds
    \begin{equation}\label{eq:Cbound}
    \Vert C^\delta\Vert_{L^{2,1}(\R)\to L^{2,1}(\R)}\lesssim\delta^\tau\,. 
    \end{equation}
    Thus the operator $\mathbb{I}+C^\delta$ admits a bounded inverse on $L^{2,1}(\R)$.
\end{proposition}
\begin{proof}
For $\delta$ small enough, we start with the following estimates for the first two terms in $C^\delta$
    \begin{equation}\label{eq:C1est}
        \Vert(\widehat\D^\delta-\widehat\D_0)\widehat\zeta\Vert_{L^2(\R)}\lesssim\delta^{1-\tau}\Vert \widehat\zeta\Vert_{L^{2,1}(\R)}\,,
    \end{equation}
    \begin{equation}\label{eq:C2est}
        \Vert \LL^\delta \widehat\zeta\Vert_{L^2(\R)}\lesssim\delta^\tau \Vert \widehat\zeta\Vert_{L^{2,1(\R)}}\,.
    \end{equation}
    The latter inequality follows arguing verbatim as in the proof of \cite[Proposition 6.13]{FLW17} (also in view of \eqref{eq:I_delta_est}). As for the estimate in \eqref{eq:C1est}, it is enough to observe that
    \[
    \begin{split}
     \Vert(\widehat\D^\delta-\widehat\D)\widehat\zeta\Vert^2_{L^2(\R)}
     &=\Vert\chi(\vert\xi\vert>\delta^{\tau-1}) L(\widehat\zeta)\Vert^2_{L^2(\R)} \lesssim \int_{\{\vert\xi\vert>\delta^{\tau-1}\}}\vert \xi\vert^{-2}\vert\xi\vert^2\vert L(\widehat\zeta)\vert^2\,d\xi   \\
     &\leq\delta^{2(1-\tau)}\int_{\R}\lvert \xi L(\widehat\zeta) \rvert ^2 \,d\xi= \delta^{2(1-\tau)}\int_{\R}\left| \partial_x \mathcal{F}^{-1}\left(L(\widehat\zeta)\right)(x)\right| ^2 \,dx\\
     &\lesssim \delta^{2(1-\tau)} \lVert \zeta \rVert_{H^1(\R)} \simeq \delta^{2(1-\tau)}\Vert \widehat\zeta\Vert^2_{L^{2,1}(\R)}\,,
       \end{split}
    \]
    where in the last line we use the definition \eqref{eq:defLzeta} of $L(\widehat\zeta)$ and the fact that all the terms in the matrix of \eqref{eq:defLzeta} belong to $\mathcal{S}(\R)$.

    By \eqref{eq:C1est} and \eqref{eq:C2est}, the desired estimate \eqref{eq:Cbound} will follow if we show that $\widehat{\D}_0^{-1}$ is bounded from $L^{2}(\R)$ to $L^{2,1}(\R)$. To this end, we study the operator $\D_0^{-1}$ in physical space and exploit the equivalence $H^1(\R)\simeq L^{2,1}(\R)$. We start observing that the resolvent $\D_0^{-1}$ is not only bounded as an operator from $L^2(\R)$ to $L^2(\R)$, but can be seen as a bounded operator from $L^2(\R)$ to the operator domain $\operatorname{dom}(\D_0)\subseteq H^1(\R)$ endowed with the \emph{graph norm}
    \[
    \Vert \psi\Vert^2_{\D_0}:=\Vert \D_0\psi\Vert^2_{L^2(\R)}+\Vert \psi\Vert^2_{L^2(\R)}\,,\qquad \psi\in H^1(\R)\,.
    \]
    To see this, it suffices to observe that
\[
\begin{split}
\Vert(\D_0^{-1})\psi\Vert^2_{\D_0}&=\Vert\D_0(\D_0^{-1} )\psi\Vert^2_{L^2(\R)}+\Vert(\D_0^{-1})\psi\Vert^2_{L^2(\R)} =\Vert  \psi\Vert^2_{L^2(\R)}+\Vert(\D_0^{-1})\psi\Vert^2_{L^2(\R)}\leq C\Vert \psi\Vert^2_{L^2(\R)}
    \end{split}
    \]
for some constant $C>0$, thanks to the $L^2$--boundedness of the resolvent.

Let us now prove that the graph norm of $\D_0$ is equivalent to the standard $H^1$ norm, arguing as follows. Take $\psi\in H^1(\R)$ and consider the quantity
\begin{equation}\label{eq:modDnorm}
    \Vert \D_0\psi\Vert^2_{L^2(\R)}+ M \Vert \psi\Vert^2_{L^2(\R)}\,,
\end{equation}
where $M>0$ is a constant to be chosen later. Using \eqref{eq:Ddelta} to compute the first term in \eqref{eq:modDnorm} we have
\begin{equation}\label{eq:D1s}
\begin{split}
    \Vert\D_0\psi\Vert^2_{L^2(\R)}&=\,c_\sharp^2\|\partial_x\psi\|_{L^2(\R)}^2+\left(\vartheta_\sharp^2+\mu_\sharp^2\right)\|\psi\|_{L^2(\R)}^2 \\[.1cm]
    &\,+2c_\sharp(\theta_\sharp-\mu_\sharp)\Re\left(i\left\langle\sigma_3\partial_x\psi,\psi\right\rangle_{L^2(\R)}\right)-2c_\sharp\Re\left(i\left\langle\sigma_3\partial_x\psi,\FF^{-1}\left(L(\widehat\psi)\right)\right\rangle_{L^2(\R)}\right)\\[.1cm]
    &\,-2\vartheta_\sharp\mu_\sharp\Re\left(\left\langle\sigma_1\psi,\psi\right\rangle_{L^2(\R)}\right)-2\vartheta_\sharp\Re\left(\left\langle\sigma_1\psi,\FF^{-1}\left(L(\widehat\psi)\right)\right\rangle_{L^2(\R)}\right)\\[.1cm]
    &\,-2\mu_\sharp\Re\left(\left\langle\psi,\FF^{-1}\left(L(\widehat\psi)\right)\right\rangle_{L^2(\R)}\right)+\left\|\FF^{-1}\left(L(\widehat\psi)\right)\right\|_{L^2(\R)}^2.
\end{split}
\end{equation}
The terms on the second line in \eqref{eq:D1s} are of the schematic form $\Im\left(\left\langle\partial_x\psi,\psi\right\rangle_{L^2(\R)}\right)$ up to, constants, smooth and bounded matrix--valued functions, so that 
\[
\begin{split}
    \text{the second line of \eqref{eq:D1s}}\:\geq -C_1\int_\R\vert\partial_x\psi\vert\,\vert\psi\vert\,dx\geq -\frac{C_1\varepsilon}{2}\int_\R\vert \partial_x\psi\vert^2\,dx -\frac{C_1}{2\varepsilon}\int_\R\vert\psi\vert^2\,dx
\end{split}
\]
for a suitable $C_1>0$ and for every $\varepsilon>0$, recalling the inequality $ab=(\varepsilon^{1/2}a)(\varepsilon^{-1/2}b)\leq \frac{\varepsilon}{2}a^2+\frac{1}{2\varepsilon}b^2$. Moreover, the terms on the last two lines of \eqref{eq:D1s} can be bounded from below by
\[
-C_2\int_\R\vert\psi\vert^2\,dx
\]
for some other constant $C_2>0$. Combining the previous estimates one obtains
\[
\Vert \D\psi\Vert^2_{L^2(\R)}+ M \Vert \psi\Vert^2_{L^2(\R)} \geq \left(c_\sharp^2-\frac{C_1}{2}\varepsilon\right)\int_{\R}\vert\partial_x\psi\vert^2\,dx+\left(M+\vartheta_\sharp^2+\mu_\sharp^2-C_2-\frac{C_1}{2\varepsilon}\right)\int_\R\vert\psi\vert^2\,dx\,.
\]
Fixing $\varepsilon\leq c_\sharp^2/C_1$ and
\[
M\geq C_2+\frac{C_1}{2\varepsilon}-\frac{\vartheta_\sharp^2}{2}-\mu_\sharp^2\,,
\]
we find
\[
\Vert \D\psi\Vert^2_{L^2(\R)}+ M \Vert \psi\Vert^2_{L^2(\R)} \geq \frac{1}{2}\Vert \psi\Vert^2_{H^1(\R)}\,.
\]
Conversely, given the smoothness and boundedness of the coefficients in \eqref{eq:D1s}, it is not hard to see that an estimate in the other way around holds, namely
\[
\Vert \D\psi\Vert_{L^2(\R)}^2+M\|\psi\|_{L^2(\R)}^2 \leq B \Vert\psi\Vert_{H^1(\R)}^2\,,
\]
for a suitable constant $B>0$ Together with the previous estimates, this yields the equivalence between the two norms. In view of the already proved boundedness in the graph norm, we can conclude.
\end{proof}


\subsection{Conclusion of the fixed point argument} \label{sec5.5}

Relying on Proposition \ref{prop:Cbound}, we rewrite \eqref{eq:D_perp_2} in the form
\begin{equation}\label{eq:fpzeta}
\widehat\zeta=A^\delta(\widehat\zeta)\,,
\end{equation}
with 
\begin{equation}\label{eq:F}
A^\delta(\widehat\zeta):=\left((\mathbb{I}+C^\delta)^{-1}\widehat\D_0^{-1}\right)\RR^\delta(\widehat\zeta)\,,
\end{equation}
and we complete a fixed point argument for $A^\delta$ in the Fourier analogue of the space $Y$ of Proposition \ref{prop:kernelD}, namely
\begin{equation}\label{eq:Y}
\widehat Y:=\left\{\widehat\psi\in L^{2,1}(\R)\,:\, \overline{\widehat{\zeta}(\xi)}=\sigma_1 \widehat{\zeta}(-\xi)\,,\,
\widehat{\zeta}(\xi)=\sigma_1 \widehat{\zeta}(-\xi) \right\}\,.
\end{equation}

\begin{proposition}
    \label{lem:F}
There exists $\delta_0 >0$ and a map
\[
(0,\delta_0) \ni \delta\mapsto \widehat \zeta^\delta \in \widehat Y
\]
such that $\widehat \zeta^\delta$ is a solution of \eqref{eq:fpzeta}. Moreover, it satisfies the bound 
\begin{equation}
\label{eq:boundZeta}
\lVert \widehat \zeta^\delta \rVert_{L^{2,1}(\R) } \lesssim \delta^{-1}.
\end{equation}
\end{proposition}

\begin{proof}
Our goal is to prove that $A^\delta$ is a contraction on a suitable ball $B(M\delta^{-1})\subset \widehat Y$ of radius $M\delta^{-1}$ with $\delta$ sufficiently small and $M$ to be suitably chosen. Since the equation of $\widehat\zeta$ involves $\eta^{\mathrm{far}}$, preliminarily we have to guarantee that also in this setting it is actually a function of $\eta^{\mathrm{near}}$ and $\delta$ (as used up to this point), with $\eta^{\mathrm{near}}$ connected to $\widehat\zeta$ through \eqref{eq:bl}, \eqref{eq:zeta} and \eqref{eq-dinuovonear}. In other words, we have to guarantee that Corollary \ref{cor:solfar} and the last part of Proposition \ref{prop:lipbound} are still valid here. However, arguing as in \cite[Lemma 6.9]{FLW17}, we get
\begin{equation}\label{eq:nears}
\Vert\eta^{\mathrm{near}}\Vert_{H^2(\R)}\lesssim \delta^{1/2}\Vert \zeta \Vert_{L^2(\R)}\lesssim M\delta^{-1/2},\qquad\forall \widehat\zeta\in B(M\delta^{-1}).
\end{equation}
Thus, properly using \eqref{eq:farbound} and \eqref{eq:Efarlip}, Corollary \ref{cor:solfar} and the last part of Proposition \ref{prop:lipbound} hold replacing $R$ with $M\delta^{-1}$, whenever (for instance) $\gamma\simeq\sqrt{2+M^2}$ in the definition of $\rho_\delta$. Clearly, the result holds for $\delta>0$ sufficiently small depending on $M$.

Now, we can focus on the proof of the fact that
\begin{equation}\label{eq:ball}
A^\delta(B(M\delta^{-1}))\subseteq B(M\delta^{-1}).
\end{equation}
Note that, since by Proposition \ref{prop:Cbound} the operator $(\mathbb{I} + C^\delta)^{-1}$ is bounded from $L^{2,1}(\R)$ to itself and by Proposition \ref{prop:kernelD} the operator $\widehat \D_0^{-1}$ is bounded from $L^2(\R)$ to $L^{2,1}(\R)$, it is sufficient to discuss  $\RR^\delta $ as a map from $L^{2,1}(\R)$ to $L^2(\R)$. To this aim, we denote by $I,\,II,\,III,\,IV$ the four addends of $\RR^\delta(\widehat\zeta)$ given by \eqref{eq:Rdelta} and estimate them separately. We also consider at each step $\widehat \zeta \in B(M \delta^{-1})$.

First, since $W \in L^\infty (\R)$, by \cite[Lemma 6.12]{FLW17} we have 
\begin{equation}\label{eq:I}
\begin{split}
    \lVert I \rVert_{L^2(\R)} \lesssim \delta^{-\frac 12} \lVert W \eta^{\mathrm{far}} \rVert_{L^2(\R)} \lesssim \delta^{-\frac 12} \lVert \eta^{\mathrm{far}} \rVert_{L^2(\R)}\,.
\end{split}
\end{equation}
To estimate the last norm above, we exploit the fact that $\eta^{\mathrm{far}}$ solves the fixed point equation \eqref{fp_E} in the ball $B(\rho_\delta)\subset H_s^{2,\mathrm{far}}(\R)$, with $\rho_\delta$ defined as in Proposition \ref{prop:lipbound} and $\gamma$ as before, since this entails 
\begin{equation}
\label{eq:fars}
\Vert \eta^{\mathrm{far}}\Vert^2_{H^2(\R)} \lesssim \delta^{1-2\tau}(2+M^2) \,,
\end{equation}
(again for $\delta>0$ small depending on $M$). In view of \eqref{eq:I} this yields
\begin{equation}\label{eq:Ies}
\Vert I\Vert_{L^2(\R)}\leq C_1\delta^{-\tau}(2+M^2)
\end{equation}
for some constant $C_1>0$. Note that for $II$ one has a completely analogous estimate since $U_0(\cdot,\delta\cdot),\,U_1(\cdot,\delta\cdot)\in L^\infty(\R)$ (see \eqref{eq:U0}, Theorem \ref{thm:main2} and the proof of Lemma \ref{lem:estU}).

To estimate $III$, we combine \eqref{est_F} with \cite[Lemma 6.12]{FLW17} to obtain
\begin{equation}\label{eq:IIes}
\lVert III \rVert_{L^2(\R)} \lesssim \delta^{-\frac 12}\lVert F \rVert_{L^2(\R)} \leq C_2 \, \delta^{-1}\,,
\end{equation}
for some other constant $C_2>0$.

To conclude, again by \cite[Lemma 6.9 \& Lemma 6.12]{FLW17}, together with \eqref{est_N} and Corollary \ref{cor:solfar}, we have 
\begin{equation}\label{eq:III}
    \lVert IV \rVert^2_{L^2(\R)} \lesssim \delta^{- 1} \lVert \mathcal N_\delta(\eta) \rVert^2_{L^2(\R)} \lesssim \delta\Vert \eta^{\mathrm{far}}+\eta^{\mathrm{near}}\Vert^4_{H^2(\R)}+\delta^3\Vert \eta^{\mathrm{far}}+\eta^{\mathrm{near}}\Vert^6_{H^2(\R)}\,.
\end{equation}
Thus, using \eqref{eq:nears}, \eqref{eq:fars} and the triangle inequality we find
\begin{equation}\label{eq:IIIes}
\Vert IV\Vert_{L^2(\R)}\leq C_3(M) \delta^{-1/2}
\end{equation}
where the constant $C_3=C_3(M) >0$ depends on $M$. This holds again provided $\delta>0$ small with respect to $M$.

Summing up,
\begin{equation}\label{eq:Rest}
\Vert\mathcal R^\delta(\widehat\zeta)\Vert_{L^2(\R)} \leq C_1\delta^{-\tau}(2+M^2)+ C_2 \delta^{-1}+ C_3(M)\delta^{-1/2}
\end{equation}
and, fixing
\[
M:=1+C_2\|(\mathbb{I}+C^{\delta^\dagger})^{-1}\|_{L^{2,1}(\R)\to L^{2,1}(\R)}\|\widehat\D_0^{-1}\|_{L^{2}(\R)\to L^{2,1}(\R)}
\]
with $\delta^\dagger$ given by Proposition \ref{prop:Cbound}, since $0<\tau <1/2$ there holds 
\[
\lVert A^\delta(\widehat \zeta) \rVert_{L^2} \leq M \delta^{-1}\,,
\]
for $\delta>0$ small enough, which proves \eqref{eq:ball}.

Moreover, by similar computations and using \eqref{eq:Efarlip}, one obtains
\[
\begin{split}
\lVert A^\delta(\widehat \zeta_1 ) - A^\delta(\widehat \zeta_2) \rVert_{L^2(\R)} &\lesssim \delta^{-\frac12} \left( \lVert \eta^{\mathrm{far}}_1 - \eta^{\mathrm{far}}_2 \rVert_{L^2(\R)} + \lVert \mathcal N_\delta(\eta_1) - \mathcal N_\delta(\eta_2) \rVert_{L^2(\R)} \right)\\
& \lesssim \delta^{-\frac 12 } \delta^{1-\tau}  \lVert \eta^{\mathrm{near}}_1 - \eta^{\mathrm{near}}_2 \rVert_{L^2(\R)} \lesssim \delta^{\frac12-\tau} \lVert \widehat\zeta_1 - \widehat\zeta_2 \rVert_{L^2(\R)}\,,
\end{split}
\]
ensuring the existence of $\delta_0=\delta_0(\tau,M)>0$ such that $A^\delta$ is a contraction on $B(M \delta^{-1})$, for $0<\delta<\delta_0$. Thus, there exists a unique solution to \eqref{eq:fpzeta} in $L^{2,1}(\R)$ satisfying the bound \eqref{eq:boundZeta}.
\end{proof}

We are now in a position to conclude the proof of Theorem \ref{thm:main1}, assuming $\vartheta_\sharp>0$. In Remark \ref{rmk:proof_theta_neg} we describe the changes needed to handle the case $\vartheta_\sharp<0$.

\begin{proof}[End of the proof of Theorem \ref{thm:main1} when $\vartheta_\sharp>0$]
By Proposition \ref{lem:F} and Remark \ref{rem:local}, for $\delta>0$ sufficiently close to $0$ the system \eqref{eq:diracnear} admits a solution $\widehat\zeta^\delta\in \widehat Y\subseteq L^{2,1}(\R)$ with 
\begin{equation}
\label{eq:stimafin}
\Vert \widehat\zeta^\delta\Vert_{L^{2,1}(\R)}\lesssim\delta^{-1}
\end{equation}
supported on the set $\{\vert\xi\vert\leq \delta^{\tau-1} \}$, namely
    \[
    \widehat\zeta^\delta(\xi)=\chi(\vert\xi\vert\leq \delta^{\tau-1})\widehat\zeta^\delta(\xi).
    \]
Let then $\eta_\delta^{\mathrm{near}}$ be the function associated to $\zeta^\delta$ by\eqref{eq:bl}, \eqref{eq:zeta} and \eqref{eq-dinuovonear}, and $\eta_\delta^{\mathrm{far}}=\eta^\mathrm{far}(\eta^\mathrm{near}_\delta,\delta)$ the functions given by Corollary \ref{cor:solfar}. Setting $\eta_\delta:=\eta_\delta^{\mathrm{near}}+\eta_\delta^{\mathrm{far}}$, $u_\delta$ as in \eqref{eq:ansatz} and $\mu_\delta=\mu_*+\delta\mu_\sharp$, then $u_\delta\in H^2(\R)$ solves \eqref{NLS} and it is real--valued and even by Remark \ref{rem:realU} and Proposition \ref{prop:real}. Moreover,  by \eqref{eq:stimafin}, \cite[Lemma 6.9]{FLW17} and \eqref{eq:farbound} we obtain $\Vert \eta_\delta\Vert_{H^2(\R)}\lesssim\delta^{-1/2}$, which together with Lemma \ref{lem:estU} yields
\[
	 \begin{split}
	\Vert u_\delta(\cdot) -\sqrt{\delta} U_0(\cdot,\delta\cdot)\Vert_{H^2(\R)}&\,=
	\delta^{3/2} \Vert U_1(\cdot,\delta\cdot)+ \eta_\delta(\cdot)\Vert_{H^2(\R)}\\
	&\,\leq \delta^{3/2} \left(\Vert U_1(\cdot,\delta \cdot)\Vert_{H^2(\R)}+ \Vert\eta_\delta\Vert_{H^2(\R)}\right)\lesssim \delta\,,
	 \end{split}
	 \]
proving \eqref{eq:uapp} and completing the proof of Theorem \ref{thm:main1}.
\end{proof}
\begin{remark}\label{rmk:proof_theta_neg}
    As we anticipated at the beginning of Section \ref{sec:proofmain2}, the proof of Theorem \ref{thm:main1} in the case $\vartheta_\sharp<0$ is almost identical to the one developed so far for positive values of $\vartheta_\sharp$. The major difference lies in the fact that when $\vartheta_\sharp<0$ one looks from the very beginning for a solution $\eta_\delta$ which is an odd function. To do this, it is enough to replace the definition of $L_s^{2,\mathrm{near}}(\R)$, $L_s^{2,\mathrm{far}}(\R)$ given in \eqref{eq:nfspaces} with the following ones
    \[
    \begin{split}
   L^{2,\mathrm{near}}_s(\R)& \coloneqq \{ f \in L^2(\R) \colon \widetilde f_n(k) \equiv \widetilde f_{n}^{\mathrm{near}}(k)\;\; \forall n\in I,\, \widetilde f_\pm(2\pi-k)=\overline{\widetilde f_\mp}(k)=-\widetilde f_\mp(k)\}\\
   L^{2,\mathrm{far}}_s(\R) &\coloneqq \{ f \in L^2(\R) \colon \widetilde f_n(k) \equiv \widetilde f_{n}^{\mathrm{far}}(k)\;\;\forall n\in I, \, \widetilde f_\pm(2\pi-k)=\overline{\widetilde f_\mp}(k)=-\widetilde f_\mp(k)\,,\\
   &\qquad\qquad\qquad\qquad\qquad\qquad\qquad\qquad\qquad \widetilde f_n(2\pi-k)=\overline{\widetilde f_n}(k)=-\widetilde f_n(k)\;\;\forall n\in I_1\}\,.
  \end{split}
    \]
With these new definitions, the proof of Proposition \ref{prop:real} adapts to show that functions in $L^{2,\mathrm{near}}_s(\R)$ and $L^{2,\mathrm{far}}_s(\R)$ are now real and odd, and it is again readily seen that the system \eqref{eq_near/far} is invariant under the symmetries embodied in these new spaces. The fixed point argument performed to construct $\eta^{\mathrm{far}}$ works exactly as above, whereas for $\eta^{\mathrm{near}}$ this change reflects into the following new symmetries for $\widehat\zeta(\xi)=(\widehat\eta_-, \widehat\eta_+)^T$
\[
\begin{cases}
\overline{\widehat{\zeta}(\xi)}=\sigma_1 \widehat{\zeta}(-\xi)\\
\widehat{\zeta}(\xi)=-\sigma_1 \widehat{\zeta}(-\xi)\,.
\end{cases}
\]
in place of \eqref{eq:xisym}, that in physical space corresponds to substituting condition \eqref{eq:xsym2} with $\zeta(x)=-\sigma_1\zeta(-x)$. This affects only the proof of the invertibility of the operator $\D_0$, as we now need to show that the kernel of this operator is trivial when restricted to the space $Y=\{\zeta\in H^1(\R)\,:\,\overline\zeta=\sigma_1\zeta,\,\zeta(\cdot)=-\sigma_1\zeta(-\cdot)\}$. However, this can be done arguing exactly as in the proof of Proposition \ref{prop:kernelD}, since when $\vartheta_\sharp<0$ the function $\Psi'$ is such that $\Psi'(x)=\sigma_1\Psi'(-x)$ by \eqref{eq:psisimm}. The rest of the argument can then be repeated verbatim, completing the proof of Theorem \ref{thm:main1} also in the case of negative $\vartheta_\sharp$.  
\end{remark}


\section*{Statements and Declarations}


\noindent\textbf{Conflict of interest.}  The authors declare that they have  no conflict of interest.
\medskip

\noindent\textbf{Data availability statement.} Data sharing is not applicable to this article as it has no associated data. 
\medskip

\noindent\textbf{AI declaration.} The authors acknowledge the use of AI tools for the generation of Figures \ref{fig:gap-opening}--\ref{fig:dirac-profiles}--\ref{fig:dirac-soliton-scales}.

\bigskip

\noindent\textbf{Acknowledgements.} The authors acknowledge that this study was carried out within the project E53D23005450006 ``Nonlinear dispersive equations in presence of singularities'' -- funded by European Union -- Next Generation EU within the PRIN 2022 program (D.D. 104 - 02/02/2022 Ministero dell'Universit\`a e della Ricerca). This manuscript reflects only the author's views and opinions and the Ministry cannot be considered responsible for them. The present research has been supported by MUR grant ``Dipartimento di Eccellenza'' 2023-2027 of Dipartimento di Matematica, Politecnico di Milano. The authors are members of {\em Gruppo Nazionale per l'Analisi Matematica, la Probabilit\`a e le loro Applicazioni} (GNAMPA) of the {\em Istituto Nazionale di Alta Matematica} (INdAM).


\end{document}